\documentclass{amsart}
\usepackage{amssymb,amsbsy,amsthm,amsmath,graphicx,epsfig}

\usepackage{setspace}
\usepackage{version}

\numberwithin{equation}{section}
\theoremstyle{plain}
\newtheorem{theorem}[subsection]{Theorem}
\newtheorem{proposition}[subsection]{Proposition}
\newtheorem{lemma}[subsection]{Lemma}
\newtheorem{corollary}[subsection]{Corollary}

\theoremstyle{definition}
\newtheorem{definition}[subsection]{Definition}

\theoremstyle{remark}
\newtheorem{remark}[subsection]{Remark}

\newtheorem{example}{Example}
\renewcommand{\leq}{\leqslant}
\renewcommand{\geq}{\geqslant}
\newsavebox{\proofbox}
\savebox{\proofbox}{\begin{picture}(7,7)  \put(0,0){\framebox(7,7){}}\end{picture}}

\newcommand{\md}[1]{\ensuremath{(\operatorname{mod}\, #1)}}

\newcommand\Z{\mathbb{Z}}
\newcommand\R{\mathbb{R}}
\newcommand\T{\mathbb{T}}

\newcommand\N{\mathbb{N}}
\newcommand\A{\mathbf{A}}
\newcommand\G{\mathbf{G}}

\newcommand\GL{\operatorname{GL}}

\newcommand\Spec{\operatorname{Spec}}

\newcommand\F{\mathbb{F}}
\newcommand\Q{\mathbb{Q}}

\newcommand\n{\mathfrak{n}}

\newcommand\eps{\varepsilon}
\newcommand\id{{\operatorname{id}}}
\newcommand\st{{\operatorname{st}}}
\newcommand\diam{{\operatorname{diam}}}

\newcommand{\ultra}{{{}^*}}
\onehalfspace
\begin{document}

\title{The structure of approximate groups}

\author{Emmanuel Breuillard}
\address{Laboratoire de Math\'ematiques\\
B\^atiment 425, Universit\'e Paris Sud 11\\
91405 Orsay\\
FRANCE}
\email{emmanuel.breuillard@math.u-psud.fr}

\author{Ben Green}
\address{Centre for Mathematical Sciences\\
Wilberforce Road\\
Cambridge CB3 0WA\\
England }
\email{b.j.green@dpmms.cam.ac.uk}

\author{Terence Tao}
\address{Department of Mathematics, UCLA\\
405 Hilgard Ave\\
Los Angeles CA 90095\\
USA}
\email{tao@math.ucla.edu}


\begin{abstract}  Let $K \geq 1$ be a parameter. A $K$-approximate group is a finite set $A$ in a (local) group which contains the identity, is symmetric, and such that $A \cdot A$ is covered by $K$ left translates of $A$.

The main result of this paper is a qualitative description of approximate groups as being essentially finite-by-nilpotent, answering a conjecture of H. Helfgott and E. Lindenstrauss. This may be viewed as a generalisation of the Freiman-Ruzsa theorem on sets of small doubling in the integers to arbitrary groups.

We begin by establishing a correspondence principle between approximate groups and locally compact (local) groups that allows us to recover many results recently established in a fundamental paper of Hrushovski. In particular we establish that approximate groups can be approximately modeled by Lie groups.

To prove our main theorem we apply some additional arguments essentially due to Gleason.  These arose in the solution of Hilbert's fifth problem in the 1950s.

Applications of our main theorem include a finitary refinement of Gromov's theorem, as well as a generalized Margulis lemma conjectured by Gromov and a result on the virtual nilpotence of the fundamental group of Ricci almost nonnegatively curved manifolds.
\end{abstract}

\maketitle

\renewcommand{\abstractname}{R\'esum\'e}

\begin{abstract} Soit $K \geq 1$ un param\`etre. On appelle groupe $K$-approximatif toute partie finie $A$ d'un groupe (ou d'un groupe local) qui est sym\'etrique, contient l'identit\'e et est telle que $A \cdot A$ peut \^etre recouvert par au plus $K$ translat\'es \`a gauche de $A$.

Le r\'esultat principal de cet article montre que tout groupe approximatif est, grossi\`erement, fini-par-nilpotent, ce qui r\'epond par l'affirmative \`a une conjecture de H.Helfgott et de E.Lindenstrauss. On peut interpr\'eter ce th\'eor\`eme comme une g\'en\'eralisation \`a un groupe quelconque du th\'eor\`eme de Freiman-Ruzsa sur la structure des parties finies \`a petit doublement du groupe additif des entiers relatifs.

Nous commen\c{c}ons par \'etablir un principe de correspondence entre les groupes approximatifs et les groupes (ou groupes locaux) localement compacts, puis nous en d\'eduisons de nombreuses cons\'equences issues d'un important article r\'ecent de Hrushovski. En particulier, nous montrons que tout groupe approximatif peut-\^etre repr\'esent\'e par un groupe de Lie.

Pour d\'emontrer notre th\'eor\`eme principal, nous appliquons des arguments, en substance dus \`a Gleason  pour la plupart, qui ont vus le jour dans le contexte de la solution du cinqui\`eme probl\`eme de Hilbert dans les ann\'ees 50.

En guise d'application du th\'eor\`eme principal, nous montrons une version fine du th\'eor\`eme de Gromov, ainsi qu'un lemme de Margulis g\'en\'eralis\'e conjectur\'e par Gromov, et un r\'esultat sur la presque nilpotence des groupes fondamentaux des vari\'et\'es \`a courbure de Ricci presque positive.
\end{abstract}

\maketitle

\setcounter{tocdepth}{1}	

\tableofcontents

\section{Introduction}

\textsc{Approximate groups.} A fair proportion of the subject of additive
combinatorics is concerned with approximate analogues of exact algebraic
properties, and the extent to which they resemble those algebraic
properties. In this paper we are concerned with sets that are approximately closed under multiplication, which we do not necessarily assume to be commutative, and more specifically with \emph{approximate groups}.  These are finite non-empty sets $A$ with group-like properties which we shall state precisely later. First we will motivate the definition of an approximate group with some discussion and examples.

Suppose first of all that $A$ is a finite subset of some ambient group $G = (G,\cdot)$. This is the setting considered in essentially all of the existing literature, and the one of importance in applications.  However, as we shall see later, our method of proof is in fact more naturally adapted to a more general setting, in which $A$ lies in a \emph{local} group rather than a global one.

It is easy to see that a finite non-empty subset $A$ of $G$ is a genuine
subgroup if, and only if, we have $xy^{-1} \in A$ whenever $x,y \in A$.
Perhaps the most natural way in which a set $A$ may be \emph{approximately}
a subgroup, then, is if the set $A \cdot A^{-1} := \{xy^{-1} : x,y \in A\}$ has
cardinality not much bigger than the cardinality of $A$: for example, we might ask that $|A\cdot A^{-1}| \leq K|A|$ for
some constant $K$.

Sets with this property or with the closely related property $|A^2| \leq K|A|$, where $A^2 := A \cdot A = \{ xy: x,y \in A \}$, are said to have \emph{small doubling}, and this is
indeed a commonly encountered condition in various fields of mathematics, in particular in additive combinatorics. It is a
perfectly workable notion of approximate group in the abelian setting and
the celebrated Freiman-Ruzsa theorem, Theorem \ref{frei-ruz} below, describes subsets of $\Z$
with this property. However in \cite{tao-noncommutative} it was noted that in noncommutative
settings a somewhat different, though closely related, notion of approximate
group is more natural: $A$ is an approximate group if it is symmetric in the sense that the identity $\id$ lies in $A$, if $a^{-1} \in A$ whenever $a \in A$, and if $A \cdot A$ is covered by $K$ left-translates of $A$.

As suggested above we consider in this paper a slightly more general (and perhaps more natural, in retrospect) ``local'' definition of approximate group in which there is no ambient global group $G$.   It will be convenient to introduce the following definition. This requires the concept of a \emph{local group}, which is discussed at some length in Appendix \ref{local-sec}.

\begin{definition}[Multiplicative set]\label{hundred-def}  A \emph{multiplicative set} is a finite non-empty set $A$ contained in a (symmetric) local group $G = (G,\cdot)$, such that the product set $(A \cup A^{-1})^{200}$ is well-defined, where $A^{-1} := \{ a^{-1}: a \in A \}$ is the inverse of $A$. Strictly speaking, one should refer to the pair $(A,G)$ as the multiplicative set rather than just $A$, but we will usually abuse notation and omit the ambient local group $G$.

In some (abelian) examples, we will use additive group notation $G = (G,+)$ rather than multiplicative notation $G = (G,\cdot)$.  In such cases, we will refer to multiplicative sets as \emph{additive sets} instead.
\end{definition}

Clearly, any finite non-empty subset of a (global) group $G$ is a multiplicative set. The reader should probably keep this model case in mind throughout a first reading of this paper. Indeed the additional generality afforded by the local setting is only needed at a single, albeit critical, place in the argument in Section \ref{endgame}.  One should informally think of a multiplicative set $A$ as a set that behaves ``as if'' it were in a global group, so long as one only works ``locally'' in the sense that one only considers products of up to $200$ elements of $A$ and their inverses.
The exponent $200$ in Definition \ref{hundred-def} is somewhat arbitrary, but for the purposes of studying approximate groups, the exact choice of this exponent is not important in practice, so long as it is at least $8$ (see Theorem \ref{scs} for a precise formalisation of this assertion). For the reader familiar with Freiman homomorphisms (cf. \cite[\S 5.3]{tv-book}), we remark that these are essentially the morphisms in the category of multiplicative sets.

\begin{definition}[Approximate groups]\label{approx-group-def}
Let $K \geq 1$. A \emph{$K$-approximate group} is a multiplicative set $A$ with the following properties:
\begin{enumerate}
\item the set $A$ is symmetric in the sense that $\id \in A$ and $a^{-1} \in A$ if $a \in A$;
\item there is a symmetric subset $X \subset A^3$ with $|X| \leq K$
such that $A \cdot A \subseteq X \cdot A$.
\end{enumerate}
We will sometimes refer to actual (global) groups as \emph{genuine} groups, in order to distinguish them from approximate groups.  We define a \emph{global $K$-approximate group} to be a $K$-approximate group $A$ that lies inside a global group $G$.  We refer to $K$ as the \emph{covering parameter} of the approximate group $A$.
\end{definition}

\begin{remark} We will also have occasion to deal with infinite $K$-approximate groups, which are defined exactly as ordinary $K$-approximate groups, except that they are no longer required to be finite sets. A convex body in a Euclidean space, or a small ball in a Lie group, are examples of infinite approximate groups. Later we will introduce the important notion of an \emph{ultra approximate group}, which is another example. However, by default, approximate groups in this paper will be understood to be finite unless otherwise stated.\end{remark}

The connection between sets with small doubling and the apparently stronger property of being an approximate group was worked out in \cite{tao-noncommutative}, building on work of Ruzsa \cite{ruzsa-finite-field}; see Remark \ref{dubex} below.

When we speak of an ``approximate group'' we shall generally imagine that $K$ is fixed (e.g. $K = 10$) and that $|A|$ is large. Let us give some examples.\vspace{11pt}

\begin{example}[Finite group]\label{fing} A $1$-approximate group is the same thing as a finite group.
\end{example}

\begin{example}[Arithmetic/geometric progression]\label{ap} If $N \in \N$ is a natural number, then the arithmetic progression $P(1;N) := \{-N,\dots,N\}$ (which one can view inside the (additive) global group $\Z$, or the local group $\{-200N,\ldots,200N\}$) is a $2$-approximate group. More generally, if $G = (G,\cdot)$ is any (global) group and $g \in G$ then the geometric progression $P(g,N) := \{g^{-N},\dots, g^N\}$ is a $2$-approximate group.
\end{example}

\begin{example}[Generalised arithmetic progression]\label{gap}  Let $G = (G,+)$ be an abelian group, let $u_1,\ldots,u_r \in G$ for some $r \geq 0$, and let $N_1,\dots,N_r > 0$ be real numbers.  We refer to the set
\begin{equation*}
P(u_1,\ldots,u_r; N_1,\ldots,N_r) := \{n_1 u_1 + \dots + n_r u_r : n_1,\ldots,n_r \in \Z; |n_1| \leq N_1, \dots , |n_r| \leq N_r\}
\end{equation*}
as a \emph{generalised arithmetic progression} of rank $r$.  One easily verifies that this is a $2^r$-approximate group.
\end{example}

\begin{example}[Homomorphic images]\label{hom}  Let $\phi: G \to H$ be a homomorphism between local or global groups.  If $A$ is a $K$-approximate subgroup of $G$, then $\phi(A)$ is a $K$-approximate subgroup of $H$.
This observation can be generalised to the case when $\phi$ is a \emph{Freiman homomorphism} (of order $3$) rather than a group homomorphism; see \cite[\S 5.3]{tv-book} for more discussion. Indeed, Freiman homomorphisms are very similar to homomorphisms of local groups, although for technical reasons we will rely on the latter concept rather than the former.

Conversely, if $B$ is a $K$-approximate subgroup of $H$, $\phi$ is surjective, and $\ker(\phi)$ is finite, then $\phi^{-1}(B)$ is a $K$-approximate subgroup of $G$.  In the latter case one can view the $K$-approximate group $\phi^{-1}(B)$ as a ``finite extension'' of the $K$-approximate group $B$ by the genuine group $\ker(\phi)$.
\end{example}

\begin{example}[Large subsets]  Let $A$ be a $K$-approximate group, and let $A'$ be a symmetric neighbourhood of the identity in $A$ such that $A$ is covered by $K'$ left-translates of $A'$.  Then $A'$ is a $KK'$-approximate group.  This hints that approximate groups are considerably more numerous than genuine groups, because the latter property is preserved under passage to ``large'' subsets, whereas the former is not.
\end{example}

\begin{example}[Heisenberg example]\label{heisen-ex} Let $G$ be the free nilpotent group of step $2$ generated by two generators $u_1,u_2$.  More concretely, one can take $G$ to be the \emph{Heisenberg group}
\begin{equation}\label{heisendef}
G := \left(
\begin{smallmatrix}
1 & \Z & \Z \\
0 & 1 & \Z \\
0 & 0 & 1
\end{smallmatrix}
\right)
\end{equation}
with generators
\begin{equation*}
u_1 := \left(
\begin{smallmatrix}
1 & 1 & 0 \\
0 & 1 & 0 \\
0 & 0 & 1
\end{smallmatrix}
\right) \quad \mbox{and} \quad u_2 := \left(
\begin{smallmatrix}
1 & 0 & 0 \\
0 & 1 & 1 \\
0 & 0 & 1
\end{smallmatrix}
\right).
\end{equation*}
Consider also the commutator
\begin{equation*}
[u_1,u_2] := u_1^{-1}u_2^{-1}u_1 u_2 = \left(
\begin{smallmatrix}
1 & 0 & 1 \\
0 & 1 & 0 \\
0 & 0 & 1
\end{smallmatrix}
\right);
\end{equation*}
one has
$$ \left(
\begin{smallmatrix}
1 & n_1 & n_{12} \\
0 & 1 & n_2 \\
0 & 0 & 1
\end{smallmatrix}
\right) = u_1^{n_1} u_2^{n_2} [u_1,u_2]^{n_{12}}$$
for all integers $n_1,n_2,n_{12}$.

Let $N_1,N_2 \geq 10$ be real numbers.  Define the \emph{nilprogression} $P(u_1,u_2;N_1,N_2)$ to be the set of all words in $u_1, u_1^{-1}, u_2, u_2^{-1}$ that involve at most $N_1$ occurrences of $u_1, u_1^{-1}$ and at most $N_2$ occurrences of $u_2, u_2^{-1}$.  It is not difficult to verify that $P(u_1,u_2;N_1,N_2)$ is a symmetric neighbourhood of the identity which contains the set
$$
\{u_1^{n_1}u_2^{n_2}[u_1,u_2]^{n_{12}} : |n_1| \leq N_1/10, |n_2| \leq N_2/10,
|n_{12}| \leq N_{1}N_{2}/10\}
$$
and is contained in the set
$$
\{u_1^{n_1}u_2^{n_2}[u_1,u_2]^{n_{12}} : |n_1| \leq 10 N_1, |n_2| \leq 10 N_2,
|n_{12}| \leq 10 N_{1}N_{2}\}.
$$
One can easily verify that $P(u_1,u_2;N_1,N_2)$ is a $K$-approximate group  for some absolute constant $K$ (for instance, one could take $K=100$).
\end{example}

\begin{remark} The above example was constructed inside the Heisenberg group.  Later on we will discuss a generalisation of this example to arbitrary nilpotent groups.  These examples, which we will call \emph{nilprogressions}, will be needed to state the precise version of our main theorem (Theorem \ref{main-theorem}) below. We will define them later in this introduction.
\end{remark}

\begin{example}[Direct products] The direct product of a $K_1$-approximate group and a $K_2$-approximate group is a $K_1 K_2$-approximate group, and so one may build up examples of approximate groups using both subgroups and nilprogressions. \end{example}

\begin{example} [Helfgott's example]\label{helfex} The following example of Helfgott\footnote{See {\tt terrytao.wordpress.com/2009/06/21/freimans-theorem-for-solvable-groups/\#comment-39705}.} is a less obvious way of combining a subgroup and a nilprogression.

 Let $A \subseteq \GL_3(\F_p)$ be the following set of $3 \times 3$ matrices:
\[ A := \left\{ \left(\begin{smallmatrix} r^n & x & z \\ 0 & s^n & y \\ 0 & 0 & (rs)^{-n}\end{smallmatrix}\right) : x,y,z \in \F_p, -N \leq n \leq N\right\}.\] Here, $r,s \in \F_p^{\times}$ are fixed and $N$ is large yet much smaller than $p$. Then $A$ is a $O(1)$-approximate group.

Note that $A$ has the following form: it admits a subgroup $H$, normalised by $A$, such that $A/H$ is a a geometric progression. Indeed
\[ H = \left\{ \left(\begin{smallmatrix} 1 & x & z \\ 0 & 1 & y \\ 0 & 0 & 1\end{smallmatrix}\right) : x,y,z \in \F_p\right\}.\]
In the language of Example \ref{hom}, $A$ is a finite extension of a geometric progression by the finite group $H$.
\end{example}

Each of the above examples was rather ``algebraic'' in nature, whereas the definition of approximate group is somewhat combinatorial.  We also have some more combinatorial criteria for generating approximate groups using sets of small doubling or tripling.

\begin{remark}[Relationship between small doubling and approximate groups]\label{dubex}  Let $A$ be a non-empty finite subset of a global group $G$.  If $|A^3| \leq K|A|$, then the set $H := (A \cup \{\id\} \cup A^{-1})^2$ is\footnote{Here and in the rest of the paper we use $X = O_K(Y)$, $X \ll_K Y$, or $Y \gg_K X$ for two (standard) quantities $X,Y$ and a (standard) parameter $K$ to denote the assertion that $|X| \leq C_K Y$ for some (standard) quantity $C_K>0$ depending only on $K$, and similarly for other choices of subscripted parameters.  We also adopt an analogous notation for nonstandard quantities; see Appendix \ref{nsa-app}.} a $O(K^{O(1)})$-approximate group that contains $A$; see \cite[Theorem 3.9]{tao-noncommutative}.  In a similar vein if $|A^2| \leq K|A|$ or $|A \cdot A^{-1}| \leq K|A|$, then there exists a $O(K^{O(1)})$-approximate group $H$ of size $|H|=O(K^{O(1)}|A|)$ such that $A$ can be covered by $O(K^{O(1)})$ left-translates $gH$ of $H$; see \cite[Theorem 4.6]{tao-noncommutative}.
\end{remark}

Our aim in this paper is to ``describe'' the structure of approximate subgroups in an arbitrary ambient group in terms of more explicit
algebraic objects such as those listed in the examples. Here is one form of our main result in this regard.

\begin{theorem}[Main theorem, simple form]\label{hl-conj}  Let $A$ be a global $K$-approximate group, thus it is contained in a \textup{(}global\textup{)} group $G$.  Then there exists a subgroup $G_0$ of $G$ and a finite normal subgroup $H$ of $G_0$ with the following properties:
\begin{enumerate}
\item $A$ can be covered by $O_K(1)$ left-translates of $G_0$;
\item $G_0/H$ is nilpotent and finitely generated of rank\footnote{The \emph{rank} of a finitely generated group is the least number of generators required to generate the group. The \emph{step} is the length of the lower central series, minus 1.} and step at most $O_K(1)$;
\item $A^4$ contains $H$ and a generating set of $G_0$.
\end{enumerate}
\end{theorem}

In particular, the group $G_0$ is finite-by-nilpotent, and hence also virtually nilpotent. Indeed, the stabiliser in $G_0$ of the conjugation action on $H$ has finite index in $G_0$ and is a central extension of a finite index subgroup of $G_0/H$, and therefore is also nilpotent.



By specialising Theorem \ref{hl-conj} to the combinatorial examples in Remark \ref{dubex} we obtain an analogous structure theorem for sets of small doubling.

\begin{corollary}[Freiman-type theorem]\label{hl-conj-dub}  Let $A$ and $B$ be finite non-empty subsets in a \textup{(}global\textup{)} group $G$ such that $|AB| \leq K |A|^{\frac{1}{2}}|B|^{\frac{1}{2}}$. Then there exists a subgroup $G_0$ of $G$ and a finite normal subgroup $H$ of $G_0$ with the following properties:
\begin{enumerate}
\item $A$ can be covered by at most $O_K(1)$ right translates of $G_0$;
\item $G_0/H$ is nilpotent and finitely generated of rank and step $O_K(1)$. In particular, $G_0$ is finite-by-nilpotent and hence also virtually nilpotent.
\end{enumerate}
\end{corollary}

\begin{proof} By \cite[Theorem 4.6]{tao-noncommutative}, there exists a $O(K^{O(1)})$-approximate group $A'$ of size $O(K^{O(1)} |A|)$ such that $A$ can be covered by $O(K^{O(1)})$ right translates of $A'$ and $B$ can be covered by $O(K^{O(1)})$ left translates of $A'$. We may thus apply Theorem \ref{hl-conj} to $A'$.
\end{proof}

Theorem \ref{hl-conj} (or Corollary \ref{hl-conj-dub}) answers in the affirmative a conjecture that we have been referring to as the \emph{Helfgott-Lindenstrauss Conjecture}, on account of its having been raised  independently in private communications by both Harald Helfgott and Elon Lindenstrauss.  In fact, the conjecture is reasonably explicit in the comments surrounding \cite[Theorem 1.1]{helfgott-sl3}.

\begin{remark}[The linear case]
Various forms of the main theorem are also known in groups of Lie type of bounded dimension, as a consequence of results of many authors \cite{breuillard-green-unitary, bg, bgt, gill-helfgott, gill-helfgott-small,helfgott-sl2, helfgott-sl3, hrush, pyber-szabo}.  For instance, in \cite{gill-helfgott} an analogue of Theorem \ref{hl-conj} was established in the case when $G$ is a solvable algebraic group of bounded dimension over a finite field of prime order.  In that case, the group $G_0/H$ has bounded rank, and the number of cosets of $G_0$ needed to cover $A$ is \emph{polynomial} in $K$. We have no examples to rule out the possibility that this polynomiality in $K$ holds in all groups $G$, perhaps at the cost of weakening the rank and step bounds on $G_0/H$. Unfortunately our methods, which rely on ultrafilter arguments, give no quantitative bounds on the covering number whatsoever.
\end{remark}

\begin{remark}[Bounds on the nilpotent group]\label{bounds-nilpotent} Our method  allows us to give an explicit bound on the dimension (rank and step) of the nilpotent group $G_0/H$ in Theorem \ref{hl-conj} at the expense of replacing $A^4$ in item (iii) by a larger power of $A$. Namely, if we allow for $H$ and the generating set of $G_0$ to be contained in $A^{12}$, then we may ensure that the nilpotent group $G_0/H$ is $\ell$-nilpotent with $\ell = O(K^2 \log K)$. If we are happy to go as far as $A^{O_K(1)}$, then this may be further reduced to $\ell \leq 6 \log_2 K$. Here we say that a group is \emph{$\ell$-nilpotent} if it admits a generating set $u_1,\ldots,u_\ell$ such that $[u_i,u_j] \in \langle u_{j+1}, \ldots, u_\ell \rangle$ for all $i<j$. In particular such a group admits a normal series with cyclic factors of length at most $\ell$, and so is also nilpotent of step at most $\ell$. We refer the reader to Theorem \ref{dimension-bound} and to Section \ref{rank-reduction} for a detailed statement and proof.
\end{remark}

\begin{remark} Note that no bound is provided on the size of the finite group $H$ in Theorem \ref{hl-conj}, other than that it is finite.  Indeed, by considering $A$ to be a large finite simple group it is not difficult to see that $H$ can be arbitrarily large.
\end{remark}

We will in fact prove a much more precise version of Theorem \ref{hl-conj} involving a slightly complicated type of approximate group which we call a \emph{coset nilprogression}. We discuss this concept in some detail in the next section. For many applications, however, Theorem \ref{hl-conj} is quite sufficient.\vspace{11pt}

\textsc{Applications.} We now give a small selection of applications to growth in groups and to Riemannian geometry; a greater variety is assembled in \S \ref{gromov-sec}, which also contains proofs of these statements. \vspace{11pt}

\emph{Polynomial growth conditions and Gromov's theorem.} Firstly, Theorem \ref{hl-conj} yields a quick proof of Gromov's theorem \cite{gromov} on groups of polynomial growth.

\begin{theorem}[Gromov's theorem]\label{gromov-thm}
Let $G$ be a group of polynomial growth. That is, $G$ is generated by a finite symmetric set $S$, and there are constants $C$ and $d$ such that $|S^n| \leq Cn^d$ for all $n \in \N$. Then $G$ is virtually nilpotent.
\end{theorem}

\begin{remark} In fact, our arguments show that there is some function $f : \N \rightarrow \N$, $f(n) \rightarrow \infty$, such that, if $G$ does not have polynomial growth, then $|S^n| \geq n^{f(n)}$ for all $n$. We do not get an explicit function $f$. However, if the control parameter $O_K(1)$ in Theorem \ref{main-theorem} were known to be polynomial in $K$, we could take $f(n) = c\log n$. The best (in fact only) lower bound known for this function at present is $(\log \log n)^c$, due to Shalom and the third author \cite{shalom-tao}. It is conjectured by some, in the absence of any examples to the contrary, that $f(n)  > n^c$, and possibly even that $|S^n| \geq e^{c\sqrt{n}}$.
\end{remark}

In \cite{hrush} Hrushovski also gave a derivation of Gromov's theorem from his Lie model theorem (see Theorem \ref{lie-model} below). He in fact proved a strengthening of Gromov's theorem (see \cite[Theorem 7.1]{hrush} or Theorem \ref{hrush-grom} below). We will be able to recover Hrushovski's result more directly (see Corollary \ref{main-cor} below). In fact, our approach can also yield the following other strengthening of Gromov's theorem, which is uniform in the size of the generating set $S$ and appears to be new. Recall that if $\ell \in \N$ then we say that a group is $\ell$-nilpotent if it admits a generating set $u_1,\ldots,u_\ell$ such that $[u_i,u_j] \in \langle u_{j+1}, \ldots, u_\ell \rangle$ for all $i<j$.

\begin{theorem}\label{petrunin-again} Let $d>0$. Then there is $n_0=n_0(d)>0$ such that if $G$ is a group generated by a finite symmetric set $S$ with $1 \in S$ for which $|S^n| \leq n^d |S|$ for some $n \geq n_0(d)$, then $G$ is virtually nilpotent. In fact $G$ has a normal subgroup of index at most $O_d(1)$ which is finite-by-\textup{(}$O(d)$-nilpotent\textup{)}.
\end{theorem}

\begin{proof} The proof of this (and hence of Theorem \ref{gromov-thm}) is a short enough deduction that we can give it here in the introduction. We refer the reader to Section \ref{gromov-sec} for more details. Let $N = N(d)$ be a large quantity to be specified later, and let $n_0$ be sufficiently large depending on $N$ and $d$.  By the pigeonhole principle and the hypothesis $|S^n| \leq n^d|S|$ we see that if $n_0$ is sufficiently large depending on $N$ then there exists $n'$,  $N \leq n' \leq n_0/100$, such that $|S^{100n'}| \leq (200)^d |S^{n'}|$. By Corollary \ref{isgroup} (which is quite easy) this implies that $S^{2n'}$ is a $e^{O(d)}$-approximate group.  By our main theorem, Theorem \ref{hl-conj} (and Remark \ref{bounds-nilpotent}), we can thus find a finite-by-($O(d)$-nilpotent) and hence virtually nilpotent group $G_0$ such that $S^{2n'}$ is covered by $O_{d}(1)$ left-translates of $G_0$.  By the pigeonhole principle, if $N$ is large enough, we can find a nonnegative $m < 2n'$ such that $S^{m+1} G_0 = S^m G_0$. Multiplying on the left by $S$ repeatedly we conclude that $S^{m+k} G_0 = S^m G_0$ for all $k \geq 0$.  Since $S$ generates $G$, we conclude that $G = S^m G_0 = S^{2n'} G_0$.  Since $S^{2n'}$ was covered by $O_d(1)$ left-translates of $G$, $G_0$ has index $O_d(1)$ in $G$, and so $G$ is also virtually nilpotent.
\end{proof}

\emph{Riemannian manifolds.} A. Petrunin suggested to us some years ago\footnote{See also {\tt http://mathoverflow.net/questions/11091}.} that a result such as Corollary \ref{petrunin-again} would give a purely group-theoretical proof of a theorem of Fukaya and Yamaguchi \cite{fukaya-yamaguchi} according to which fundamental groups of almost non-negatively curved manifolds are virtually nilpotent. Recall that a closed manifold $M$ is said to be almost non-negatively curved if one can find a sequence $g_n$ of Riemannian metrics on it for which $\diam(M,g_n) \leq 1$ while $K_{M,g_n} \geq -1/n$ where $K_{M,g_n}$ is the sectional curvature. Indeed, a simple application of the Bishop-Gromov inequalities combined with Corollary \ref{petrunin} yields the following improvement assuming only a lower bound on the Ricci curvature and an upper bound on the diameter.

\begin{corollary}[Ricci gap]\label{ricci-again} Given $d \in \N$, there is $\eps(d)>0$ such that the following holds. Let $M = (M,g)$ be an $d$-dimensional compact Riemannian manifold with Ricci curvature bounded below by $-\eps$ and diameter $\diam(M) \leq 1$. Then $\pi(M)$ is virtually nilpotent.
\end{corollary}

This result is known to differential geometers and follows from the works of Cheeger-Colding \cite{cheeger-colding} and Kapovitch and Wilking \cite{kapovitch-wilking}. We refer the reader to Section \ref{hrush-grom} for more discussion and references concerning the above result. We only note that Corollary \ref{petrunin} yields in fact an explicit bound on the nilpotency class, namely that after passing to a subgroup of $\pi_1(M)$ with index $O_d(1)$ and quotienting by a finite normal subgroup, we obtain a $O(d)$-nilpotent group. \vspace{11pt}

\emph{Generalised Margulis lemma.} Another corollary of Theorem \ref{hl-conj} is a ``generalised Margulis lemma'' for metric spaces of a type conjectured by Gromov in \cite[\S 5.F]{gromov-pansu-lafontaine}.
A metric space $X$ is said to have \emph{bounded packing} with packing constant $K$ if there is $K>0$ such that every ball of radius $4$ in $X$ can be covered by at most $K$ balls of radius $1$. Say that a subgroup $\Gamma$ of isometries of $X$ acts \emph{discretely} on $X$ if every orbit is discrete in the sense that $\{\gamma \in \Gamma : \gamma \cdot x \in \Sigma\}$ is finite for every $x \in X$ and for every bounded set $\Sigma \subseteq X$.

\begin{corollary}[Generalized Margulis Lemma]\label{gromov-conj} Let $K \geq 1$ be a parameter. Then there is some $\eps(K) > 0$ such that the following is true. Suppose that $X$ is a metric space with packing constant $K$, and that $\Gamma$ is a subgroup of isometries of $X$ which acts discretely. Then for every $x \in X$ the ``almost stabiliser'' $\Gamma_\eps(x) = \langle S_{\eps}(x)\rangle$, where $S_{\eps}(x) : =\{ \gamma \in \Gamma: d(\gamma \cdot x,x) < \eps \}$, is virtually nilpotent.
\end{corollary}

Note that the space $X$ is not assumed to be a manifold. The traditional Margulis lemma estbalishes a similar statement for subgroups of isometries of pinched negatively curved manifolds, or more generally under a curvature lower bound.\vspace{11pt}

\emph{Approximate groups and polynomial growth.} Finally we remark on an additive-combinatorial application, which asserts that approximate groups have large subsets with ``polynomial growth''.

\begin{theorem}[Approximate groups are locally of polynomial growth]\label{local-poly}
Suppose that $A$ is a global $K$-approximate group. Then $A^4$ contains a $O_K(1)$-approximate group $A'$ with $(A')^4 \subset A^4$ and $|A'| \gg_K |A|$ such that $|(A')^m| \ll_K m^{O_K(1)} |A'|$ for all $m \geq 1$.
\end{theorem}

This theorem is an immediate consequence of Theorem \ref{main-theorem} and Proposition \ref{poly} below.

\begin{remark} The above argument converted nilpotent structure (or more precisely, coset nilprogression structure, see below) to polynomial growth.  In the reverse direction, there is the result of Sanders \cite{sanders-monomial} in certain monomial groups, in which polynomial growth is shown to imply a metric ball type structure, at least under the (rather strong) restriction that the approximate group $A$ is normal in the ambient group $G$.
\end{remark}

\noindent\textsc{Acknowledgments.} EB is supported in part by the ERC starting grant 208091-GADA. He also acknowledges support from MSRI where part of this work was finalized. TT is supported by a grant from the MacArthur Foundation, by NSF grant DMS-0649473, and by the NSF Waterman award.

The first author would like to thank E.~Lindenstrauss from whom he first learned about these questions and for several related discussions. We also acknowledge the huge intellectual debt we owe to prior work of Hrushovski \cite{hrush}, without which we would not have started this project. We are grateful to him for several enlightening discussions regarding his work and the subject matter of the present paper. We also thank I.~Goldbring and L.~van den Dries for showing us a preliminary version of their notes on Hilbert's fifth problem and its local versions, and J.~Lott for help with the references. Finally, all three authors would like to thank T.~Sanders for a number of valuable discussions concerning this work.

\section{Coset nilprogressions and a more detailed version of the Main Theorem}\label{precise-sec}

This section concerns the more precise variants of our main theorem, whose existence we hinted at in the first introductory section. Let us first recall the fundamental inverse sumset theorem for abelian approximate groups. This was first introduced by Freiman \cite{freiman-book}, and a simplified argument was subsequently given in the paper \cite{ruzsa-freiman} of Ruzsa. Here is the theorem in the torsion-free setting. Recall the notion of a generalised arithmetic progression, defined in Example \ref{gap} above.

\begin{theorem}[Freiman-Ruzsa theorem]\label{frei-ruz} Let $G = (G,+)$ be a torsion-free \textup{(}global\textup{)} abelian group, and let $K \geq 2$
be a parameter. Suppose that $A \subseteq G$ is a $K$-approximate
group. Then $4A = A+A+A+A$ contains a generalised arithmetic progression
$$P = P(u_1,\ldots,u_r; N_1,\ldots,N_r)$$
with $r \leq \log^{O(1)} K$ and $|P| \gg e^{-\log^{O(1)} K} |A|$.  In particular $A$ can be covered by $O( e^{\log^{O(1)} K})$ translates of $P$.
\end{theorem}

\begin{proof}  See \cite{sanders-freiman} for the main part of this; the final assertion is then a consequence of Ruzsa's covering lemma, Lemma \ref{loc}.  For earlier results of this type with weaker bounds on $r$ and $P$, see \cite{chang-freiman,ruzsa-freiman}.  In \cite{green-tao-compression} it was noted that one can take $r$ as small as $\lfloor \log_2 K + \eps\rfloor$ for any $\eps > 0$, at the cost of decreasing the size of $|P|$ somewhat; see also \cite{bilu-2,bilu} for prior results along these lines.
\end{proof}

Roughly speaking, Theorem \ref{frei-ruz} asserts that, in a global torsion-free abelian group such as the integers $\Z$, approximate groups are ``controlled'' by generalised arithmetic progressions of bounded rank.  In the case of abelian groups with torsion, the class of generalised arithmetic progressions is not sufficient, as one must also now deal with the example of finite genuine groups (Example \ref{fing}).  It is thus natural to introduce the concept of a \emph{coset progression} $H+P$: the sum of a finite genuine group $H$ and a generalised arithmetic progression $P = P(u_1,\ldots,u_r;N_1,\ldots,N_r)$. This concept is sufficient for the formulation of a Freiman type theorem in an arbitrary abelian group.

\begin{theorem}[Abelian Freiman-Ruzsa theorem]\label{frei-ruz-gen} Let $G = (G,+)$ be a \textup{(}global\textup{)} abelian group, and let $K \geq 2$
be a parameter. Suppose that $A \subseteq G$ is a $K$-approximate
group. Then $4A$ contains a coset progression $H+P$, where
$$P = P(u_1,\ldots,u_r; N_1,\ldots,N_r)$$
is a generalised arithmetic progression with $r \leq \log^{O(1)} K$, $H$ is a finite abelian subgroup disjoint from $P$, and $|H+P| = |H| |P| \gg e^{-\log^{O(1)} K} |A|$.  In particular, $A$ can be covered by $O( e^{\log^{O(1)} K})$ translates of $H+P$.
\end{theorem}

\begin{proof} Again, see \cite{sanders-freiman}; see also \cite{green-ruzsa} for an earlier result in this direction.
\end{proof}

We turn now to the business of dropping the commutativity assumption. We will also drop the assumption that $A$ is contained in a \emph{global} group and merely assume that $A$ is a subset of a \emph{local} group $G$. Informally, this means that we will not require the multiplication  law to be defined everywhere in $G$, but only in a certain neighborhood of $\id$. We refer the reader to Appendix \ref{local-sec} for a precise definition and basic properties; see also \cite[IV.3]{serre} for a discussion of the closely related notion of \emph{group chunk}.  We generalise the concept of a generalised arithmetic progression to this setting as follows.

\begin{definition}[Non-commutative progression] Let $u_1,\ldots,u_r$ be $r$ elements in a local group $G = (G,\cdot)$, and let $N_1,\ldots,N_r$ be $r$ positive \emph{real numbers}. If all products  $g_1\ldots g_n$ are well-defined in $G$, where each $g_i$ is equal to one of $u_j$ or $u_j^{-1}$ and, for each $j=1,\ldots,r$, the formal expression\footnote{For this definition, we consider $u_i$ and $u_j$ to be distinct formal expressions when $i \neq j$, even if $u_i$ and $u_j$ take the same value in $G$, and similarly for $u_i^{-1}, u_j^{-1}$.  Thus, for instance, $P(u_1,u_2;1,1)$ contains $u_1 u_2$ even if $u_1,u_2$ are equal.} $u_j$ and its inverse $u_j^{-1}$ appear at most $N_j$ times, then we call the set of such products a non-commutative progression of rank $r$ and side lengths $N_1,\ldots,N_r$ and we denote it by $P(u_1,\ldots,u_r; N_1,\ldots,N_r)$. We refer to $r$ as the \emph{rank} of the non-commutative progression.
\end{definition}

\begin{remark} One can view non-commutative progressions as multiparameter variants of balls in a word metric.  For instance when all $N_j$ take the same value $N$ and one is working in a global group, the progression $P(u_1,\ldots ,u_r;N,\ldots ,N)$ is comparable with the  word ball $B(N)$ of radius $N$ in the group $\langle u_1,\ldots u_r \rangle$ for the word metric with generating set $\{u_1,\ldots ,u_r\}$ in the sense that $B(N) \subseteq P(u_1,\ldots ,u_r;N,\ldots ,N) \subseteq B(rN)$.\end{remark}

In the global abelian setting, all generalised arithmetic progressions of bounded rank are automatically approximate groups with a bounded covering parameter $K$.  This is not the case in general non-abelian groups, even in the global setting.  For instance, if $F$ is the free non-abelian group on two generators $e_1, e_2$, then the non-commutative progression $P(e_1,e_2; N,N)$ (which, as remarked earlier, is essentially the ball of radius $N$ in $F$) grows exponentially in $N$, and one can easily verify that $P(e_1,e_2; N,N)$ is only a $K$-approximate group for $K$ growing exponentially in $N$.  However, the situation is much closer to the abelian case if the ambient group $G$ is \emph{nilpotent}. Given the link between progressions and balls, the reader familiar with Gromov's theorem on groups of polynomial growth \cite{gromov} (to be discussed later on) will not find this surprising.
  Indeed, it can be shown (though we will not do so here) that if $G$ is a global nilpotent group of step $s$, a non-commutative progression $P(u_1,\ldots,u_r;N_1,\ldots,N_r)$ in $G$ will be a $O_{r,s}(1)$-approximate group if $N_1,\ldots,N_r$ are sufficiently large depending on $r$ and $s$.

This motivates the following definition. Given some generators $u_1,\ldots,u_r$, let us recursively define an \emph{iterated commutator of degree $k$} involving these generators for a natural number $k \geq 1$ by declaring $u_1^{\pm 1},\ldots,u_r^{\pm 1}$ to be the iterated commutators of degree $1$, and $[g,h]$ to be a iterated commutator of degree $j+k$ whenever $g, h$ are iterated commutators of weight $j,k$ respectively for some $j,k \geq 1$.  Thus for instance $[[u_2,u_3^{-1}],[u_2^{-1},u_4]]$ is an iterated commutator of $u_1,u_2,u_3,u_4$ of degree $4$.

\begin{definition}[Nilprogression]\label{nilprogression-def} Suppose that $G$ is a local group and that $s \geq 0$ is an integer. A \emph{nilprogression} of rank $r$ and $s$ is a non-commutative progression $P(u_1,\ldots ,u_r;N_1,\ldots ,N_r)$ with the property that every iterated commutator of degree $s+1$ in the generators $u_1,\ldots ,u_r$ is well-defined and equals the identity $\id$.
\end{definition}

\begin{example} The generalised arithmetic progressions $P(u_1,\ldots,u_r;N_1,\ldots,N_r)$ in Example \ref{gap} is a nilprogression (in additive notation) of rank $r$ and step $1$.  The set $P(u_1,u_2;N_1,N_2)$ in Example \ref{heisen-ex} is a nilprogression of rank $2$ and step $2$.
\end{example}

It can be shown (though we shall not do so here) that if $N_1,\ldots,N_r$ are sufficiently large depending on $r,s$, and $P(u_1,\ldots,u_r;CN_1,\ldots,CN_r)$ is a well-defined nilprogression of step $s$ for some sufficiently large $C$ depending on $r, s$, then $P(u_1,\ldots,u_r; N_1,\ldots,N_r)$ is a $O_{r,s}(1)$-approximate group.

The concept of a nilprogression as defined above is related to, though not quite identical with, the one given in \cite{breuillard-green}. As a byproduct of our proof methods, we will be able to work with a more tractable subclass of nilprogressions, which we will call \emph{nilprogressions in $C$-normal form}. These generalise the notion of a \emph{proper} generalised arithmetic progression in the additive combinatorics literature, and are also close in spirit to the nilprogressions introduced in \cite{tao-solvable}.

\begin{definition}[$C$-normal form]\label{normal-def}
Let $C \geq 1$.  A noncommutative progression \[ P(u_1,\ldots ,u_r;N_1,\ldots ,N_r)\] is said to be in \emph{$C$-normal form} if the following axioms are obeyed.
\begin{enumerate}
\item (Upper-triangular form) For every $i,j$ with $1 \leq i < j \leq r$ and for all four choices of signs $\pm$ one has
\begin{equation}\label{comm}
[u_i^{\pm 1},u_j^{\pm 1}] \in P\left( u_{j+1}, \ldots, u_r; \frac{C N_{j+1}}{N_i N_j}, \ldots, \frac{C N_r}{N_i N_j} \right).
\end{equation}
In particular, $[u_i,u_r] = \id$ whenever $1 \leq i < r$.
\item (Local properness) The expressions $u_1^{n_1} \ldots u_r^{n_r}$ are distinct as $n_1,\ldots,n_r$ range over integers with $|n_i| \leq \frac{1}{C} N_i$, $i = 1,\dots,r$.
\item (Volume bound) One has
\begin{equation}\label{cnn}
 \frac{1}{C} (2\lfloor N_1\rfloor +1) \ldots (2\lfloor N_r\rfloor +1) \leq |P| \leq C (2\lfloor N_1\rfloor +1) \ldots (2\lfloor N_r\rfloor+1).
 \end{equation}
\end{enumerate}
\end{definition}

The somewhat ugly expression $(2\lfloor N_1\rfloor +1) \ldots (2\lfloor N_r\rfloor +1)$ is convenient to have in \eqref{cnn} for some minor technical reasons, but it would not do much harm for the reader to mentally substitute $N_1 \ldots N_r$ for this expression instead if desired.  The volume bound \eqref{cnn} is morally (up to some degradation in the constants $C$) implied by the other axioms of a nilprogression in $C$-normal form, when the $N_1,\ldots,N_r$ are sufficiently large, and one is working in a global group (or at least if one assumes $P(u_1,\ldots ,u_r;DN_1,\ldots ,DN_r)$ to be well-defined for some sufficiently large $D=D_{r,s}$), but for some further minor technical reasons it is convenient to state this bound explicitly in the definition.

\begin{example} The generalised arithmetic progressions $P(u_1,\ldots,u_r; N_1,\ldots,N_r)$ in Example \ref{gap} will be in $1$-normal form if it is proper, i.e. if all the expressions $n_1 u_1 + \ldots + n_r u_r$ for $|n_i| \leq N_i$ are distinct.
\end{example}

\begin{example} The set $P(u_1,u_2;N_1,N_2)$ in Example \ref{heisen-ex} is not in $C$-normal form for any bounded $C$, because $[u_1,u_2]$ is non-trivial.  However, the closely related nilprogression
$$ P( u_1, u_2, [u_1,u_2]; N_1, N_2, N_1 N_2 )$$
of rank $3$ and step $2$ is in $1$-normal form.  The two sets are ``comparable'' in a number of ways; for instance, one can easily verify that
$$ P(u_1,u_2; \frac{1}{C} N_1, \frac{1}{C} N_2) \subset P( u_1, u_2, [u_1,u_2]; N_1, N_2, N_1 N_2 ) \subset P(u_1,u_2; C N_1, C N_2) $$
for some absolute constant $C$ (e.g. one can take $C=100$).
\end{example}

\begin{remark}\label{step} Note that in the global group case, the step of a nilprogression in $C$-normal form is less or equal to its rank.
\end{remark}

In Lemma \ref{cosn} we will show that any non-commutative progression $P(u_1,\ldots,u_r;N_1,\ldots,N_r)$ in $C$-normal form is ``essentially'' a $O_{r,C}(1)$-approximate group.  More precisely, we will show that $P(u_1,\ldots,u_r; \eps N_1,\ldots,\eps N_r)$ is a $O_{r,C,\eps}(1)$-approximate group whenever $\eps>0$ is sufficiently small and the $N_i$'s are sufficiently large depending on $C, r$ . We will also show that every element of $P(u_1,\ldots,u_r; \eps N_1,\ldots,\eps N_r)$ can be rewritten in the form $u_1^{n_1} \ldots u_r^{n_r} h$, where $h \in H$ and $|n_i|=O_{r,s}(\eps N_i)$, while conversely every such product with $|n_i| \leq \eps N_i$ obviously belongs to $P(u_1,\ldots,u_r; \eps N_1,$ $\ldots,\eps N_r)$.

Just as in the abelian case, we need to account for genuine subgroups. The analogue of coset progression is a coset nilprogression, a concept we first define in the simpler setting of global groups.

\begin{definition}[Global coset nilprogression]\label{coset-nilprogression-def-global}
Let $G$ be a (global) group. By a coset nilprogression of rank $r$ and step $s$ in $G$, we mean a set $P$ of the form $\pi^{-1}(Q)$, where $G_0$ is a subgroup of $G$, $H$ is a finite normal subgroup of $G_0$, $\pi: G_0 \to G_0/H$ is the quotient map, and $Q$ is a nilprogression of rank $r$ and step $s$ in $G_0/H$.

We say that $P$ is in \emph{$C$-normal form} if $Q$ is in $C$-normal form.
\end{definition}

We can extend this definition to local groups, using the local notion of quotient group reviewed in Lemma \ref{quotient}.

\begin{definition}[(Local) coset nilprogression]\label{coset-nilprogression-def}
Let $G$ be a (local) group, which we endow with the discrete topology. By a coset nilprogression of rank $r$ and step $s$ in $G$, we mean a set $P$ of the form $\pi^{-1}(Q)$, where $H$ is a finite genuine subgroup of $G$ with a cancellative normalising neighbourhood $G_0$, $W$ is a neighbourhood of $H$ in $G_0$ with $W^6 \subset G_0$, $WH=HW=W$, $\pi: W \to W/H$ is the quotient map defined in Lemma \ref{quotient}, and $Q$ is a nilprogression of rank $r$ and step $s$ in $W/H$.

We say that $P$ is in \emph{$C$-normal form} if $Q$ is in $C$-normal form.

We call $H$ the \emph{finite group} associated with $P$, and $Q$ the \emph{nilprogression} associated with $P$.  If $Q = P(u_1,\ldots,u_r;N_1,\ldots,N_r)$, then we write $P = P_H(u_1,\ldots,u_r;N_1,\ldots,N_r)$.
\end{definition}

\begin{example} A subgroup is a coset nilprogression of rank $0$ and step $0$.  More generally, the direct product of a subgroup with a nilprogression of rank $r$ and step $s$ is a coset nilprogression of rank $r$ and step $s$.  The coset nilprogression will be in $C$-normal form if the associated nilprogression is.
\end{example}

\begin{example} The set $A$ constructed in Example \ref{helfex} is a coset nilprogression of rank $1$ and step $1$, and is also in $1$-normal form as long as $N < \frac{p-1}{2}$.
\end{example}

Again, coset nilprogressions in normal form are essentially approximate groups; see Lemma \ref{cosn} for a precise version of this statement.

We are now ready to state our main technical theorem, which among other things implies Theorem \ref{hl-conj}, and whose proof will occupy the bulk of this paper.

\begin{theorem}[Main theorem] \label{main-theorem}
Let $A$ be a $K$-approximate group. Then $A^4$ contains a coset nilprogression $P$ of rank and step $O_K(1)$ and $|P| \gg_K |A|$.
Furthermore, $P$ can be taken to be in $O_K(1)$-normal form.
\end{theorem}

We remark that precursor results to this theorem in the case of nilpotent or solvable groups were obtained in \cite{breuillard-green,bg,fisher-katz-peng,gill-helfgott, tao-noncommutative,tao-solvable}.  Theorem \ref{main-theorem} also provides an independent proof of a qualitative version of the abelian results of Theorem \ref{frei-ruz} and Theorem \ref{frei-ruz-gen}, which, in contrast to the other known proofs of these results, manages to almost completely avoid the use of Fourier analysis\footnote{However, our argument still uses results relating to Hilbert's fifth problem which require Fourier-analytic tools, such as Pontryagin duality, even in the abelian setting.}.

It is easy to see that Theorem \ref{main-theorem} implies Theorem \ref{hl-conj}, by taking $G_0$ to be the global group generated by $P$. The key point here is that a group generated by a set $u_1,\ldots,u_r$ is nilpotent of step at most $s$ if every iterated commutator of the $u_1,\ldots,u_r$ of degree $s+1$ is trivial. A proof of this assertion may be found in Hall's book \cite{MHall}.\vspace{11pt}

By standard non-commutative product estimates, we can also establish the following Freiman-type theorem for sets of bounded doubling.

\begin{corollary}[Freiman-type theorem] \label{main-theorem-doubling}
Let $K\geq 1$. Let $G$ be a \textup{(}global\textup{)} group and $A, B$ be finite non-empty subsets of $G$ such that $|AB|\leq K|A|^{1/2}|B|^{1/2}$.  Then there exists a coset nilprogression $P$ of rank and step $O_K(1)$ with $|P| \gg_K |A|$ which is in $O_K(1)$-normal form, such that $A$ can be covered by $O_K(1)$ left-translates of $P$, and $B$ can be covered by $O_K(1)$ right-translates of $P$.
\end{corollary}

\begin{proof}  This follows immediately from combining Theorem \ref{main-theorem} with \cite[Theorem 4.6]{tao-noncommutative}.
\end{proof}

In Section \ref{rank-reduction}, we will show the following explicit bounds on the rank and step of $P$.

\begin{theorem}[Bounds on the rank and step of the nilprogression]\label{dimension-bound} In Corollary \ref{main-theorem-doubling} \textup{(}and in Theorem \ref{main-theorem} if  $A$ is assumed to be a \emph{global} $K$-approximate group\textup{)}, at the expense of replacing the conclusion $P \subseteq A^4$ with the weaker statement that $P \subseteq A^{12}$, the coset nilprogression $P$ can be taken to have rank and step at most $O(K^2 \log K)$ while remaining in $O_K(1)$-normal form. Moreover, if we settle for the weaker inclusion $P \subset A^{O_K(1)}$, one can ensure that $P$ has rank and step at most $6\log_2 K$ \textup{(}while still remaining in $O_K(1)$-normal form\textup{)}.
\end{theorem}

It is likely that the numerical constants $6$ and $12$ here can be improved, but we will not pursue such improvements here.\vspace{11pt}

\emph{Local approximate groups can be embedded in global groups.} As we have remarked above, the approximate groups $A$ considered in this paper are \emph{local} in the sense that we do not need to assume that $A$ lies in a global group $G$. However as a consequence of Theorem \ref{main-theorem}, the more detailed version of our main theorem, we have the following statement. It asserts that, at least at the qualitative level, there is in fact no loss of generality in dealing with the global case.

\begin{theorem}\label{local-global}
Suppose that $A$ is a $K$-approximate group. Then $A^4$ contains a $O_K(1)$-approximate group $A'$ with $(A')^4 \subset A^4$ and $|A'| \gg_K |A|$ which is isomorphic to a subset of a global group $G$.
\end{theorem}

This theorem follows from Theorem \ref{main-theorem} and the fact (which we prove in Lemma \ref{normglob}) that a large portion of a coset nilprogression in normal form can be embedded in a global group.  This theorem can be viewed as a discrete analogue to a recent result of Goldbring and van den Dries \cite{dries}, who established that every locally compact local group is locally isomorphic near the identity to a locally compact \emph{global} group (thus there is a neighbourhood of the identity in the former group that is isomorphic to a neighbourhood of the identity in the latter group).  One should also compare this result with Lie's third theorem that every local Lie group is locally isomorphic to a global Lie group (see Theorem \ref{lie-third} and the discussion in Serre's book \cite{serre}).

\section{Ultra approximate groups and Hrushovski's Lie Model Theorem}\label{corresp}

In the next section we will give an outline of the argument we shall use to prove Theorem \ref{main-theorem}. An extremely important component of it will be a \emph{Lie Model Theorem} that implicitly appears in a remarkable paper of Hrushovski \cite[Theorem 4.2]{hrush}, which provided the foundation for much of the work here, and for which we will give a self-contained proof later in this paper.  We can state this theorem very informally as follows:

\begin{theorem}[Hrushovski's Lie Model Theorem, informal version]\label{hrushovski-model}  In a suitable limit, an approximate group is virtually modelled by a precompact neighbourhood of the identity in a Lie group.
\end{theorem}

Of course, to make this theorem more precise, one has to formalise terms such as ``suitable limit'', ``virtually'', and ``modelled''.  We shall do so presently, but first we point out that Theorem \ref{hrushovski-model} is very similar in spirit to a key step \cite[\S 7]{gromov} in Gromov's proof of his  celebrated theorem on groups of polynomial growth, which we state informally as follows.

\begin{theorem}[Gromov's Lie Model Theorem, informal version]\label{gromov-model}  In a suitable limit, a group of polynomial growth can be modeled by a finite-dimensional locally compact space with a transitive isometric action of a Lie group.
\end{theorem}

To deepen the analogy between the two results, we note that Theorem \ref{hrushovski-model} and Theorem \ref{gromov-model} both require the deep body of results surrounding the solution to Hilbert's fifth problem on the topological description of the category of Lie groups (see \cite{montgomery-zippin-book}) in order to bring into view the Lie structure, which is not manifestly present when one first takes a limit.  There are however some technical differences between the precise formulations of Theorem \ref{hrushovski-model} and Theorem \ref{gromov-model}.  In the latter theorem, one has a group $G$ (of polynomial growth) generated by a finite set $S$. This gives a metric on $G$, the word metric given by the generating set $S$. Gromov then looks at the discrete balls $S^n$, $n = 1,2,3\dots$ ``from a distance'' to get some continuous limit metric space $X$. For example if $G = \Z$ and $S = \{-1,0,1\}$, then $S^n = \{-n,\dots, n\}$, and it is heuristically clear that these discrete intervals $S^n$, after rescaling by $n$, ``converge'' in a suitable sense to the continuous interval $[-1,1] \subseteq \R$.

To effect this limit, Gromov introduced what is now known as \emph{Gromov-Hausdorff convergence} of a sequence of metric spaces. In subsequent work of van der Dries  and Wilkie \cite{dries-wilkie} a slightly different approach, using ultralimits (or non-standard analysis) was pioneered. This construction is now known, in the geometric group theory literature, as the \emph{asymptotic cone}.

The asymptotic cone, then, is (a quotient of) an ultraproduct of the sequence of balls $(S^n)_{n \in \N}$.  We will use a similar limit\footnote{In \cite{hrush}, more saturated limits (not necessarily constructed using ultrafilters) were also considered, but we will not need such constructions here.} in order to formalise Theorem \ref{hrushovski-model}, namely an \emph{ultraproduct} $\A$ of an arbitrary sequence $(A_\n)_{\n \in \N}$ of $K$-approximate groups, an object we call an \emph{ultra approximate group}.  We now define this term more precisely.

\begin{definition}[Ultra approximate group]\label{ultra-approx-group}  Throughout this paper, we fix a non-principal ultrafilter $\alpha \in \beta \N\backslash \N$ (see Lemma \ref{ultralem} for a definition of this concept).  If $K > 0$ is a real number then an \emph{ultra $K$-approximate group} is an ultraproduct $\A := \prod_{\n \to \alpha} A_\n$, where each $A_\n$ is a (standard) $K$-approximate group.   Thus, $\A$ is the space of all formal limits $\lim_{\n \to \alpha} a_\n$ with $a_\n \in A_\n$, where two formal limits $\lim_{\n \to \alpha} a_\n$ and $\lim_{\n \to \alpha} a'_\n$ are considered equal if $a_\n = a'_\n$ for all $\n$ sufficiently close to $\alpha$ (i.e. for all $\n$ in an $\alpha$-large subset of $\N$).  See Appendix \ref{nsa-app} for more discussion on ultraproducts. Often we will not need to refer to $K$ explicitly, in which case we speak simply of an \emph{ultra approximate group}.
\end{definition}

Note that we allow the approximate groups $A_\n$ to lie in different ambient groups $G_\n$ (much as the notion of Gromov-Hausdorff convergence also does not require the spaces $X_\n$ involved to all live in a common ambient space).  Ultraproducts are a model-theoretic limit, in contrast to the more geometric notion of a limit defined by Gromov-Hausdorff convergence.  There are two key properties of these model-theoretic limits that make them convenient to use for our purposes.  The first is \emph{{\L}os's theorem}, which roughly speaking asserts that any property that can be stated in the language of first-order logic holds for an ultraproduct $\A = \prod_{\n \to \alpha} A_\n$ if and only if it holds for those $A_\n$ with $\n$ sufficiently close to $\alpha$; see Theorem \ref{los-param}.  The second is \emph{countable saturation}, which we will use to establish the completeness of a certain (pseudo)metric space associated to an ultra approximate group; see Proposition \ref{locally-compact-model}.

Next, we discuss what it would mean to\footnote{Our use of the term ``model'' here is not, strictly speaking, the precise notion that is used in model theory, but is closer to the notion of a ``Freiman model'' from additive combinatorics, as used for instance in \cite{ruzsa-freiman}, \cite{green-ruzsa}.} ``model'' an ultra approximate group $\A$.  Informally, a model would seek to describe the ``coarse-scale'' behaviour of $\A$, and in particular be able to predict when an orbit $\id, a, a^2, a^3, \ldots$ of an element $a$ of $\A$ will ``escape'' $\A$, while ignoring the ``fine-scale'' behaviour of $\A$.  Such a model will be formalised by a homomorphism $\phi: \A^8 \to L$ of local groups that obey certain good properties (see Definition \ref{good-model-def} below).  Before we present this formal definition, though, we first discuss some key examples of ultra approximate groups and their models.

\begin{example}[Nonstandard finite groups]\label{nfg}  Suppose that $A_\n$ is a sequence of (standard) finite groups; then the ultraproduct $\A := \prod_{\n \to \alpha} A_\n$ is an ultra approximate group.  In this case, $\A$ is in fact a genuine group, with group operation given by the law
$$ (\lim_{\n \to \alpha} a_\n) \cdot (\lim_{\n \to \alpha} b_n) := \lim_{\n \to \alpha}(a_n b_n).$$
We will refer to such groups as \emph{nonstandard finite groups}.  A typical example of a nonstandard finite group is the nonstandard cyclic group\footnote{This group is the analogue of the profinite completion $\hat \Z = \lim_{\leftarrow} \Z/n\Z$ of the integers, but is built using the machinery of ultralimits rather than inverse limits.  The two groups are however not  identical.  For instance, $\hat \Z$ is torsion-free, whereas $\Z/N\Z$ can contain torsion; for example if $N$ is even, or equivalently if the set of even natural numbers is $\alpha$-large, then $\Z/N\Z$ contains the element $N/2 \mod N$, which has order $2$. But see Remark \ref{inv} below for a link between ultraproducts and inverse limits.}
$$ \Z/N\Z := \prod_{\n \to \alpha} \Z/\n\Z,$$
where $N \in \ultra \N$ is the nonstandard natural number
\begin{equation}\label{ndef}
N := \lim_{\n \to \alpha} \n.
\end{equation}

In a nonstandard finite group $\A$, there are no elements that ever escape $\A$: if $a \in \A$, then one has $a^n \in \A$ for all $n\in \N$.  As such, it will turn out that $\A$ can be modeled by a trivial homomorphism $\phi: \A \to \{\id\}$ to the trivial group.
\end{example}

\begin{example}[Nonstandard intervals]\label{nsi}  Now consider the sequence $A_\n := P(1;\n) = \{-\n,\ldots,\n\}$ of (standard) arithmetic progressions in $\Z$.  The ultraproduct $\A := \prod_{\n \to \alpha} A_\n$ can be viewed as the nonstandard arithmetic progression $\A = P(1;N) = \{-N,\ldots,N\}$ in the nonstandard integers $\ultra \Z := \prod_{\n \to \alpha} \Z$, where $N$ was defined in \eqref{ndef}.  Then $\A$ is an ultra approximate group, and it can also be viewed as a local group inside the nonstandard integers $\ultra \Z$.

Consider now the map $\pi : \A \rightarrow \R$ defined by
$$\pi(\lim_{\n \to \alpha} a_\n) := \st \lim_{\n \to \alpha} \frac{a_{\n}}{{\n}},$$
where $\st x$ is the standard part of a nonstandard real $x$ (see Appendix \ref{nsa-app}).  Thus, for every standard $\eps > 0$, one has
$$ \pi(\lim_{\n \to \alpha} a_\n) - \eps \leq \frac{a_\n}{\n} \leq \pi(\lim_{\n \to \alpha} a_\n)  + \eps$$
for all $\n$ sufficiently close to $\alpha$.   One may also write
$$ \pi(a) = \st \frac{a}{N}$$
for all $a \in \A$.  The map $\pi$ is a homomorphism of local groups from $\A$ into $[-1,1]$. It is surjective since, for any $\gamma \in [-1,1]$, the nonstandard integer
$$ x := \lfloor \gamma N \rfloor = \lim_{\n \to \infty} \lfloor\gamma \n\rfloor,$$
where $\lfloor \rfloor$ is the integer part function, has image $\pi(x) = \gamma$.  The kernel $\ker(\pi)$ is the set of $x \in \A$ with $x=o(N)$ (thus if $x = \lim_{\n \to \alpha} x_\n$ and $\eps>0$ is standard, then $|x_\n| \leq \eps \n$ an $\alpha$-large set of $\n$).  For instance, every standard integer lies in $\ker(\pi)$, as do some non-standard integers such as $\lfloor \sqrt{N} \rfloor = \lim_{\n \to \infty} \lfloor \sqrt{\n}\rfloor$.

There are similar maps from\footnote{Strictly speaking, as we are currently in an additive setting, one should write $m\A = \A + \ldots + \A$ rather than $\A^m = \A \cdot \ldots \cdot \A$ here.} $\A^m$ to $[-m,m]$ for any fixed natural number $m$, which by abuse of notation we also call $\pi$.   Informally, these maps \emph{model} $\A$ by the interval $[-1,1]$, and more generally model $\A^m$ by $[-m,m]$.  In this particular case, the model $\pi: \A \to [-1,1]$ of the ultraproduct $\A$ can be viewed as a limiting object for models $\pi_\n: A_\n \to [-1,1]$ of the individual factors $A_\n$, by defining $\pi_\n(a) := \frac{a}{\n}$.  However, in more general situations, the model for the ultraproduct is only a limit for \emph{approximate} models of the factors, and this is one reason why we need to work in the ultraproduct setting as much as we do.

The model $\pi: \A^m \to [-m,m]$ is not injective: if $\pi(a)$ is trivial, this does not imply that $a$ is trivial.  However, $\pi$ does have an injectivity-like property which will be important later, which roughly speaking asserts that if $\pi(a)$ is \emph{small}, then $a$ is \emph{small}.  For instance, observe that if $a \in \A^{1000}$ is such that\footnote{This claim is not quite true when $\pi(a)$ is $-1$ or $+1$, as can be seen for instance by considering $a = N+1 = \lim_{\n \to \alpha} \n+1$.} $\pi(a) \in (-1,1)$, then $a \in \A$.    This property on the model $\pi$ can be used to derive some important facts about the ultraproduct $\A$; for instance it implies the \emph{escape property} that if $a, a^2, \ldots, a^{100}$ all lie in $\A^{10}$, then $a$ lies in $\A$.  These sorts of escape properties will play a major role in our arguments in later sections.
\end{example}

\begin{example}[Generalised arithmetic progression]  We still work in the integers $\Z$, but now take $A_\n$ to be the rank two generalised arithmetic progression
$$ A_\n := P(1,\n^{10}; \n,\n) := \{ a+b\n^{10}: a,b \in \{-\n,\ldots,\n\}\}.$$
Then the ultraproduct $\A := \prod_{\n \to \alpha} A_\n$ is the subset of the nonstandard integers $\ultra \Z$ of the form
$$ \A = P(1,N^{10};N,N) = \{ a+bN^{10}: a,b \in \{-N,\ldots,N\}\}.$$
This is an ultra approximate group which can be modeled by the Euclidean plane $\R^2$, using the model maps $\pi: \A^m \to \R^2$ defined for each standard $m$ by the formula
$$ \pi( a + bN^{10} ) := \left( \st \frac{a}{N}, \st \frac{b}{N} \right)$$
whenever $a,b = O(N)$.  The image $\pi(\A^m)$ is then the square $[-m,m]^2$.  As before, if $a \in \A^{1000}$ is such that $\pi(a) \in (-1,1)^2$, then $a \in \A$; this can be used to conclude that if $a,a^2,\ldots,a^{100} \in \A^{10}$, then $a \in \A$.  Note here that while $\A$ lives in a ``one-dimensional'' group $\ultra \Z$, the model $\R^2$ is ``two-dimensional''.  This is also reflected in the volume growth of the powers $A_\n^m$ of $A_\n$ for small $m$ and large $\n$, which grow quadratically rather than linearly in $m$.
\end{example}

\begin{example}[Heisenberg box, I]  This example is related to the Heisenberg example in Example \ref{heisen-ex}. We take each $A_\n$ to be the ``nilbox''
\begin{equation}\label{nilbox}
 A_\n := \left\{ \left(\begin{smallmatrix} 1 & x_{\n} & z_{\n} \\ 0 & 1 & y_{\n} \\ 0 & 0 & 1\end{smallmatrix}\right)
\in \left(\begin{smallmatrix} 1 & \Z & \Z \\ 0 & 1 & \Z \\ 0 & 0 & 1\end{smallmatrix}\right)
: |x_{\n}|, |y_{\n}| \leq {\n}, |z_{\n}| \leq \n^{2}\right\}.
\end{equation}
This is not quite an approximate group because it is not quite symmetric (cf. Example \ref{heisen-ex}), but we will ignore this technicality for sake of exposition. In any case it can be repaired in a number of ways, for instance by replacing $A_\n$ with $A_\n \cup A_\n^{-1}$.  Once again we consider the ultraproduct $\A := \prod_{\n \rightarrow \alpha} A_\n$; this is a subset of the nilpotent (nonstandard) group $\left(\begin{smallmatrix} 1 & \ultra \Z & \ultra \Z \\ 0 & 1 & \ultra \Z \\ 0 & 0 & 1\end{smallmatrix}\right)$, consisting of all elements $\left(\begin{smallmatrix} 1 & x & z \\ 0 & 1 & y \\ 0 & 0 & 1\end{smallmatrix}\right)$ with $|x|, |y| \leq N$ and $|z| \leq N^2$; again, this is a (discrete) local group.

Consider now the map
\[ \pi : \A^8 \rightarrow \left(\begin{smallmatrix} 1 & \R & \R \\ 0 & 1 & \R \\ 0 & 0 & 1\end{smallmatrix}\right) \] defined by
\begin{equation}\label{pidef}
 \pi\left( \left(\begin{smallmatrix} 1 & x & z \\ 0 & 1 & y \\ 0 & 0 & 1\end{smallmatrix}\right)\right) := \left(\begin{smallmatrix} 1 & \st \frac{x}{N} & \st \frac{z}{N^2} \\ 0 & 1 & \st \frac{y}{N} \\ 0 & 0 & 1\end{smallmatrix}\right).
 \end{equation}
This is easily seen to be a homomorphism (of local groups) to the Heisenberg group, whose image is the compact set
\begin{equation}\label{heisenbox}
\left\{ \left(\begin{smallmatrix} 1 & x & z \\ 0 & 1 & y \\ 0 & 0 & 1\end{smallmatrix}\right)
\in \left(\begin{smallmatrix} 1 & \R & \R \\ 0 & 1 & \R \\ 0 & 0 & 1\end{smallmatrix}\right)
: |x|, |y|, |z| \leq 1\right\}.
\end{equation}
Informally, $\pi$ models $\A$ (or $\A^8$) by what is essentially a unit ball in this Lie group.
As before, we have the injectivity-like property that if $a \in \A^{1000}$ is such that $\pi(a)$ is sufficiently close to the identity, then $a \in \A$; as such, one can again establish the \emph{escape property} that if $a, a^2, \ldots, a^{100}$ all lie in $\A^{10}$, then $a$ lies in $\A$.
\end{example}

\begin{example}[Heisenberg box, II]  This is a variant of the preceding example, in which the (not quite) approximate groups $A_\n$ now take the form
\begin{equation}\label{tertius}
A_\n := \left\{ \left(\begin{smallmatrix} 1 & x_{\n} & z_{\n} \\ 0 & 1 & y_{\n} \\ 0 & 0 & 1\end{smallmatrix}\right) : |x_\n|, |y_\n| \leq \n, |z_{\n}| \leq \n^{10}\right\}
\end{equation}
so that the ultralimit $\A := \prod_{\n \rightarrow \alpha} A_\n$ takes the form
\[ \A := \left\{ \left(\begin{smallmatrix} 1 & x & z \\ 0 & 1 & y \\ 0 & 0 & 1\end{smallmatrix}\right) \in \left(\begin{smallmatrix} 1 & \ultra \Z & \ultra \Z \\ 0 & 1 & \ultra \Z \\ 0 & 0 & 1\end{smallmatrix}\right): |x|, |y| \leq N, |z| \leq N^{10}\right\}.\]
 Now consider the map
\[ \pi : \A^8 \rightarrow \R^3\] defined by
\[ \pi\left( \left(\begin{smallmatrix} 1 & x & z \\ 0 & 1 & y \\ 0 & 0 & 1\end{smallmatrix}\right) \right)= \left(\st \frac{x}{N}, \st \frac{y}{N}, \st \frac{z}{N^{10}}\right).\]
The image of this map is the unit cube $[-1,1]^3$, and is in particular compact.  It is also a homomorphism of local groups, since
\[ \pi \left(\left(\begin{smallmatrix} 1 & x & z \\ 0 & 1 & y \\ 0 & 0 & 1\end{smallmatrix}\right)\left(\begin{smallmatrix} 1 & x' & z' \\ 0 & 1 & y' \\ 0 & 0 & 1\end{smallmatrix}\right)\right) = \left(\st \frac{x + x'}{N}, \st \frac{y + y'}{N}, \st \frac{z+ z' + x y'}{N^{10}}\right),\]
but the nonstandard real $x y'/N^{10} = O(N^2/N^{10})$ is infinitesimal, and so the previous expression is equal to
\[ \left(\st\frac{x+x'}{N}, \st \frac{y+y'}{N}, \st \frac{z+z'}{N^{10}}\right)\]
which establishes the homomorphic nature of $\pi$.

Here we note that the homomorphism $\pi: \A^8 \to \R^3$ is not associated to any exact homomorphisms $\pi_\n$ from $\A_n^8$ to $\R^3$.  Instead, it is only associated to \emph{approximate} homomorphisms
$$ \pi_\n\left( \left(\begin{smallmatrix} 1 & x_\n & z_\n \\ 0 & 1 & y_\n \\ 0 & 0 & 1\end{smallmatrix}\right) \right) := \left(\frac{x_\n}{\n}, \frac{y_\n}{\n}, \frac{z_\n}{\n^{10}}\right)$$
into $\R^3$.  Such approximate homomorphisms are somewhat less pleasant to work with than genuine homomorphisms; one of the main reasons why we work in the ultraproduct setting is so that we can use genuine group homomorphisms, or at least local group homomorphisms, throughout the paper.

Note that the preceding example \eqref{nilbox} admits a homomorphism $\tilde\pi$ onto the abelian group $\R^2$ by composing the map \eqref{pidef} with the natural map from $\left(\begin{smallmatrix} 1 & \R & \R \\ 0 & 1 & \R \\ 0 & 0 & 1\end{smallmatrix}\right)$ to its abelianisation $\R^2$. However the kernel of $\tilde\pi$ is, for us, too ``big''. In particular it contains every $\left(\begin{smallmatrix} 1 & 0 & z \\ 0 & 1 & 0 \\ 0 & 0 & 1\end{smallmatrix}\right)$, and in particular contains elements of $\A^8$ not in $\A$. By contrast there are no such elements in the example \eqref{tertius}.  In particular, we can still use the model $\pi$ to establish the same escape property for $\A$ as before, namely that whenever $a, \ldots, a^{100} \in \A^{10}$, one has $a \in \A$.

We also note the sets $A_\n^m$ for small $m$ and large $\n$ grow cubically in $m$ in this example, and quartically in $m$ in the previous example.  This is consistent with the model groups having homogeneous dimension $3$ in the current example and $4$ in the previous example.
\end{example}

In all the above examples, the model group $L$ was a Lie group.  We give now give some examples to show that the model need not \emph{initially} be of Lie type, but can then be replaced with a Lie model after some modification.

\begin{example}[Nonstandard cyclic group, revisited]
The first example is the nonstandard cyclic group $\A := \Z/2^N\Z = \prod_{\n \to \alpha} \Z/2^\n \Z$.  This is a nonstandard finite group and can thus be modeled by the trivial group $\{\id\}$ as discussed in Example \ref{nfg}.  However, it can also be modeled by the compact abelian group $\Z_2$ of $2$-adic integers using the model $\pi: \A \to \Z_2$ defined by the formula
$$ \pi(a) := \lim_{n \to \infty} a \md{2^n}$$
where for each standard natural number $n$, $a \md{2^n} \in \{0,\ldots,2^{n-1}\}$ is the remainder of $a$ modulo $2^n$ (this is well-defined in $\A$) and the limit is in the $2$-adic metric.  Note that the image $\pi(\A)$ of $\A$ is the entire group $\Z_2$, and conversely the preimage of $\Z_2$ in $\A^8 = \A$ is trivially all of $\Z/2^N\Z$; as such, one can quotient out $\Z_2$ in this model and recover the trivial model of $\A$.
\end{example}

\begin{example}[Nonstandard abelian $2$-torsion group]\label{2tor}
In a similar spirit to the preceding example, the nonstandard $2$-torsion group $\A := (\Z/2\Z)^N = \prod_{\n \to \alpha} (\Z/2\Z)^\n$ can be modeled by the compact abelian group $(\Z/2\Z)^\N$ by the formula
$$ \pi(a) := \lim_{n \to \infty} \pi_n(a)$$
where $\pi_n: \A \to (\Z/2\Z)^n$ is the obvious projection, and the limit is in the product topology of $(\Z/2\Z)^\N$.  As before, we can quotient out $(\Z/2\Z)^\N$ and model $\A$ instead by the trivial group.
\end{example}

\begin{remark}\label{inv} The above two examples can be generalised to model any nonstandard finite group $G = \prod_{\n \to \alpha} G_\n$ equipped with surjective homomorphisms from $\G_{\n+1}$ to $\G_\n$ by the inverse limit of the $G_\n$.
\end{remark}

\begin{example}[Lamplighter group]  Let $G$ be the lamplighter group $\Z \ltimes (\Z/2\Z)^\Z$, where $\Z$ acts on $(\Z/2\Z)^\Z$ by the shift $T: (\Z/2\Z)^\Z$ defined by $T(a_n)_{n \in \Z} := (a_{n+1})_{n \in \Z}$. Thus the group law in $G$ is given by
$$ (i,x) (j,y) := (i+j, x+T^i y).$$
For each $\n$, we then set $A_\n \subseteq G$ to be the set
$$ A_\n := \{ (i,x) \in G: i \in \{-1,0,+1\}; x \in (\Z/2\Z)^\n \},$$
where we identify $(\Z/2\Z)^\n$ with the space of elements $(a_n)_{n \in \Z}$ of $(\Z/2\Z)^\Z$ such that $a_n \neq 0$ only for $n \in \{1,\ldots,\n\}$.  These sets $A_\n$ are not quite approximate groups because they are not symmetric, but they are close enough to approximate groups for this discussion. For instance, they have bounded doubling or bounded tripling, and $A_\n \cup A_\n^{-1}$ is an approximate group.  We model the ultraproduct $\A := \prod_{\n \to \alpha} A_\n \subset \Z \ltimes \ultra (\Z/2\Z)^\Z$ by the group
$$G \times_\Z G := \{ ((i,x),(j,y)) \in G \times G: i = j \}$$
using the map
$$ \pi( (i, \lim_{\n \to \alpha} (a^{(\n)}_n)_{n \in \Z} ) ) :=
( (i, (\lim_{\n \to \alpha} a^{(\n)}_n)_{n \in \Z} ), (i, (\lim_{\n \to \alpha} a^{(\n)}_{n+\n} )_{n \in \Z} ) ).$$
Roughly speaking, $\pi(a)$ captures the behaviour of $a$ at the two ``ends'' of $(\Z/2\Z)^N$.  The image $\pi(\A)$ of $\A$ under this model is then the compact neighbourhood of the identity
$$ \pi(\A) = \{ ((i,x),(i,y)) \in G: i \in \{-1,0,+1\}, x \in (\Z/2\Z)^\N, y \in (\Z/2\Z)^{\Z \backslash \N} \}$$
where we embed $(\Z/2\Z)^\N$ and $(\Z/2\Z)^{\Z \backslash \N}$ as subgroups of $(\Z/2\Z)^\Z$ in the usual manner.  One can also compute the images $\pi(\A^m)$ for larger values of $m$, although they are a bit more complicated.  One can verify the escape property that if $g,g^2,\ldots,g^{100} \in \pi(\A^{10})$ for some $g \in G \times G$, then $g \in \pi(\A)$; here it is essential that we use both of the two factors of $G \times_\Z G$, as the claim is false if we project $\pi$ to just one of the two factors $G$, or to the base group $\Z$.  So, in this case, one needs a moderately complicated (though still locally compact) group $G \times_\Z G$ to properly\footnote{This can also be seen from volume growth considerations: $A_\n^m$ grows like $4^m$, which is also the rate of volume growth of $\pi(\A)$ in $G \times_\Z G$, whereas the volume growth in a single factor $G$ would only grow like $2^m$, and the volume growth in $\Z$ is only linear in $m$.} model $\A$ and its powers $\A^m$.  However, if we pass to the large subset $\A'$ of $\A$ defined by $\A' := \prod_{\n \to\alpha} A'_\n$, where
$$ A'_\n := \{ (i,x) \in G: i = 0; x \in (\Z/2\Z)^\n \}$$
then $\A'$ is now a nonstandard finite group (isomorphic to the group $(\Z/2\Z)^N$ considered in Example \ref{2tor} and can be modeled simply by the trivial group $\{\id\}$.  Thus we see that we can sometimes greatly simplify the modeling of an ultra approximate group by passing to a large ultra approximate subgroup.
\end{example}

Let us formalise the properties enjoyed by the above examples in the following definition, which will play a key role in this paper.

\begin{definition}[Good models]\label{good-model-def}
Let $\A$ be an ultra approximate group. A \emph{good model} for $\A$ is a symmetric local topological group $L$ (see Definition \ref{local-def}), together with a homomorphism $\pi : \A^8 \rightarrow L$ of local groups with the following properties:
\begin{enumerate}
\item \textup{(Thick image)} There exists an open neighbourhood of the identity $U_0$ in $L$ such that $\pi^{-1}(U_0) \subseteq \A$ and $U_0 \subseteq \pi(\A)$. In particular $\ker\pi \subseteq \A$;
\item \textup{(Compact image)} $\pi(\A)$ is contained in a compact set.
\item \textup{(Approximation by ``internal'' sets)} Suppose that $F \subseteq U \subseteq U_0$, where $F$ is compact and $U$ is open. Then
there is an ultraproduct $\A' = \prod_{\n \to \alpha} A'_\n$ of finite sets $A'_\n \subseteq A_\n$ such that $\pi^{-1}(F) \subseteq \A' \subseteq \pi^{-1}(U)$.
\end{enumerate}
We will often abuse notation and refer to just $L$ or $\pi$ as the good model for $\A$, rather than the pair $(L,\pi)$.
\end{definition}

\begin{remark} Properties (i) and (ii) together imply that $L$ is locally compact. We leave it to the reader to check that the examples given above have all of the properties of this definition.  One can think of a good model as accurately describing the ``coarse-scale'' structure of the ultra approximate group $\A$, without directly controlling the ``fine-scale'' structure.  For instance, in the example \eqref{tertius} which is ``abelian at coarse scales'' but ``$2$-step nilpotent at fine scales'', the model $\pi$ only detects the abelian structure and not the $2$-step nilpotent structure.
\end{remark}

\begin{remark}\label{ult} In (iii), if $F$ and $U$ are symmetric neighbourhoods of the identity, then $\A'$ can be chosen to be symmetric (since one can replace $\A'$ with $\A' \cap (\A')^{-1}$).   As $L$ is locally compact, we may shrink $U$ to be precompact; then $U^2$ can be covered by finitely many translates of $F$, and thus $\A'$ is then an ultra approximate group.
\end{remark}

Finally, we need to explain the adjective ``virtually'' in Theorem \ref{hrushovski-model}.  In group theory, ``virtually'' means ``after passing to a finite index subgroup''.  Note that a subgroup $G'$ of a group $G$ has finite index if and only if $G$ can be covered by finitely many left-translates -- or, equivalently, right-translates -- of $G'$.  This motivates the following definition.

\begin{definition}[Large approximate subgroups]\label{large}  Let $\A, \A'$ be ultra approximate groups.  We say that $\A'$ is a \emph{large ultra approximate subgroup} of $\A$ if one has $(\A')^4 \subset \A^4$, and $\A$ can be covered by finitely many left-translates of $\A'$.
\end{definition}

\begin{remark} It would be more aesthetically pleasing to have $\A' \subset \A$ instead of $(\A')^4 \subset \A^4$, but we need the exponent $4$ in the inclusion for some minor technical reasons.  Note that the property of being a large ultra approximate subgroup is transitive.
\end{remark}

We are now in a position to state Hrushovski's Lie Model Theorem.

\begin{theorem}[Hrushovski Lie Model Theorem]\label{lie-model}
Let $\A$ be an ultra approximate group. Then there is a large ultra approximate subgroup $\A'$ of $\A$ such that $\A'$ admits a local Lie group as a good model.
\end{theorem}

We will prove this theorem in \S \ref{metric-sec}.  As stated above, the basic idea of the proof is to first establish that $\A$ itself admits a locally compact local group as a good model. Here results of multiplicative combinatorics, and in particular a lemma of Sanders \cite{sanders} (see also \cite{croot-sisask}), are critical.  Once this is done, Theorem \ref{lie-model} follows relatively quickly from the deep results in the literature on Hilbert's fifth problem.  This theorem will then play a key role in the proof of Theorem \ref{main-theorem} in two ways: firstly by allowing us to establish certain ``escape'' properties on (ultra) approximate groups that will be used to build useful metric structures on these groups; and secondly by giving a natural notion of the ``dimension'' of an (ultra) approximate group which we will need to induct on.  Note that one can invoke Lie's third theorem (Theorem \ref{lie-third}) to upgrade the local Lie group in Theorem \ref{lie-model} to a connected, simply connected, global Lie group, but for technical inductive reasons it will be more convenient to keep the model in the category of local Lie groups for now.

Theorem \ref{lie-model} will be proven in \S \ref{metric-sec}.  We will also establish a ``global'' variant of this theorem later, first in a weak form as Proposition \ref{weak-global} and then in a stronger form as Theorem \ref{strong-global}.

\section{An outline of the argument}\label{outline-sec}

In the previous section we introduced the notion of a (Hrushovski) Lie model, one of the key technical tools we will use to prove Theorem \ref{main-theorem}. In this section we outline the argument for this proof as a whole.

Our aim is to show that every $K$-approximate group is controlled in some sense by a coset nilprogression of rank and density $O_K(1)$. We shall prove this by contradiction, assuming that there is a sequence $(A_\n)_{\n \in \N}$ of $K$-approximate groups for which the statement fails in the limit for any given choice of implied constant in the $O_K(1)$ notation. In particular, the cardinality $|A_\n|$ will go to infinity as $\n \to \infty$. We assemble these approximate groups into an ultra approximate group $\A := \prod_{\n \rightarrow \alpha} A_\n$. Our assumption implies that $\A$ is not ``controlled'' in a certain sense by what we call an ultra coset nilprogression, which we now define.

\begin{definition}[Ultra coset nilprogression]  An \emph{ultra coset nilprogression} is an ultraproduct $P = \prod_{\n\to \alpha} P_\n$ of coset nilprogressions $P_\n = P( u_{1,\n},\ldots,u_{r,\n}; N_{1,\n}, \ldots, N_{r,\n})$ of fixed (standard) rank $r$ and step $s$.  We then say that $P$ has rank $r$ and step $s$.  If the $P_\n$ are also all in $C$-normal form for some (standard) $C$ independent of $\n$, we say that the ultra coset nilprogression is \emph{in normal form}.  We call $N_i := \lim_{\n \to \alpha} N_{i,\n}$ for $i=1,\ldots,r$ the \emph{lengths} of the ultra coset nilprogression, and say that the nilprogression is \emph{nondegenerate} if all the $N_i$ are unbounded.

We define the concept of an \emph{ultra nilprogression} similarly, but replacing ``coset nilprogression'' by ``nilprogression'' throughout.
\end{definition}

As with all ultraproducts, it suffices to have the $P_\n$ obey the stated properties for all $\n$ sufficiently close to $\alpha$, as one can redefine $P_\n$ arbitrarily on the remaining values of $\n$ without affecting the ultraproduct $P$.  Note that an ultra nilprogression $P$ can be expressed as
$$ P = P(u_1,\ldots,u_r; N_1, \ldots, N_r)$$
where $r$ is the rank, $u_1,\ldots,u_r$ are elements of the ambient nonstandard local group, and $N_1,\ldots,N_r$ are nonstandard positive reals.

To obtain the contradiction, then, it is sufficient to establish the following ultraproduct version of our main theorem.

\begin{theorem}\label{hl-conj-nonst}
Suppose that $\A$ is an ultra approximate group. Then $\A^4$ contains a nondegenerate ultra coset nilprogression $P$ in normal form with $|P|\gg |\A|$.
\end{theorem}

Here $|P|\gg |\A|$ means that the non-standard numbers $|\A|$ and $|P|$ satisfy $|\A|=O(|P|)$, or in other words that there is a (standard) number $C>0$ such that $|A_\n| \leq C |P_\n|$ for an $\alpha$-large set of $\n \in \N$. See the end of Appendix \ref{nsa-app} for more information.

The Hrushovski Lie model theorem, Theorem \ref{lie-model}, will be a key tool in establishing this, as we discuss below. In addition to this theorem, a further fundamental additional concept in our argument will be the notion of an \emph{escape norm}.

\begin{definition}[Escape norm]\label{escaped}  Let $A$ be a multiplicative set.  For a group element $g \in A^{10}$, we define the \emph{escape norm} $\|g\|_{e,A} \in [0,1]$ to be the quantity
$$ \|g\|_{e,A} := \inf \left\{ \frac{1}{n+1}: n \in \N; g^i \in A \hbox{ for all } 0 \leq i \leq n \right\}.$$
Recall that by convention, the statement $g^i \in A$ is false if $g^i$ is not well-defined.
Now suppose that $\A$ is a nonstandard multiplicative set, i.e. an ultraproduct $\A = \prod_{\n \rightarrow \alpha} A_\n$ of standard multiplicative sets $A_\n$. If $g = \lim_{\n \to \alpha} g_\n$ is an element of $\A^{10}$, we define the \emph{escape norm} $\|g\|_{e,\A} \in \ultra [0,1]$ to be the quantity
$$ \|g\|_{e,\A} := \lim_{\n \to \alpha} \|g_\n\|_{e,A_\n}.$$
\end{definition}

The escape norm can always be defined, but there are some remarkable lemmas essentially due to Gleason \cite{gleason2} concerning its properties when $A$ is an approximate group. Specifically we will show in \S \ref{gleason-sec} that there is a set $A'$ controlling $A$ for which the escape norms satisfy (precise versions of) the following estimates:

\begin{enumerate}
\item \textup{(Product property)} If $g_1,\dots, g_n \in A'$ then $\Vert g_1 \dots g_n \Vert_{e,A'} \ll \Vert g_1 \Vert_{e, A'} + \dots + \Vert g_n \Vert_{e, A'}$;
\item \textup{(Conjugation property)} If $g,h \in A'$ then $\Vert h^{-1} g h \Vert_{e, A'} \ll \Vert g \Vert_{e, A'}$;
\item \textup{(Commutator property)} If $g,h \in A'$ then $\Vert [g,h] \Vert_{e,A'} \ll \Vert g \Vert_{e, A'} \Vert h \Vert_{e, A'}$.
\end{enumerate}

These estimates, which we shall informally term ``Gleason's lemmas'' , will be proven in \S \ref{gleason-sec}. They are valid in both the finitary and the ultralimit settings; the latter will be deduced, quite straightforwardly, from the former.

The remarks in the following paragraph pertain to the finitary situation. To prove the Gleason lemmas, the set $A'$ must be what we call a \emph{strong approximate group}. The precise definition of this is Definition \ref{sag-def}. It is by no means obvious that there is a large strong approximate group $A'$ contained in $A^4$, but this will follow from the Hrushovski Lie model theorem (Theorem \ref{lie-model}), basically because small balls in a Lie group are automatically strong approximate groups, and can then be pulled back by the model map.

One $A'$ has been defined, Gleason's lemmas are proven by an argument closely analogous to that of Gleason himself \cite{gleason2}. We shall say nothing further about the details here; the argument is self-contained and is discussed in \S \ref{gleason-sec}.

With Gleason's lemmas in hand, let us describe the rest of the argument.

Firstly, the set $H = \{ g : \Vert g \Vert_{e, A'} = 0 \}$ of elements which do not escape is a normal (genuine) subgroup of $A'$; this follows from (i) and (ii). We may quotient by $H$ to get an ultra approximate group $A_0 := A'/H$, all of whose non-identity elements have nonzero escape norm.  We shall call such approximate groups \emph{NSS approximate groups}, in analogy with the \emph{no small subgroups} property in the theory of locally compact groups.

Now, if $g_1 \in A_0^4$ is an element other than the identity with smallest (nonzero) $\Vert \cdot \Vert_{e, A_0}$-escape norm then we shall see that in fact, if $A'$ is chosen appropriately, $g_1 \in A_0$.  Item (iii) then implies that for any $h \in A_0$, $[g_1,h] \in A_0^4$ has smaller escape norm than $g_1$, and hence must be the identity. In other words, $g_1$ is central in $A_0$ and we may quotient again to get a new approximate group $A_1 := A_0 /\langle g_1\rangle$. We are being quite fuzzy at this point; in fact, the quotienting takes place in the category of \emph{local} groups and one is quotienting not by the entire group $\langle g_1\rangle$ but by an appropriate geometric progression within it.

Continuing in this vein we pick $g_2 \in A_1^4$ other than the identity with smallest $\Vert \cdot \Vert_{e, A_1}$-norm. We shall see that this norm is automatically nonzero, a consequence of the \emph{local} nature of the quotienting operation.

Continuing further, we pick $g_3,g_4,\dots$.

All of this makes sense at the level of ultralimits as well, and in this setting one can show that $\A_i$ has a Hrushovski Lie model $L_i$ with $\dim L_i < \dim L_{i-1}$ for all $i$. Because of this, the quotienting procedure terminates in finite time with an element $g_k$ and one concludes by reversing these finitely many quotienting operations that $\A$ is controlled by an ultra coset nilprogression with ``generators'' $H, g_1,\dots, g_k$, thereby leading to a proof of Theorem \ref{hl-conj-nonst}.

This concludes our brief summary of the argument. Let us summarise the content of the remaining core sections of the paper.

\begin{itemize}
\item In Section \ref{sanders-croot-sisask-sec}, we discuss results from multiplicative combinatorics, essentially due to Sanders and Croot-Sisask, which are relevant to the proof of Hrushovski's Lie model theorem.
\item In Section \ref{metric-sec}, we prove the Hrushovski Lie model theorem.
\item In Section \ref{finite-deductions-sec}, we use the Hrushovski Lie model theorem to construct strong approximate groups.
\item In Section \ref{gleason-sec}, we state and prove Gleason's lemmas.
\item In Section \ref{endgame}, we give details of the inductive strategy outlined above for constructing $H$ and $g_1,\dots, g_k$, and conclude the proof of Theorem \ref{main-theorem} (except for the rank bound).
\item In Section \ref{rank-reduction} we show that the rank and step of the coset nilprogression can be bounded by $6\log_2 K$ in the global case.
\item Section \ref{gromov-sec} is devoted to various applications to the growth of groups and to Riemannian geometry. We prove there the corollaries stated in the introduction.
\end{itemize}

\section{Sanders-Croot-Sisask theory}\label{sanders-croot-sisask-sec}

In the next section we will establish Hrushovski's Lie Model Theorem (Theorem \ref{lie-model}), in which an ultra approximate group is related first to a locally compact metrisable local group and then, via Goldbring's solution \cite{goldbring-local} of the local Hilbert's Fifth problem, to a local Lie group. In locally compact metrisable local groups we have total boundedness, which means that the unit (say) ball $B(\id,1) := \{x \in G: d(x,\id) \leq 1\}$ may be covered by $O_{\eps}(1)$ smaller balls $B(x_i,\eps) := \{x \in G : d(x,x_i) \leq \eps\}$. On the other hand, by continuity of the group operation, $B(\id,1)$ will contain high powers like $B(\id,\eps)^{100}$ for suitably small $\eps$.

It is not surprising, then, that we need tools for showing (roughly speaking)  that approximate groups $A$ contain high powers of somewhat smaller, but still quite large, approximate subgroups $A'$, which do not immediately escape $A$ in the sense that $(A')^m$ is contained inside $A$ (or perhaps a slightly larger set such as $A^4$) for a reasonably large value of $m$.  Such a tool is provided by a result from multiplicative combinatorics due to Sanders \cite{sanders} and to Croot-Sisask \cite[Theorem 1.6]{croot-sisask}, namely Theorem \ref{scs} below. We shall also need a ``normal'' variant of this result, which essentially follows by combining Theorem \ref{scs} with \cite[Lemma 13.1]{sanders-nonabelian-idempotent}. Our version of this is Theorem \ref{scs-normal} below, and once again we provide a self-contained proof.

Let us remark that by appealing to these results from multiplicative combinatorics we differ fairly substantially from the approach taken by Hrushovski \cite{hrush}, although one may perceive structural similarities in the model-theoretic arguments he uses.

All of the results below are essentially already in the literature, but always for subsets $A$ of some ambient (global) group $G$. As it turns out, though, the proofs of these results end up being equally valid for the more local setting of multiplicative sets. Indeed, most of the tools used in multiplicative combinatorics (with the notable exception of the Fourier transform) are already ``local'' in nature in that they only require one to do $O(1)$ multiplications.

Our first such tool is Ruzsa's covering lemma, which essentially allows one to select a ``complete set of coset representatives'' in the approximate group setting.

\begin{lemma}[Local Ruzsa covering lemma]\label{loc}  Let $A,B$ be finite sets, and suppose that $A \cup B$ is a multiplicative set. Then there exists a finite set $X \subseteq B$ with $|X| \leq |A B|/|A|$ and $B \subseteq A^2 X$. Similarly there exists a finite set $Y \subseteq B$ such that $|Y| \leq |B A|/|A|$ and $B \subseteq Y A^2$.
\end{lemma}
\begin{proof}  Let $X$ be a subset of $B$ such that the sets $A \cdot x$ for $x \in X$ are disjoint, and such that $X$ is maximal with respect to set inclusion; then we have $|X| \leq |A  B|/|B|$.  If $b \in B$, then $A \cdot b$ and $A \cdot x$ must intersect for some $x$, thus $a \cdot b = a' \cdot x$ for some $a, a' \in A$.  Multiplying on the left by $a^{-1}$, we conclude that $b = a^{-1} \cdot a' \cdot x$, and the claim follows.
\end{proof}

A corollary of this is the following result, which allows one to produce an approximate group from a set with small growth.

\begin{corollary}\label{isgroup} Let $A$ be a symmetric multiplicative set, and suppose that $|A^5| \leq K |A|$. Then $A^2$ is a $2K$-approximate group.
\end{corollary}

\begin{proof}  Clearly $A^2$ is a symmetric set containing the identity.  Since $|A^5| \leq K|A| \leq K|A^2|$, we see from Lemma \ref{loc} that there exists $X \subseteq A^4$ with $|X| \leq K$ such that $A^4 \subseteq A^2 X$, and there similarly exists $Y \subseteq A^4$ with $|Y| \leq K$ such that $A^4 \subseteq YA^2$.  Taking the union of $X$ and $Y$ we obtain the claim.
\end{proof}

We turn now to the result of Sanders \cite{sanders} that drives our whole approach.

\begin{theorem}[Small neighbourhoods]\label{scs}
Suppose that $A$ is a $K$-approximate group, and let $m \geq 1$ be an integer. Then there is a $O_{K,m}(1)$-approximate group $S$ with $|S| \gg_{K,m} |A|$ such that $S^m \subseteq A^4$.
\end{theorem}

\begin{remark} Explicit bounds for the implied constants are given in, for example, \cite[Theorem 1.6]{croot-sisask}.  As much of the remainder of the argument is not explicitly effective with respect to bounds, we do not worry about such quantitative issues here. Similar remarks can be made in connection with the normal variant, Theorem \ref{scs-normal} below.
\end{remark}

\begin{proof}  We use the argument from \cite{sanders}, generalised to the setting of multiplicative sets. For the convenience of the reader, we reproduce it here.  A somewhat different proof of Theorem \ref{scs} can also be obtained by using the techniques of \cite{croot-sisask}.

For each $0 < t < 1$, let $f(t)$ denote the quantity
$$ f(t) := \inf \left\{ \frac{|A B|}{|A|}: B \subseteq A; |B| \geq t |A| \right\}.$$
Since $|A^2| \leq K |A|$, we have $1 \leq f(t) \leq K$
for all $0 < t < 1$.  By the pigeonhole principle, we can thus find $t \gg_{K,m} 1$ such that
\begin{equation}\label{fat}
 f\left(\frac{t^2}{2K^2}\right) \geq (1 - \frac{1}{100m}) f(t).
\end{equation}
Fix this $t$.  As there are only finitely many sets $B$ that make up the infimum for $f$, we can find a $B \subset A$ with $|B| \geq t|A|$ such that
\begin{equation}\label{ab}
 |A B| = f(t) |A|.
\end{equation}
For each $a \in A$, the set $B a$ has cardinality $|B|$ and is contained in $A^2$.
$$ \sum_{x \in A^2} \sum_{a \in A} 1_{B  a}(x) = |A| |B|$$
and hence by Cauchy-Schwarz we obtain
$$ \sum_{x \in A^2} (\sum_{a \in A} 1_{B  a}(x))^2 = \frac{|A|^2 |B|^2}{|A^2|}.$$
The left-hand side can be rewritten as \[ \sum_{a \in A} \sum_{a' \in A} |Ba \cap Ba'|,\]
and so by the pigeonhole principle, there exists $a_0 \in A$ such that
$$ \sum_{a \in A} |Ba\cap Ba_0| \geq \frac{|A| |B|^2}{|A^2|}.$$
Since $|B| \geq t|A|$ and $|A^2| \leq K|A|$, we thus have
$$ \sum_{a \in A} |Ba\cap Ba_0| \geq \frac{t^2}{K^2} |A|^2,$$
and hence we can find a subset $C$ of $A$ of cardinality
\begin{equation}\label{cbound}
|C| \geq t^2/2K^2 |A|
\end{equation}
such that $|Ba\cap Ba_0| \geq t^2 |A|/2K^2$
for all $a \in C$.  Multiplying by $a_0^{-1}$ and by $a^{-1}$, we see that $|Bh \cap B| \geq t^2|A|/2K^2$
for all $h \in S_0$, where $S_0 := a_0^{-1} C \cup C^{-1} a_0 \cup \{\id\}$ is a symmetric subset in $A^2$ containing the identity.  From \eqref{fat}, we conclude that
$$ A (Bh \cap B)| \geq \left(1-\frac{1}{100m}\right) f(t) |A|.$$
From \eqref{ab}, we conclude that
$$ | ABh \cap AB | \geq \left(1-\frac{1}{100m}\right) |AB|.$$
Using induction (and the hypothesis that $A^8$ is well-defined, noting that $B \subset A$ and $S_0 \subset A^2$) we then see that for any $1 \leq m' < 100 m$, the set $S_0^{m'}$ is well-defined and
$$ | ABh \cap AB | \geq \left(1-\frac{m'}{100m}\right) |AB|.$$
for all $h \in S_0^{m'}$, which in particular implies that $S_0^{m'} \subset A^4$.  On the other hand, from \eqref{cbound} we have $|S_0| \gg_{K,m} |A|$.  From Corollary \ref{isgroup} we see that $S := S_0^2$ is a $O_{K,m}(1)$-approximate group.  Since $S^m = S_0^{2m} \subset A^4$, we obtain the first claim of the lemma.  The second claim follows by applying the Ruzsa covering lemma (with $B := S_0$).
\end{proof}

\begin{remark}\label{consequence}
Let us pause to note a consequence of this result. We defined multiplicative sets to be ones in which one was at liberty to take up to $100$ multiplications (i.e. $A^{100}$ is well-defined), and the associative law would hold to this extent. Theorem \ref{scs}, or more accurately a close examination of the proof of it, says that if $A$ is an approximate group and a multiplicative set in which merely $8$ multiplications are allowed (i.e. $A^8$ is well-defined) then $A$ is $O_{m,K}(1)$-controlled by an $O_{m,K}(1)$-approximate group $A' = S$ in which up to $m$ multiplications are defined an associative. For this reason Theorem \ref{main-theorem} holds if only $8$ multiplications are allowed. We shall not dwell on such details further in this paper, allowing ourselves the luxury of $100$ multiplications.
\end{remark}

We turn now to proving a ``normal'' variant of Theorem \ref{scs}. Here, we use the notation
\[ a^b := b^{-1} a b \]
and
\[ A^B := \{ a^b : a \in A, b \in B\}\]
for elements $a,b$ and subsets $A,B$ of a local group.

\begin{theorem}[Small normal neighbourhoods]\label{scs-normal}
Suppose that $A$ is a $K$-approximate group, and let $m \geq 1$ be an integer. Let $S \subseteq A^4$ be a $K'$-approximate group with $|S| = \delta |A|$. Then there is an $O_{m,K,K',\delta}(1)$-approximate group $\tilde S$ with $|\tilde S| \gg_{K,K',m,\delta} |A|$ such that $(\tilde S^m)^{A^4} \subseteq S^4$.
\end{theorem}

Theorem \ref{scs-normal} will be deduced from Theorem \ref{scs}.  To motivate the argument, let us first recall a standard lemma from group theory.

\begin{lemma}  Let $A$ be a finite group, and let $S$ be a subgroup of $A$ with $|S| \geq |A|/K$.  Then there exists a further subgroup $\tilde S \subset S$ of $A$ with $|\tilde S| \gg_K |A|$ which is normal in $A$.
\end{lemma}

Note that this lemma would easily yield Theorem \ref{scs-normal} from Theorem \ref{scs} in the special case when $A$ and $\tilde S$ are genuine groups and not merely approximate groups.

\begin{proof}
Let $x_1,\dots,x_k$ be a complete set of right coset representatives for $S$ in $A$, and set
\[ \tilde S = \bigcap x_i^{-1} S x_i = \bigcap_{x \in A} x^{-1} S x.\]
All the claims of the lemma are immediate, except for the claim that $|\tilde S| \gg_K |A|$.  However, this follows from iterating the fact that if $H_1, H_2 \leq G$ are subgroups of small index in a group $G$ then so is $H_1 \cap H_2$; in fact we have the well-known inequality
\begin{equation}\label{index-eq}
[G : H_1 \cap H_2] \leq [G : H_1] [G : H_2].
\end{equation}
\end{proof}

To adapt this argument to the approximate setting we need an analogue of \eqref{index-eq} for approximate groups.  This is provided by the following lemma.

\begin{lemma}
Suppose that $A$ is a $K$-approximate group and that $A_1, A_2 \subseteq A$ are sets with $|A_i| = \delta_i |A|$. Then $A_1 A^{-1}_1 \cap A_2 A^{-1}_2$ contains a set $B B^{-1}$ with $B \subseteq A$ and $|B| \geq \delta_1 \delta_2 |A|/K$.
\end{lemma}

\begin{proof}
Since $A^{-1}_1 A_2 \subseteq A^2$, we have $|A^{-1}_1 A_2| \leq K|A|$. It follows that there is some $x$ with at least $\delta_1\delta_2|A|/K$ representations as $a^{-1}_1 a_2$. Let $B$ be the set of all values of $a_2$ that appear.  Obviously $B B^{-1} \subseteq A_2 A^{-1}_2$. Suppose that $a_2, a'_2 \in B$. Then there are $a_1, a'_1$ such that $x = a_1^{-1} a_2 = (a'_1)^{-1} a'_2$, and so $a'_1 a^{-1}_1 = a'_2 a_2^{-1}$. Thus $B B^{-1}$ lies in $A_1A_1^{-1}$ as well.
\end{proof}

By iterating the above lemma we obtain the following corollary.

\begin{corollary}\label{approx-index-formula}
Suppose that $A$ is a $K$-approximate group and that $A_1,\dots, A_k \subseteq A$ are sets with $|A_i| \geq \delta |A|$ for each $i$. Then $|\bigcap_{i=1}^k A_i A_i^{-1}| \gg_{\delta,k,K} |A|$.
\end{corollary}

Now we can prove Theorem \ref{scs-normal}.

\begin{proof}[Proof of Theorem \ref{scs-normal}]
By Theorem \ref{scs}, there is an $O_{l,K,K'}(1)$-approximate subgroup $S_0 \subseteq S^4$, \[ |S_0| \gg_{m,K, K',\delta} |A|,\] such that
\begin{equation}\label{inclu} S_0^{4m + 4} \subseteq S^4.\end{equation}
The Ruzsa covering lemma allows us to do the analogue of picking a complete set of coset representatives in the approximate group setting. Specifically, there are
$x_1,\dots,x_k$, $k = O_{m,\delta,K}(1)$, such that
\begin{equation}\label{covering}
A^4 \subseteq \bigcup_{i = 1}^k  S_0^2 x_i.
\end{equation}

Let us assume without loss of generality that $x_1 = \id$.

By Corollary \ref{approx-index-formula}, the set
\[ T := \bigcap_{i = 1}^k x_i S_0^{2} x_i^{-1}\]
has cardinality
\[ |T|  \gg_{m,K,K',\delta} |A|.\] We claim that the set $\tilde S := T^2$ has the required properties. First of all note that, by Corollary \ref{isgroup}, $\tilde S$ is indeed an $O_{m,K,K',\delta}(1)$-approximate group.

Next observe that, since $x_1 = \id$,
\begin{equation}\label{observ}  x_i^{-1} T x_i \subseteq S_0^{2}\end{equation} for each $i$.

Suppose that $x \in A^4$. Then,  by \eqref{covering}, there is some $i$ with $1 \leq i \leq k$ and some $s \in S_0^2$ such that $x = s x_i$. It follows from this, \eqref{observ} and \eqref{inclu} that \[ x^{-1} \tilde S^m x = x^{-1} T^{2m} x = s^{-1} x_i^{-1} T^{2m}  x_i s = s^{-1} (x_i^{-1} T x_i)^{2m} s \subseteq S_0^{4m + 4} \subseteq S^4.\] This concludes the proof.
\end{proof}

\section{Proof of the Hrushovski Lie model theorem}\label{metric-sec}

In this section we establish Theorem \ref{lie-model}. The reader may wish to reread \S \ref{corresp}, which gave an overview of this theorem.
We will deduce this theorem from the following two propositions.

\begin{proposition}[Locally compact model]\label{locally-compact-model}
Let $\A$ be an ultra approximate group. Then $\A^4$ admits a model $\pi: \A^{32} \rightarrow G$ by a metrisable locally compact local group $G$.
\end{proposition}

\begin{proposition}[From locally compact models to Lie models]\label{lie-model-prop}
Let $\A$ be an ultra approximate group and suppose that $\A^4$ admits a model $\pi : \A^{32} \rightarrow G$ into a locally compact local group $G$. Then there is a large ultra approximate group $\tilde \A$ of $\A$ \textup{(}thus $\tilde \A^4 \subset \A^4$\textup{)} which admits a model $\tilde\pi : \tilde \A^8 \rightarrow L$ into a connected, simply-connected Lie group $L$.
\end{proposition}
\vspace{11pt}

It is clear that the above two propositions together imply Theorem \ref{lie-model}. \vspace{11pt}

We will give a self-contained proof of Proposition \ref{locally-compact-model}, using the multiplicative combinatorics results of the previous section, together with the countable saturation property of ultraproducts.
In contrast, the proof of Proposition \ref{lie-model-prop} requires deep material related to (the local version of) Hilbert's fifth problem, for which we provide suitable references.\vspace{11pt}

\textsc{Building metrics on local groups.} We now begin the proof of Proposition \ref{locally-compact-model}.
Suppose that we have a pseudometric $d : G \times G \rightarrow [0,\infty)$ on some local group $G$, that is to say $d$ satisfies the axioms of a metric, except that we may have $d(x,y) = 0$ when $x \neq y$. Then we may of course define the balls $B(\id,\eps) := \{ x \in G : d(x, \id) < \eps\}$, and these will be nested in the sense that $B(\id,\eps) \subseteq B(\id,\eps')$ if $\eps < \eps'$. We now examine ways to reverse this construction, beginning with a quite general way to construct pseudometrics on symmetric local groups; this will be needed to prove Proposition \ref{locally-compact-model}.

Let $G$ be a symmetric local group.  For any function $\psi: G \to \R$ and $g \in G$, we define the shift $T_g \psi: G \to \R$ by setting
$$ T_g \psi(x) := \psi(g^{-1} x)$$
if $g^{-1} x$ is well-defined in $G$, and $T_g \psi(x) = 0$ otherwise.  We then define the ``derivative'' operator
$$ \partial_g \psi := \psi - T_g \psi.$$

The expression
$$ \| \partial_g \psi \|_{\ell^\infty(G)} := \sup_{x \in G} |\partial_g \psi(x)| $$
can be viewed heuristically as a ``norm'' of $g$ relative to $\psi$, and this makes it natural to consider the function
\begin{equation}\label{d-def}
d(g,h) := \Vert T_g \psi - T_h \psi \Vert_{\ell^{\infty}(G)} = \Vert \partial_{h^{-1} g}\psi \Vert_{\ell^{\infty}(G)}.
\end{equation}
One can view $d$ as the pullback of the metric on $\ell^\infty(G)$ to $G$ using the translation action $g \mapsto T_g \psi$ of $G$ on $\psi$.

\begin{lemma}[Using functions to build (pseudo-)metrics]\label{supp}  Let $G$ be a local group, and let $A$ be a symmetric neighbourhood of the identity such that $A^{128}$ is well-defined in $G$.  Let $\psi: G \to \R$ be non-negative and supported on $A$.
\begin{enumerate}
\item[(i)]  We have $\| \partial_g \psi \|_{\ell^\infty(G)} \leq \| \psi \|_{\ell^\infty(G)}$ for all $g \in A^{128}$, with equality holding when $g \not \in A^2$.
\item[(ii)] Whenever $g, h \in A^{128}$, one has
\begin{equation}\label{ghpsi}
 \| \partial_{gh} \psi\|_{\ell^\infty(G)} \leq \| \partial_g \psi \|_{\ell^\infty(G)} + \| \partial_h \psi \|_{\ell^\infty(G)}.
\end{equation}
\item[(iii)]  For any $g \in A^{128}$, we have
\begin{equation}\label{ginv}
\| \partial_{g^{-1}} \psi \|_{\ell^\infty(G)} = \| \partial_g \psi \|_{\ell^\infty(G)}.
\end{equation}
\item[(iv)]  The function $d: A^{64} \times A^{64} \to \R^+$ defined by the formula \eqref{d-def} is a left-invariant pseudo-metric on $A^{64}$.\end{enumerate}
\end{lemma}

\begin{remark}  To spell out what we mean in (iv), we are asserting that $d(g,g) = 0$, that $d(g,h) = d(h,g)$, and that $d(g,k) \leq d(g,h) + d(h,k)$ for all $g,h,k \in A^{64}$.  Furthermore it has the left-invariance property $d(gh,gk) = d(h,k)$ whenever $h,k \in A^{64}$, $g \in A^{128}$, and $gh,gk \in A^{64}$.
Later on, when proving Gleason's lemmas, we shall require some slightly more exotic properties of these cocycle ``norms'', related to commutation and a certain ``Taylor expansion''.
\end{remark}

\begin{proof}  The property (i) is clear from construction.  For $g,h \in A^{128}$ we have the representation property $T_g T_h \psi = T_{gh} \psi$ and hence the \emph{cocycle identity}
$$ \partial_{gh} \psi = \partial_g \psi + T_g \partial_h \psi$$
which gives \eqref{ghpsi}.

Similarly, for $g \in A^{128}$ we have the inverse identity
$$ \partial_{g^{-1}} \psi = - T_{g^{-1}} \partial_g \psi$$
which gives \eqref{ginv}.

The claims in (iv) follow easily from (ii) and (iii).
\end{proof}

In the next lemma we give a variant of the Birkhoff-Kakutani construction \cite[\S 1.22]{montgomery-zippin-book}, in which a function $\psi$ is constructed so that the pseudometric $d(g,h) = \Vert \partial_{h^{-1} g} \psi \Vert_{\ell^{\infty}(G)}$ is adapted to a given nested sequence of symmetric sets which are supposed to resemble ``balls'' in this pseudometric.

\begin{lemma}[Birkhoff-Kakutani construction]\label{birkhoff-kakutani}
Suppose that $G$ is a local group and that we have a sequence of symmetric neighbourhoods $A_0, A_1,\dots$ of the identity in $G$ with the nesting property that $A_{i+1}^2 \subseteq A_i$ for $i = 0,1,2,\dots$, and with $A_0^{200}$ well-defined.
Then there is a pseudometric
\[ d : A_0^{64} \times A_0^{64} \rightarrow [0,1]\]
such that we have the inclusions
\begin{equation}\label{key-inclusions} \{ g \in A_0^{64} : d(g,\id) < 2^{-k} \}\subseteq A_k \subseteq \{ g \in A_0^{64} : d(g,\id) \leq 2 \cdot 2^{-k} \}\end{equation} for all nonnegative integers $k$. In particular $x_n \rightarrow x$ in the pseudometric $d$ if and only if, for each $k \in \N$, we have $x^{-1} x_n \in A_k$ for all sufficiently large $n$.
\end{lemma}
\begin{proof}
Suppose that $q = 2^{-i_1} + \dots + 2^{-i_k}$, $0 < q < 1$, is a dyadic rational, and define
\[ B_q := A_{i_k} A_{i_{k-1}} \dots A_{i_1}.\]
Even though the definition uses a potentially large number $k$ of multiplications, the nesting property of the $A_i$ means that these sets $B_q$ are well-defined in the local group $G$.

We claim that $B_{q} \subseteq B_{q + 2^{-k}}$ whenever $q$ is a dyadic rational with denominator dividing $2^k$; this easily implies that
\begin{equation}\label{nesting-relation} B_q \subseteq B_{q'} \qquad \mbox{whenever $0 < q < q' < 1$}.\end{equation}
The claim follows by repeated use of the nesting $A_{i+1}^2 \subseteq A_i$ (the number of times it will be required is the number of carries when $2^{-k}$ is added to $q$ in binary).

In particular, $B_q \subseteq A_{i_1 - 1} \subset A_0$.

Define $\psi : A_0^{64} \rightarrow [0,1]$ by
\[ \psi(x) := \sup \{ 1 - q: 0 < q < 1; x \in B_q\} \cup \{0\},\]
and consider the pseudometric $d(g,h) := \Vert \partial_{h^{-1} g} \psi \Vert_{\ell^{\infty}(G)}$ as discussed in Lemma \ref{supp}.  Note that for $g, h \in A_0^{64}$, $\partial_{h^{-1} g} \psi$ is supported in $A_0^{192}$, and so one can replace $\ell^\infty(G)$ here with $A_0^{192}$ if desired.

If $d(g,\id) < 2^{-k}$ then $|\partial_g \psi(\id)| <  2^k$, which implies that $\psi(g) > 1 - 2^{-k}$ and therefore $g \in B_{2^{-k}}$ and hence $g \in A_k$.

Conversely, suppose that $g \in A_k$: we are to show that $d(g,\id) \leq 2 \cdot 2^{- k}$. To show this we must confirm that $|\partial_g \psi(h)| < 2^{1 - k}$ for all $h \in G$.   As discussed before, we may assume that $h \in A_0^{192}$.  Suppose that $h \in B_q$, where $0 < q < 1 - 2^{-k}$ is an integer multiple of $2^{-k}$, but that $h \notin B_{q - 2^{-k}}$. Then $\psi(h) \leq 1 - q + 2^{-k}$. On the other hand, $g^{-1} h \in A_k B_q \subseteq B_{q + 2^{-k}}$, by the claim established above, and therefore $\psi(g^{-1} h) \geq 1 - q - 2^{-k}$. It follows that $\partial_g \psi (h) = \psi(g^{-1} h) - \psi(h) \geq - 2\cdot 2^{-k}$. Similarly, $\partial_g \psi(h) \leq 2 \cdot 2^{-k}$. Since $h$ was arbitrary it follows that $d(g,\id) = \Vert \partial_g\psi \Vert_{\ell^{\infty}(G)} \leq 2 \cdot 2^{-k}$, and the claim follows.\end{proof}

If the sets $A_i$ satisfy a certain normality condition, the group operations are continuous with respect to the pseudometric $d$:

\begin{lemma}[Normal Birkhoff-Kakutani construction]\label{normal-birkhoff-kakutani}
Suppose that $G$ is a local group and that we have a sequence of symmetric sets $A_0, A_1,\dots$ in $G$ with $A_0^{200}$ well-defined and with the nesting property that $(A_{i+1}^2)^{A_0^{100}} \subseteq A_i$ for $i = 0,1,\dots$ \textup{(}and so, in particular, we certainly have the weaker nesting property $A_{i+1}^2 \subseteq A_i$ required by the preceding lemma\textup{)}. Consider the pseudometric $d : A_0^{64} \times A_0^{64}  \rightarrow [0,1]$ defined in the preceding lemma. Then the product map $\cdot A_0^{32} \times A_0^{32} \rightarrow A_0^{64}$ and the inversion map ${}^{-1} : A_0^{32} \rightarrow A_0^{32}$ are both continuous with respect to $d$.
\end{lemma}
\begin{proof}
Suppose that $g_n \rightarrow g$ and that $h_n \rightarrow h$. We wish to show that $g_n h_n \rightarrow gh$, to which end it suffices to establish that $(gh)^{-1} g_n h_n \in A_k$ for all sufficiently large $n$.  However, for $n$ sufficiently large in terms of $k$ we have $g^{-1} g_n \in A_{k+2}$, and hence
\[ h_n^{-1} g^{-1} g_n h_n \in A_{k+2}^{h_n} \subseteq A_{k+2}^{A_0^{100}} \subseteq A_{k+1}.\]
Furthermore, $h^{-1} h_n \in A_{k+1}$ for $n$ sufficiently large, and so
\[ (gh)^{-1} g_n h_n = (h^{-1} h_n) (h_n^{-1} g^{-1} g_n h_n) \in A_{k+1}^2 \in A_k,\] as required.
The statement about the inverse map is easier. Suppose that $g_n \rightarrow g$. Then $g^{-1} g_n \in A_{k+1}$ for $n$ sufficiently large, and so
\[ g_n g^{-1} = g (g^{-1} g_n) g^{-1} \in A_{k+1}^{g^{-1}} \subseteq A_{k+1}^{A_0^{100}} \in A_k.\]
But this means that $g_n^{-1} \rightarrow g$ as $n \rightarrow \infty$.
\end{proof}

The previous lemma showed how to get a local topological group given a sequence of balls satisfying a suitable normalisation condition. The normal variant of the Croot-Sisask-Sanders lemma, Theorem \ref{scs-normal}, allows us to find precisely such a sequence of balls given any $K$-approximate group $A$.  Of course, these balls are just finite sets and, for sufficiently large $i$, $A_i$ may well consist only of the identity element $e$. This will be the case, for example, when $A = [-N, N]$. However when transferred to the setting of an ultra approximate group $\A = \prod_{n \rightarrow \alpha} A_n$, these balls have ``finite index'' in $\A$, and this ultimately leads to the important conclusion that the metric $d$ gives $\A$ the structure of a locally compact local group.

\begin{lemma}\label{ultraproduct-balls}
Let $\A$ be an ultra approximate group.  Then there is a sequence of ultra approximate groups $\A_0, \A_1, \dots$ such that $\A_0 = \A^4$,
we have the nesting property that $(\A_{i+1}^2)^{\A_0^{100}} \subseteq \A_i$ for $i = 0,1,\dots$, and each $\A_i$ is \emph{large} in the sense that $\A$ can be covered by finitely many left-translates of $\A_i$.
\end{lemma}

\begin{proof}
By definition, one has $\A = \prod_{\n \rightarrow \alpha} A_\n$ for some $K$-approximate groups $A_\n$ and some fixed $K$.
Applying Theorem \ref{scs-normal} repeatedly we see that there are, for each $\n$, $O_{K,i}(1)$-approximate groups $S_{\n,i}$, $i = 1,2,3\dots$, such that $S_{\n,0} := A_\n$ and $(S_{\n, i+1}^8)^{S_{\n,0}^{400}} \subseteq S_{\n,i}^4$ for each $i$. Furthermore we have $S_{\n,i} \subseteq A_\n^4$ and $|S_{\n,i}| \gg_{K,i} |A_\n|$ for each $i$.
Setting
\[ \A_i := \prod_{\n \rightarrow \alpha} S_{\n,i}^4,\] all of the properties except the assertion about covering are immediate. To check that each $\A_i$ is large, we need only check that $S_{\n,0}$ is covered by $O_{K,i}(1)$ left-translates of $S_{n,i}^4$, for each $i$. This, however, is an immediate consequence of Lemma \ref{loc} and the lower bound on $|S_{\n,i}|$.
\end{proof}

\begin{lemma}\label{local-compact}
Let $\A$ be an ultra approximate group.  Consider a sequence of ultra approximate groups $\A_0, \A_1,\dots$ as found in the preceding lemma, and let $d : \A^{32} \times \A^{32} \rightarrow [0,1]$ be the pseudometric associated to these sets as in Lemma \ref{birkhoff-kakutani}. Then $\A^{32}$ is locally compact with respect to the topology generated by $d$.
\end{lemma}

\begin{proof}
By the Heine-Borel theorem (which is usually stated for metrics, but which extends without difficulty\footnote{Indeed, one can deduce the pseudometric case from the metric case by quotienting out by the equivalence relation $x \sim y$ defined by the equation $d(x,y)=0$.} to pseudometrics) it suffices to show that $\A^{32}$ is complete and totally bounded. We deal with the latter task first. From the inclusion $\A_k \subseteq \{x : d(x,\id) \leq 2 \cdot 2^{-k}\}$ and the left-invariance of $d$, this follows from the fact that $\A^{32}$ is covered by finitely many left-translates of $\A_k$.

We turn now to completeness. Suppose that $(x_n)_{n \in \N}$ is a Cauchy sequence. By refining the sequence if necessary we may assume that it is rapidly Cauchy in the sense that $d(x_n, x_m) \leq 2^{-n - 1}$.

We claim that the sets $x_n \A_n$ are nested in the sense that $x_m \A_m \subseteq x_n \A_n$ whenever $m > n$. To see this note that by left-invariance we have $d(\id, x_n^{-1} x_m) \leq 2^{-n - 1}$ and hence, by the inclusions of Lemma \ref{birkhoff-kakutani}, $x_n^{-1} x_m \in \A_{n+1}$. Since $\A_{n+1}\A_m \subseteq \A_{n+1}^2 \subseteq \A_n$, it follows that $x_n^{-1} x_m \A_m \subseteq \A_n$, thereby confirming the claim.

Now each set $x_m \A_m$ is an ultraproduct $\prod_{n \rightarrow \alpha} S_{m,n}$, by construction. The nesting property just established of course implies that, for any positive integer $M$, $\bigcap_{m \leq M} x_m \A_m \neq \emptyset$. Let $y_M$ be an element of this intersection; this means that there is a set $\Sigma_M \in \alpha$ such that $(y_M)_n \in \bigcap_{m \leq M} S_{m,n}$ for all $n \in \Sigma_M$. By replacing $\Sigma_2$ with $\Sigma_1 \cap \Sigma_2$ if necessary, and so on, and using the basic properties of ultrafilters, we may assume that $\Sigma_1 \supseteq \Sigma_2 \supseteq \Sigma_3 \supseteq \dots$. By removing $2$ from $\Sigma_2$, $3$ from $\Sigma_3$ and so on, if necessary, we may also assume that no integer lies in infinitely many $\Sigma_M$.

Now define a sequence $x$ by setting $x_n = (y_M)_n$, where $M$ is the largest integer for which $n \in \Sigma_M$. Then, by construction, $x_n \in \bigcap_{m \leq M}S_{m,n}$ for all $n \in \Sigma_M$, that is to say for a set of $n$ tending to $\alpha$. This means that $x \in \bigcap_{m \leq M} x_m \A_m$ for every $M$, and hence $x \in \bigcap x_m \A_m$. In particular we have $x_m^{-1} x \in \A_m$ for every $m$ and hence $d(x,x_m) \leq 2 \cdot 2^{-m}$. It follows that $x_m \rightarrow x$, thereby confirming that $\A$ is complete with the metric $d$.\end{proof}

\begin{remark} The last part of this argument, in which an element is found in the infinite intersection $\bigcap_m x_m \A_m$ given that each finite intersection $\bigcap_{m \leq M} x_m \A_m$ is nonempty, is an instance of the \emph{countable saturation} property of the ultraproduct construction.  The completeness that is afforded by the countable saturation property is one of the main reasons why we work in the ultraproduct setting.  Note that a similar completeness also appears in the ultralimit $(X,d)/\sim$ of bounded metric spaces $(X_\n, d_\n)$, where $X := \prod_{\n \to\alpha} X_n$, $d := \st \lim_{\n \to \alpha} d_\n$, and $\sim$ is the equivalence relation defined by setting $x \sim y$ whenever $d(x,y) = 0$.  Indeed, it is not difficult to use countable saturation to verify that such ultralimits are automatically complete, even if the original spaces $X_\n$ are not.
\end{remark}

\begin{proof}[Proof of Proposition \ref{locally-compact-model}] We have shown that $\A^{32}$ has the structure of a locally compact local group with respect to the metric $d$. To complete the proof of Proposition, we need only quotient by the equivalence relation $\sim$ on $\A^{32}$, defined by $x \sim y$ if and only if $d(x,y) = 0$. The quotient $L := \A^{32} / \sim$ is then a metrisable, locally compact, local group and there is a natural map $\pi : \A^{32} \rightarrow L$. We must check that $L$ is a good model for $\A^4$ in the sense of Definition \ref{good-model-def}.

Property (i) requires us to show that there is some open neighbourhood $U_0$ of the identity in $L$ such that $\pi^{-1}(U_0) \subseteq \A^4$ and $U_0 \subseteq \pi(\A^4)$, or in other words some ball $\{x \in \A^{32} : d(x,\id) < \eps\}$ lies in $\A^4$. This again follows from \eqref{key-inclusions} and the fact that each of the sets $\A_k$ constructed in Lemma \ref{ultraproduct-balls} lies in $\A^4$.

Finally, property (ii) in the definition of good model requires us to show that $\pi(\A^4)$ is compact. This is immediate.

To prove property (iii), we first establish the following weaker property:

(iii)' : for any open neighbourhood $U$ of the identity in $L$ there is some $U' \subseteq U$ and some ultra finite set $\A' = \prod_{n \rightarrow \alpha} A'_n$ with $\pi^{-1}(U') \subseteq \A' \subseteq \pi^{-1}(U)$.

This is quite easily established: suppose that $U$ contains the ball $B(\id, 2^{-k})$. Then it follows immediately from the inclusions of Lemma \ref{birkhoff-kakutani} that we may take $\A' := A_{k+1}$ and then $U' := B(\id, 2^{-k-1})$.

We now upgrade this to property (iii) in the definition of good model. Suppose that $F \subseteq U \subseteq U_0$ with $F$ compact and $U$ open. Then there is some open neighbourhood of the identity $U'$ such that $F U' \subseteq U$. Applying (iii)', we may locate a further open set $U'' \subseteq U'$ and an ultra finite set $\A'$ such that $\pi^{-1}(U'') \subseteq \A' \subseteq \pi^{-1}(U')$. By compactness there are elements $x_1,\dots,x_M$ such that $F \subseteq \bigcup_{m = 1}^M x_m U''$; we may assume that these elements lie in $F (U'')^{-1} = F U'' \subseteq U \subseteq U_0$, and hence each is of the form $x_i = \pi(a_i)$ with $a_i \in A$. To conclude the proof of property (iii) simply take $\A'' := \bigcup_{m = 1}^M a_m \A'$. This completes the proof of Proposition \ref{locally-compact-model}.
\end{proof}

To complete the proof of Theorem \ref{lie-model}, we invoke results about Hilbert's fifth problem, and specifically the structural theorem of Goldbring \cite{goldbring-thesis} describing locally compact local groups, which we state as Theorem \ref{goldbring-thm} in the appendix.

\begin{proof}[Proof of Proposition \ref{lie-model-prop}]  Suppose that we have a model $\pi : \A^{32} \rightarrow G$ from $\A^{32}$ to a locally compact local group $G$, and let $U_0$ be the open neighbourhood of the identity featuring in the definition of good model (Definition \ref{good-model-def}), thus $\pi^{-1}(U_0) \subset \A^4$ and $U_0 \subset \pi(\A^4)$. By Theorem \ref{goldbring-thm}, there are symmetric neighbourhoods $U_2 \subseteq U_1 \subseteq U_0 \subseteq G$ with $U_2^{24} \subseteq U_1$ (say) and a compact normal subgroup $H$ of $U_2$ such that $U_1/H$ is isomorphic to a local Lie group $L$. Let $\phi : U_1 \rightarrow U_1/H$ be the projection map.

By property (iii) of Definition \ref{good-model-def} (applied to $\pi : \A \rightarrow G$) there is a symmetric ultra finite set $\tilde\A \subseteq \A^4$ with $\pi^{-1}(U^2_2) \subseteq \tilde\A\subseteq \pi^{-1}(U^3_2)$. Certainly, the map $\tilde \pi := \phi \circ \pi$ is well-defined and gives a homomorphism from $\tilde\A^8$ to $L$; since $\pi^{-1}(U^3_2)^4 \subset \pi^{-1}(U^{12}_2) \subset \A^4$, we have $\tilde \A^4 \subseteq \A^4$, and by Remark \ref{ult}, $\tilde \A$ is an ultra approximate group. We verify that this is a good model by checking (i), (ii) and (iii) of Definition \ref{good-model-def} in turn. For (i), first note that $\tilde\pi(\tilde\A)$ contains $\tilde U_0 := \phi(U_2) = U_2 H/H \subseteq L$, which is an open neighbourhood of the identity in $L$ since $U_2 H \subseteq G$ is open.

Furthermore we have
\[ \tilde \pi^{-1}(\tilde U_0)  = \pi^{-1} \phi^{-1} \phi (U_2) \subseteq \pi^{-1} (U_2 H) \subseteq \pi^{-1} (U_2^2) \subseteq \tilde \A.\]

Turning to (ii), $\tilde \pi(\tilde \A)$ is contained in the compact set $
\phi(\overline{U_2^2})$.

Finally, we check the ``approximation by internal sets'' property, which is (iii) in Definition \ref{good-model-def}. Suppose that $\tilde F \subseteq \tilde U \subseteq \tilde U_0$, with $\tilde F$ compact and $\tilde U$ open. Then $\phi^{-1} (\tilde F) = \tilde F H$ is compact, whilst $\phi^{-1}(\tilde U) = \tilde U H$ is open. The approximation by internal sets property then follows from that fact that $\pi :\A \rightarrow G$ is a good model.

Finally, we check that $\tilde \A$ is a large ultra approximate group. To see this note that $\tilde\pi (\tilde \A^2)$ is contained in a compact subset of $L$; therefore there are finitely many elements $x_1,\dots, x_k$ such that $\tilde \pi(\tilde \A^2) \subseteq \bigcup_{i=1}^k \tilde\pi(x_k) \tilde U_0$. It follows that
\[ \tilde \A^2 \subseteq \bigcup_{i=1}^k x_k \tilde\pi^{-1} U_0 \subseteq \bigcup_{i=1}^k x_k \tilde \A,\] thereby confirming that $\tilde \A$ is an ultra approximate group.
By essentially the same argument, $\A$ may be covered by finitely many translates of $\tilde \A$; thus $\tilde \A$ is indeed large.
\end{proof}

We now record some analogues of the above results in the setting of \emph{global} ultra approximate groups (i.e. ultraproducts of global $K$-approximate groups for some fixed $K$), which are closer to the results of Hrushovski \cite{hrush}.  Define a \emph{global model} $\pi: \langle \A \rangle \to G$ to be the same notion as a good model $\pi: \A^8 \to G$ from Definition \ref{good-model-def}, except that $\A^8$ is replaced by the whole group $\langle \A \rangle$ generated by $\A$, and $G$ is now required to be a global group rather than a local group.

\begin{proposition}[Global locally compact model]\label{locally-compact-model-global}
Let $\A$ be a global ultra approximate group. Then $\A^4$ admits a global model $\pi: \langle \A \rangle \rightarrow G$ by a metrisable locally compact global group $G$.
\end{proposition}

\begin{proof}  This is obtained by a modification of the proof of Proposition \ref{locally-compact-model}.  The one main change is that the nesting condition $(\A_{i+1}^2)^{\A_0^{100}} \subseteq \A_i$ appearing in Lemma \ref{ultraproduct-balls} needs to be strengthened to $(\A_{i+1}^2)^{\A_0^{100(i+1)}} \subseteq \A_i$, but this is easily accomplished.
\end{proof}

\begin{proposition}[From locally compact models to Lie models]\label{lie-model-prop-global}
Let $\A$ be a global ultra approximate group and suppose that $\A^4$ admits a global model $\pi : \langle \A \rangle \rightarrow G$ into a locally compact global group $G$. Then there is a large ultra approximate group $\tilde \A$ of $\A$ which admits a global model $\tilde\pi : \langle \A\rangle \rightarrow L$ into a connected Lie group $L$.
\end{proposition}

\begin{proof} This is obtained by a modification of the proof of Proposition \ref{lie-model-prop}.  The one main change is that one needs to replace Theorem \ref{goldbring-thm} with Theorem \ref{gleason-yamabe}.
\end{proof}

Note that in contrast to Proposition \ref{lie-model-prop} that we do \emph{not} assert that the global Lie group $L$ is simply connected (as this is not provided by the global Gleason-Yamabe theorem (Theorem \ref{gleason-yamabe}), which only promises connectedness).  And indeed, in general we do not have simple connectedness of the model.  For instance, if $\A = \{-N,\ldots,N\} \subset \Z/100N\Z$ for some unbounded nonstandard natural number $N$, then the obvious global model here is the map $\pi: \Z/100N\Z \to \R/\Z$ defined by $\pi(x) = \st(\frac{x}{100N}) \mod 1$, and of course the unit circle $\R/\Z$ is not simply connected.  On the other hand, $\A^{100} = \Z/100N\Z$ is globally modeled by the trivial group; and so one can still recover simple connectedness by passing from $\A$ to a suitably large power.  See \cite[Remark 4.11]{hrush} for some further discussion of this point, as well as Theorem \ref{strong-global} below.

Combining Proposition \ref{locally-compact-model-global} and Proposition \ref{lie-model-prop-global} we obtain the following result, originally due to Hrushovski \cite{hrush}).

\begin{proposition}[Weak global Lie model theorem]\label{weak-global}
Suppose that $\A$ is a global ultra approximate group.  Then there is a large ultra approximate group $\tilde \A$ of $\A$ which admits a global model $\tilde\pi : \langle \A\rangle \rightarrow L$ into a connected Lie group $L$.
\end{proposition}

We will strengthen this proposition in Theorem \ref{strong-global} below.

\begin{remark}\label{haar}  Let $\pi: \A^8 \to L$ be a good model for an ultra approximate group $\A = \prod_{\n \to \alpha} A_\n$ by a locally compact local group $L$, and let $U_0$ be the neighbourhood in Definition \ref{good-model-def}.  Let $U_1$ be a symmetric neighbourhood of the identity such that $U_1^{100} \subset U_0$.  For any continuous function $f: L \to \R$ with compact support in $U_1$, we can define a functional $I(f)$ by the formula
$$ I(f) = \inf \st \frac{\sum_{a \in \A} F^+(a)}{|\A|}$$
where $F^+ = \lim_{\n \to \alpha} F^+_\n$ is the ultralimit of functions $F^+_\n: A_\n \to \R$, with the nonstandard real $\sum_{a \in \A} F^+$ and nonstandard natural number $|\A|$ defined in the usual fashion as
$$ \sum_{a \in \A} F^+(a) := \lim_{\n \to \alpha} \sum_{a_\n \in A_\n} F^+_\n(a_\n)$$
and
$$ |\A| := \lim_{\n \to \alpha} |A_\n|,$$
and the infimum is over all $F^+$ for which $F^+(a) \geq f(\pi(a))$ for all $a \in \A$.  Using Definition \ref{good-model-def}(iii) it is not difficult to also obtain the equivalent formula
$$ I(f) = \sup \st \frac{\sum_{a \in \A} F^-(a)}{\sum_{a \in \A} 1}$$
where the supremum is over all $F^-$ for which $F^-(a) \leq f(\pi(a))$ for all $a \in \A$.  From these two definitions we see that $I(f)$ is both super-linear and sub-linear, and is thus a continuous linear functional on the space $C_c(U_1)$ of continuous compactly supported functions in $U_1$.  By the Riesz representation theorem, there thus exists a Radon measure $\mu$ on $U_1$ such that $I(f) = \int_{U_1} f\ d\mu$ for all $f \in C_c(U_1)$.  From the translation invariant properties of $I(f)$, we see that $\mu(gE)=\mu(E)$ for any measurable subset $E$ of $U_1$, and any $g \in L$ such that $gE$ are defined in $U_1$, and similarly for $gE$ replaced by $Eg$.  Thus $\mu$ is a bi-invariant Haar measure on $U_1$; since $\A$ can be covered by finitely many left-translates of $\pi^{-1}(F)$ for any compact neighbourhood $F$ of the identity, we see that $\mu$ is non-trivial (which implies in particular by bi-invariance that the locally compact local group $L$ is unimodular).  This Haar measure can then be used to estimate the (nonstandard) cardinality of various nonstandard finite sets that are ``close'' to $\A$ in some sense.  Indeed, from the definitions (and the regular nature of Radon measures) we see that
$$ \mu(F) |\A| \leq |\A'| \leq \mu(U) |\A|$$
whenever $F \subseteq U \subseteq U_1$, $F$ is compact, $U$ is open, and $\A'$ is a nonstandard set with
$$ \pi^{-1}(F) \subset \A' \subseteq \pi^{-1}(U).$$
We will not use this measure $\mu$ in this paper, but see \cite{hrush} for some further discussion of this measure and its relationship to \emph{Kiesler measures} from model theory.  One can also use $\mu$ to relate the volume growth of $\A^m$ to the volume growth of the model group $L$, giving some rigorous substance to some of the volume growth heuristics invoked in the examples in Section \ref{corresp}, but we will not formalise this relationship here.
\end{remark}

\begin{remark} As remarked in \cite{hrush}, the Lie Model theorem is not only valid in the context of nonstandard finite ultra approximate groups, i.e. the ultraproduct of finite $K$-approximate groups for a fixed $K$, but also for ``continuous'' ultra approximate groups, that is to say the ultraproduct of precompact open subsets of a locally compact local group that obey all of the approximate group axioms other than finiteness. See \cite{tao-noncommutative} for the basic theory of such continuous approximate groups.  Indeed, one can check that the machinery in Section \ref{sanders-croot-sisask-sec} can be adapted to this setting by replacing the cardinality of finite sets with the Haar measure of various precompact open subsets of a locally compact local group, as in \cite{tao-noncommutative}.  Some other components of this paper, such as the construction of strong approximate groups and Gleason metrics, can also be extended to this setting after some minor notational changes.  However, there will be a key place in the argument\footnote{Another, much more minor, place where ultra finiteness is used in Remark \ref{haar} above, as we implicitly used the trivial fact that counting measure is bi-invariant.  In general, one can only conclude that the measure associated to a good model is bi-invariant if each of the individual approximate groups in the ultraproduct is also equipped with a finite bi-invariant measure.} in Section \ref{endgame} in which the (nonstandard) finiteness of the ultra approximate groups is used in an absolutely crucial way, namely to locate an element in such a group element of minimal non-zero ``escape norm''.  As such, the main result of this paper, Theorem \ref{main-theorem}, does not immediately extend to the continuous setting. Indeed, the basic example of a small ball in a Lie group shows that continuous approximate groups need not resemble coset nilprogressions at all.  We will not pursue this matter further here.
\end{remark}

\noindent \textsc{Some finitary consequences of the Lie Model Theorem.}  To illustrate the power of the Lie Model Theorem in the analysis of approximate groups, we offer two fairly quick applications. The reader interested in the proof of our main results may skip ahead to the next section.

The first application is a special case of our main theorem (Theorem \ref{main-theorem}), following Hrushovski \cite[Corollary 4.18]{hrush}.

\begin{theorem}[Hrushovski]\label{babyhrush}
Suppose that $G$ be a group of exponent $m$ and suppose that $A \subseteq G$ is a $K$-approximate group. Then $A^4$ contains a genuine subgroup $H$ of $G$ with $|H| \gg_{K,m} |A|$. In particular, by Lemma \ref{loc}, $A$ is covered by $O_{K,m}(1)$ left-translates of $H$.
\end{theorem}

\begin{remark}
When $m = 2$ the group $G$ must be abelian, and in this case the theorem is due to Imre Ruzsa \cite{ruzsa-finite-field}.
\end{remark}

\begin{proof} Suppose for sake of contradiction that the claim failed.  Then we may find fixed $K, m$ and a sequence of $K$-approximate groups $A_\n \subseteq G_\n$ in groups $G_\n$ of exponent $m$, such that for each $\n$, $A_\n^4$ does \emph{not} contain a genuine subgroup $H_\n$ of cardinality $|H_\n| \geq |A_\n|/\n$.   As usual we form the ultra approximate group $\A := \prod_{\n \to \alpha} A_\n$.  The ultraproduct group $\G := \prod_{\n \to \alpha} G_\n$ also has exponent $m$, and by Hrushovski's Lie model theorem we can find a large approximate group $\A'\subseteq \A^4$ with a local Lie model $\pi : (\A')^8 \rightarrow L$.   By Definition \ref{good-model-def}(i), we may find a neighbourhood $U_0$ of the identity in $L$ such that $\pi^{-1}(U_0) \subseteq \A'$ and $U_0 \subseteq \pi(\A')$.  Using the fact that the exponential map is a homeomorphism near the identity of $L$, we may then find a neighbourhood $U_1$ of the identity with $U_1^m \subseteq U_0$ such that $U_1$ contains no elements of order $m$ other than the identity.  If $a \in \pi^{-1}(U_1)$, then we conclude that $a^m$ is well-defined in $\A'$ with $\pi(a)^m = \pi(a^m) = \id$, and so $\pi(a)$ is trivial, which means that $\pi^{-1}(U_1) = \ker(\pi)$.  As $\pi(\A')$ is precompact, we conclude that $\A'$ is covered by a finite number of translates of $\ker(\pi)$; as $\A'$ is large, $\A$ is also covered by $M$ such translates for some (standard) finite $M$.

From Definition \ref{good-model-def} (iii), we see that the set $\pi^{-1}(U_1) = \ker(\pi)$ is a nonstandard finite set, and so $\ker(\pi) = \prod_{\n \to \alpha} H_\n$ for some finite subsets $H_\n$ of $G_\n$.  Since $\ker(\pi) \subseteq \A^4$ is a group and $\A$ is covered by $M$ translates of $\ker(\pi)$, we see from {\L}os's theorem (Theorem \ref{los-param}) that for all $\n$ sufficiently close to $\alpha$, $H_\n \subseteq A_\n^4$ is a group and $A_\n$ is covered by $M$ translates of $H_\n$. However if one takes $\n$ larger than $M$ then this contradicts the construction of $A_\n$, and the claim follows.
\end{proof}

\begin{remark} The astute reader will notice that the only properties of the local Lie group $L$ that were really used in the above argument were that $L$ was locally compact and had the NSS (no small subgroups property).  Thus, one could prove Theorem \ref{babyhrush} using a weaker form of the Gleason-Yamabe theorem (Theorem \ref{gleason-yamabe}), in which the model group is merely locally compact NSS rather than Lie. (The machinery of Hilbert's fifth problem implies that these two concepts coincide, but this is considerably deeper.)  However, we do not know of a proof of Theorem \ref{babyhrush} that avoids the machinery of Hilbert's fifth problem completely, and in particular some variant of the Gleason lemmas is required.
\end{remark}

Next, we prove (a slight variant of) the main theorem from Hrushovski's paper \cite[Theorem 1.1]{hrush}, which uses the Lie structure (via the Baker-Campbell-Hausdorff formula) more thoroughly than the preceding application.

\begin{theorem}[Hrushovski's structure theorem]\label{hst}  Let $A$ be a $K$-approximate group, and let $F: \N \times \N \to \N$ be a function.  Then there exist natural numbers $L, M, N$ with $N \geq F(L,M)$ and $L, M \ll_{K,F} 1$, and nested sets
$$ \{\id\} \subset A_N \subseteq \ldots \subseteq A_1 \subseteq A^4$$
with the following properties:
\begin{itemize}
\item[(i)] For each $1 \leq n \leq N$, $A_n$ is symmetric;
\item[(ii)] For each $1 \leq n < N$, $A_{n+1}^2 \subseteq A_n$;
\item[(iii)]  For each $1 \leq n \leq N$, $A_n$ is contained in $M$ left-translates of $A_{n+1}$;
\item[(iv)] For $1 \leq n,m,k \leq N$ with $k < n+m$, the set $[A_n,A_m] := \{ [g,h]: g \in A_n, h \in A_m \}$ is contained in $A_k$;
\item[(v)] $A$ can be covered by $L$ left-translates of $A_1$.
\end{itemize}
\end{theorem}

\begin{proof}  Suppose this is not the case.  Carefully negating all the quantifiers, we conclude that there exist $K, F$ and a sequence $A^{(\n)}$ of $K$-approximate groups, such that for each $\n$ and each $L, M \leq \n$, there does not exist $N \geq F(L,M)$ and $A^{(\n)}_1,\ldots,A^{(\n)}_N$ obeying the conclusions of the theorem.

As usual, we form the ultraproduct $\A := \prod_{\n \to \alpha} A^{(\n)}$, which is an ultra approximate group.  By Theorem \ref{lie-model}, we may find a large ultra approximate subgroup $\tilde \A = \prod_{\n \to \alpha} \tilde A^{(\n)}$ which has a good model $\phi: \tilde \A^8 \to L$ by a local Lie group.

Let ${\mathfrak l}$ be the Lie algebra of $L$, and fix an open bounded convex symmetric body $B$ in ${\mathfrak L}$.  Let $\eps > 0$ be a sufficiently small (standard) real number depending on $B$, $L$ to be chosen later; in particular we may assume that the exponential map is a homeomorphism from $\eps B$ to $\exp(\eps B)$, and that $\exp(\eps B)$ is contained in the neighbourhood $U_0$ appearing in Definition \ref{good-model-def}.  For each standard natural number $n \geq 1$, we apply Definition \ref{good-model-def} and Remark \ref{ult} to find an ultra approximate group $\A_n$ with
$$ \pi^{-1}(\exp(10^{-n} \eps B)) \subseteq \A_n \subseteq \pi^{-1}(\exp(2 \times 10^{-n} \eps B));$$
In particular we have the nesting
$$ \ldots \subseteq \A_2 \subseteq \A_1 \subseteq \A^4.$$
From the Baker-Campbell-Hausdorff formula, we have
$$ \exp(2 \times 10^{-n-1} \eps B)^2 \subseteq \exp(10^{-n} \eps B)$$
if $\eps$ is small enough, and thus $\A_{n+1}^2 \subseteq \A_n$.  In a similar spirit, we can find an $M$ depending only on the dimension of $L$ or ${\mathfrak l}$ such that each ball $10^{-n} \eps B$ is covered by at most $M$ translates of $4 \times 10^{-n-1} \eps B$, which by the Baker-Campbell-Hausdorff formula again implies, for small enough $\eps$, that each $\A_n$ is covered by at most $M$ left-translates of $\A_{n+1}$.  Finally, another application of the Baker-Campbell-Hausdorff formula reveals that
$$ [\exp(2\times 10^{-n} \eps B), \exp(2 \times 10^{-m} \eps B)] \subseteq \exp(10^{-k} \eps B) $$
whenever $k < n+m$, and hence $[\A_n,\A_m] \subseteq \A_k$.

Finally, since one can cover $\pi(\A)$ by a finite number of translates of $\exp(\eps B)$, we see that $\A$ can be covered by at most $L$ left-translates of $\A_1$ for some standard $L \in \N$.

Now set $\A_n = \prod_{\n \to \infty} A^{(\n)}_n$ for some finite sets $A^{(\n)}$, and set $N := F(L,M)$.  Applying {\L}os's theorem (Theorem \ref{los-param}) repeatedly (but only finitely many times), we see that for $\n$ sufficiently close to $\alpha$ the sets $A^{(\n)}_1,\ldots,A^{(\n)}_N, A^{(\n)}$ obey all the properties in the conclusion of the theorem. This contradicts the construction of the $A^{(\n)}$ for $\n$ larger than $L,M$, and the claim follows.
\end{proof}

\begin{remark} One can also use the Lie Model Theorem to establish a stronger statement than Theorem \ref{hst}, which roughly speaking asserts that given a (finite) $K$-approximate group $A$, one can find a large sub-approximate group $A'$ which has an approximate homomorphism $\pi: (A')^8 \to L$ into a local Lie group $L$ with bounded range that obeys an approximate version of Property (i) in Definition \ref{good-model-def}, where the accuracy of these approximations exceeds the ``complexity'' of the model\footnote{The complexity, which we do not define here, would be some quantity taking account of the dimension and structure constants on the Lie algebra ${\mathfrak l}$ of $L$, the diameter of the range of $\pi$ and the inradius of the neighbourhood $U_0$ appearing in Definition \ref{good-model-def}(i).} by any given function $F$.  The precise formulation of this statement, which is in fact a logically equivalent ``finitisation'' of Theorem \ref{lie-model}, is somewhat complicated. We will not need it elsewhere in the paper, and so we leave it as an exercise to the reader.
\end{remark}

\section{Strong approximate groups}\label{finite-deductions-sec}

We now give a combinatorial consequence of the Lie Model Theorem (Theorem \ref{lie-model}) which will be important later, involving a concept which we will call a \emph{strong approximate group}.

\begin{definition}[Strong Approximate Group]\label{sag-def}
Let $A$ be a $K$-approximate group for some $K \geq 1$. We say that $A$ is a \emph{strong $K$-approximate group} if it admits a symmetric subset $S$ such that
\begin{equation}\label{san-croot-sis-cond} (S^{A^4})^{1000 K^3} \subseteq A\end{equation}
and for which the following two \emph{trapping conditions} are satisfied:
\begin{enumerate}
\item (First trapping condition) If $g, g^2, g^3,\dots, g^{1000} \in A^{100}$ then $g \in A$;
\item (Second trapping condition) If $g, g^2, \dots, g^{10^6 K^3} \in A$ then $g \in S$.
\end{enumerate}

An \emph{ultra strong approximate group} is an ultraproduct $\A = \prod_{\n \to \infty} A_\n$ of strong $K$-approximate groups $A_\n$, for some $K \geq 1$ independent of $\n$.
\end{definition}

At present this definition will seem somewhat unmotivated, although it can be demystified to some extent by remarking that these properties suggest that $S$ and $A$ are behaving like very small neighbourhoods of the identity in a Lie group $L$, with $S$ much smaller than $A$. This point should become clearer shortly.  The reader should not pay too much attention to exponents such as $1000K^3$ or $10^6 K^3$ in the definition; they are chosen for the sake of concreteness.

The main reason for introducing this concept is that we will be able to show, in the next section, that the escape norm $\Vert g \Vert_{e,\A}$ (defined in Definition \ref{escaped}) for an ultra strong approximate group $\A$ has the pleasant properties outlined in \S \ref{outline-sec}. There is scope for varying the parameters in the definition of strong approximate group, but the ones we have given here are strong enough to prove the desired properties of the escape norm.

It is easy to give examples of strong approximate groups. For instance, if $A = \{-N,\dots,N\}$ (and $K=3$) then we may take $S = \{-N',\dots, N'\}$ with $N' \sim N/1000 K^3$. If $A$ is a subgroup, then we may simply take $S = A$.  On the other hand, if one randomly removes a small number (e.g. $N^{0.01}$) of elements symmetrically from $\{-N,\ldots,N\}$, the resulting set is likely to remain a $O(1)$-approximate subgroup, but not a strong $O(1)$-approximate subgroup.

The main result of this section implies the following.

\begin{proposition}[Finding a ultra strong approximate group]\label{ultra-sag-exist}
Let $\A$ be an ultra approximate group. Then there is a large ultra approximate subgroup $\tilde \A$ of $\A$ which is a strong ultra approximate group.
\end{proposition}

For use in \S \ref{endgame} we will require the following somewhat more precise result.

\begin{proposition}[Balls are ultra strong approximate groups]\label{strong-ultra}  Let $\A$ be an ultra approximate group with a good model $\pi: \A \to L$ to a local Lie group $L$.  Let $B$ be an open bounded convex symmetric subset of the Lie algebra ${\mathfrak l}$ of $L$.  Then there exists a standard radius $r_0>0$ such that for all $0 < r  < r_0$, any symmetric nonstandard finite set $\tilde \A$ with
\begin{equation}\label{a-squash}
 \pi^{-1}(\exp(rB)) \subseteq \tilde \A \subseteq \pi^{-1}(\exp(2rB))
\end{equation}
is a large strong ultra approximate subgroup of $\A$.
\end{proposition}

It is clear that Proposition \ref{ultra-sag-exist} follows from Proposition \ref{strong-ultra}, Theorem \ref{lie-model}, and Definition \ref{good-model-def} (iii); we will, however, only need Proposition \ref{strong-ultra} in the sequel.

We now prove Proposition \ref{strong-ultra}.  Let $r_0>0$ be a sufficiently small quantity depending on $\A, \pi, L, B$ to be chosen later; in particular, we take $r_0$ so small so that the exponential map is a homeomorphism from $2r_0 B$ to $\exp(2r_0 B)$, and $\exp(2r_0 B)^{100}$ is contained inside the open neighbourhood $U_0$ of $L_0$ from Definition \ref{good-model-def}.

Let $\tilde \A$ be as in the proposition.  In particular $\tilde \A^{100} \subset \A$.  By Remark \ref{ult} we may take $\tilde \A$ to be a ultra $K$-approximate group for some\footnote{Indeed, by using the Baker-Campbell-Hausdorff formula one can take $K$ to depend only on the dimension of $L$, if $r_0$ is small enough, but we will not need this fact here.} $K$, and therefore an ultra approximate subgroup of $\A$.  Since $\pi(\A)$ is precompact, it may be covered by finitely many left-translates of $\exp(rB)$, and so $\A$ can be covered by finitely many left-translates of $\tilde \A$.  Thus $\tilde \A$ is a large ultra approximate subgroup of $\A$.

It remains to establish that $\tilde \A$ is a strong ultra approximate subgroup.  Suppose that $g \in \tilde \A^{100}$ is such that $g,\ldots,g^{1000} \in \tilde A^{100}$.  Applying $\pi$, we see that $$\pi(g),\ldots,\pi(g)^{1000} \in \exp( 2r B)^{100}.$$
Working in exponential coordinates and using the Baker-Campbell-Hausdorff formula we conclude, if $r_0$ is small enough, that $\pi(g) \in \exp(r B)$ and thus $g \in \tilde \A$.  We have thus shown the first trapping condition for $\tilde \A$.

Next, we use Definition \ref{good-model-def} to find a symmetric nonstandard finite set ${\bf S}$ with
$$ \exp( 10^{-5} K^{-3} rB ) \subseteq {\bf S} \subseteq \exp( 10^{-4} K^{-3} rB ).$$
From the Baker-Campbell-Hausdorff formula we see that
$$ (\exp(10^{-4} K^{-3} rB)^{\exp(2rB)^4})^{1000K^3} \subseteq \exp(rB)$$
and thus
\begin{equation}\label{sak}
 (({\bf S})^{\tilde \A^4})^{1000K^3} \subseteq \tilde \A.
\end{equation}
Finally, suppose that $g \in \tilde \A$ is such that $g,\ldots,g^{10^6 K^3} \in \tilde \A$.  Applying $\pi$, we conclude that
$$ \pi(g),\ldots,\pi(g)^{10^6 K^3} \in \exp(2rB).$$
Working in exponential coordinates, we conclude that $\pi(g) \in \exp(10^{-5} K^{-3} rB)$ and hence $g \in {\bf S}$.  Thus we have verified the second trapping condition for $\tilde \A$.

Finally, we need to push the trapping conditions from the ultraproduct $\tilde \A$ back to the finitary setting.  Write $\A = \prod_{\n \to \alpha} A_\n$, $\tilde \A = \prod_{\n \to \alpha} \tilde A_\n$ and ${\bf S} = \prod_{\n \to \alpha} S_\n$ for some finite sets $A_\n, \tilde A_\n, S_\n$.  By {\L}os's theorem (Theorem \ref{los-param}), we see that for $\n$ sufficiently close to $\alpha$, $\tilde A_\n$ is symmetric and contains the identity with $\tilde A_\n^4 \subset A_\n^4$, with $\tilde A_\n^2$ covered by $K$ left-translates of $\tilde A_\n$, that
$$
 ((S_\n)^{\tilde A_\n^4})^{1000K^3} \subseteq \tilde A_\n,$$
 and that the first and second trapping properties hold for $\tilde A_\n$ and $S_\n$.  Thus we see that $\tilde A_\n$ is a strong $K$-approximate group for $\n$ sufficiently close to $\alpha$.  After redefining $\tilde A_\n$ suitably for all other values of $\n$, we conclude that $\tilde \A$ is an ultra strong approximate group as required.  This concludes the proof of Proposition \ref{strong-ultra} and hence Proposition \ref{ultra-sag-exist}.

\begin{remark} This proposition represents by far the most serious use of Hrushovski's Lie Model Theorem in our paper. Although we use that theorem elsewhere in the paper, it is only for this proposition that we do not currently have a plausible alternative approach.
\end{remark}

\section{The escape norm and a Gleason type theorem}\label{gleason-sec}

In this section we prove a variant of ``Gleason's lemmas'' in the setting of approximate groups. These show that if $A$ is a strong approximate group then the escape norm has pleasant properties with respect to product, conjugation and commutation. The role of these lemmas was briefly discussed in \S \ref{outline-sec}.

Here is a precise statement.

\begin{theorem}[Gleason-type theorem]\label{gleason-thm}
Suppose that $A$ is a strong $K$-approximate group. Consider the escape norm
\[ \Vert g \Vert_{e,A} := \inf \left\{ \frac{1}{n+1}: n \in \N; g^i \in A \hbox{ for all } 0 \leq i \leq n \right\},\]
with the convention that $\Vert g \Vert_{e,A} = 1$ when $g$ is undefined.
This has the following properties:
\begin{enumerate}
\item \textup{(Conjugation)} If $g, h \in A^{10}$ then $\Vert g^h \Vert_{e,A} \leq 1000 \Vert g \Vert_{e,A}$;
\item \textup{(Product)} We have $\Vert g_1 \dots g_n \Vert_{e,A} \leq K^{O(1)} (\Vert g_1 \Vert_{e,A} + \dots + \Vert g_n \Vert_{e,A})$ if $g_1,\ldots,g_n \in A^{10}$;
\item \textup{(Commutators)} If $g, h \in A^{10}$ then we have $\Vert [g, h] \Vert_{e,A} \leq K^{O(1)} \Vert g \Vert_{e,A} \Vert h \Vert_{e,A}$.
\end{enumerate}
\end{theorem}

Note that, as a consequence of (i) and (ii), the set of $g \in A$ with $\Vert g \Vert_{e,A} = 0$ is a subgroup normalised by $A^{10}$.

\begin{remark}\label{abelian-trivial} Note that this lemma is trivial when the ambient local group is abelian. For that reason, this section can be ignored by those readers interested in seeing our alternative proof of the (abelian) Freiman's theorem.
\end{remark}

\noindent\textsc{Motivation and heuristic discussion.}  We will shortly give a self-contained proof of Theorem \ref{gleason-thm}, but as motivation we first offer some comments and discussion of the context in which these ideas were first invented: the solution of Hilbert's fifth problem by Gleason, Montgomery-Zippin and Yamabe \cite{gleason,gleason2,montgomery,montgomery-zippin-book,yamabe,yamabe2} (see also \cite{goldbring-thesis} for the local group analogue of these lemmas).

In that context, the Gleason lemmas show the existence, in an arbitrary locally compact group $G$, of arbitrarily small compact neighborhoods $A$ of the identity whose associated escape norm satisfies properties (i) to (iii) as above. The Gleason lemmas lie at the heart of Hilbert's fifth problem and are used at several places in its proof, both in the reduction step from general locally compact groups to NSS (No Small Subgroups) groups, and in order to deal with NSS groups.

For example, if $G$ is NSS, the Gleason lemmas are needed in order to establish that the set of one-parameter subgroups of $G$ forms a vector space. If $X(t)$ and $Y(t)$ are two one-parameter subgroups, then a natural candidate for $X+Y$ is $\lim_{n \rightarrow +\infty} (X(t/n)Y(t/n))^{n}$. In order to show that such a limit does exist, the bound (ii) on the escape norm of a product is precisely what is needed. For the full story, the reader may wish to consult the classical references \cite{kaplansky, montgomery-zippin-book}, the more recent non-standard treatments of the Gleason lemmas by Hirschfeld \cite{hirschfeld} and by Goldbring and van den Dries \cite{vdd-goldbring}, or the blog posts of the third author\footnote{{\tt http://terrytao.wordpress.com/tag/hilberts-fifth-problem/}}.

To give a flavour of how the Gleason lemmas are proven, let us discuss a simple case of the product estimate, namely
\begin{equation}\label{product-est}\Vert uv \Vert_{e,A} \leq C(\Vert u \Vert_{e,A} + \Vert v \Vert_{e,A}).\end{equation}
Here, $A$ is a ball $B(\id,1)$ about the identity in a locally compact group $G$ with the NSS property, where the ball is with respect to some left-invariant distance $d$, and $C$ is some finite quantity depending on $A$. In the discussion below we will make use of the following points concerning this situation:

\begin{enumerate}
\item We may construct a distance $d$ with the additional property that $d(\id, x^g) \leq C d(\id, x)$ for $g,x \in B(\id, 2)$ (for example by the Birkhoff-Kakutani construction \cite[\S 1.22]{montgomery-zippin-book}).
\item The balls in $G$ enjoy an escape property quite similar to that in the definition of a strong approximate group. More precisely, given $\eps > 0$ there is an $M_{\eps} \in \N$ such that if $g,g^2,\dots, g^{M_{\eps}} \in B(\id, 1)$ then $g \in B(\id,\eps)$. The proof of this is by contradiction -- taking a limit of putative ``bad'' $g$s, one can contradict the NSS property.
\end{enumerate}

The key idea behind the proof of the product estimate \eqref{product-est} is to relate the escape norm $\Vert g \Vert_{e,A}$ to the auxillary quantity $\Vert \partial_g \Psi \Vert_{\infty}$, where $\partial_g \Psi(x) = \Psi(g^{-1}x) - \Psi(x)$ and $\Psi$ is a non-negative ``bump'' function supported on $B(\id, 1)$, let us say with $\Vert \Psi \Vert_{\infty} = \Psi(\id) = 1$. As noted in Lemma \ref{supp}, such a ``norm'' automatically satisfies the product inequality (with $C = 1$), and so we need only show that $\Vert g \Vert_{e,A} \sim \Vert \partial_g \Psi \Vert_{\infty}$ in a suitable sense, and for a suitable $\Psi$.

In one direction, it is easy to link the two quantities. Indeed if $\Vert \partial_g \Psi \Vert_{\infty} \leq \delta$ for some $\delta>0$, then a simple telescoping sum argument confirms that $\Psi(g^i) > 0$, and hence $g^i \in A$ whenever $i < 1/\delta$. Thererefore

\begin{equation}\label{easy-direction} \Vert g \Vert_{e,A} \leq \Vert \partial_g \Psi \Vert_{\infty}.\end{equation}

Suppose, conversely, that we know that $g, g^2,\dots, g^n \in A = B(\id, 1)$. Then certainly, by the escape property, we have $g, g^2,\dots, g^{n'} \in B(\id,\eps)$ for some $n' \gg_{\eps} n$. Now if $G$ were a Lie group, and if $\Psi$ were smooth with bounded derivatives, we would have

\begin{equation}\label{wish} \partial_{g^{n'}} \Psi \approx n' \partial_g \Psi,\end{equation}

the approximation being better as $\eps$ gets smaller. This immediately gives the bound $\Vert \partial_g \Psi\Vert_{\infty} \lessapprox_{\eps} 1/n$, and thus we have linked the escape norm and the auxiliary norm $\Vert \partial_g \Psi \Vert_{\infty}$ in both directions.

Now unfortunately \eqref{wish} is only an approximate identity and, more seriously, $G$ is not known to be a Lie group. In fact, as noted above, these Gleason lemmas are required to prove statements of that form. On a more positive note, observe that we only need to bound $\Vert \partial_g \Psi \Vert_{\infty}$ above in terms of $\Vert g \Vert_{e,A}$ when $g = u$ or $g = v$, and not for all $g$. We are at liberty to design the auxillary function $\Psi$ with this in mind.

Now the exact version of \eqref{wish} is basically Taylor's formula, and it reads
\begin{equation} \label{taylorexp-old} \partial_{g^n} \Psi = n\partial_g \Psi + \sum_{i = 0}^{n-1} \partial_g \partial_{g^i} \Psi.\end{equation} (We replace $n'$ by $n$ for ease of notation.) This makes it desirable to bound the second derivatives $\partial_g \partial_{g^i} \Psi$. At this point another key idea enters: it is possible to get good control on these second derivatives when $\Psi = \phi \ast \psi$ is the convolution of two ``Lipschitz'' functions, that is
\[ \Psi(x) = \phi \ast \psi (x) = \int \phi(x z^{-1}) \psi(z) \, dz,\] the integral being with respect to Haar measure on $G$. This is because of the formula

\begin{equation}\label{secondd}
\partial_{g}\partial_{h}( \phi * \psi) = \int \partial_g \phi(z) \partial_{h^{z}} \psi(z^{-1}x)dz.
\end{equation}

To make this useful, $\phi$ is chosen to be somewhat Lipschitz with respect to shifts by $g = u$ and $g = v$, and $\psi$ is chosen to be Lipschitz with respect to the distance $d$. We omit the details.\vspace{11pt}

\noindent\textsc{Rigorous argument.}  We turn now to the details of such a strategy in the discrete setting, that is to say a rigorous proof of Theorem \ref{gleason-thm}.

\begin{proof}[Proof of Theorem \ref{gleason-thm}]  To simplify the notation, we will abbreviate $\Vert \Vert_{e,A}$ in this proof as $\Vert \Vert_e$.

We start with (i), which is a relatively easy consequence of the first trapping property in the definition of strong approximate group (Definition \ref{sag-def}). Indeed suppose that $g, g^2,\dots, g^n \in A$ for some $n$; then certainly $g^h, (g^h)^2,\dots, (g^h)^n \in A^h \subseteq A^{12}$. By the first trapping property this implies that $g^h, (g^h)^2,\dots, (g^h)^{n'} \in A$ for any $n' \leq n/1000$, and this confirms (i).

The proof of (ii) is significantly trickier and is based on the construction of Gleason that was briefly discussed earlier. In order to facilitate a certain technical ``bootstrap argument'', it will be convenient to temporarily replace the escape norm $\Vert g \Vert_e$ by the regularised version $\Vert g \Vert^{(\eps)}_e := \Vert g \Vert_e + \eps$, where $\eps > 0$ is a small quantity. We shall obtain estimates uniform in $\eps$, and then let $\eps \rightarrow 0$.

It is natural to introduce the norm-like quantity
\[ d^{(\eps)}(g) :=  \inf \left\{ \sum_{i=1}^n \|g_i\|^{(\eps)}_{e} : g =g_1\ldots g_n, n \geq 1\right\}.\]
It is clear that
\begin{equation}\label{link-1} d^{(\eps)}(g) \leq \Vert g \Vert^{(\eps)}_e.  \end{equation}
We shall prove an estimate in the opposite direction, namely
\begin{equation}\label{product-to-prove} \Vert g \Vert_e \leq K^{O(1)} d^{(\eps)}(g).\end{equation}
The exponent $O(1)$ will be independent of $\eps$. This implies that, for each positive integer $n$ and all $g_1,\dots, g_n$,
\[ \Vert g_1 \dots g_n \Vert_e \leq K^{O(1)} (\Vert g_1 \Vert_e^{(\eps)} + \dots + \Vert g_n \Vert^{(\eps)}_e).\]
Letting $\eps \rightarrow 0$, we recover the product estimate (ii).

In order to establish this we shall, as in Gleason's argument, relate $\Vert g \Vert_e$ and $d^{(\eps)}(g)$ to an auxillary quantity $\Vert \partial_g \Psi \Vert_{\infty}$, where $\Psi : A^{4} \rightarrow [0,\infty)$ is a certain ``smooth'' function supported on $A^4$.
We will specify $\Psi$ shortly; as in Gleason's argument it will be constructed as a convolution of two functions $\phi$ and $\psi$. The former is taken to be a kind of smoothed version of $1_A$ defined using the metric $d^{(\eps)}$ and Lipschitz for this metric, and the latter constructed using the set $S$ appearing in the definition of strong approximate group (Definition \ref{sag-def}) and Lipschitz with respect to the word metric on $S$.

One link between these quantities is relatively easy to establish for any function $\Psi$ with $\Psi(\id) \geq 1$.  Indeed suppose that $\Vert \partial_g \Psi \Vert_{\infty} = \delta$ for some $g \in A^{100}$. Then certainly $|\Psi(g^i) - \Psi(g^{i+1})| \leq \delta$ for all $i$ with $g^i \in A^{100}$, which implies by an easy telescoping sum argument that $\Psi(g^i) \geq 1 - \delta i$ for all $i$. In particular $g^i$ lies in the support of $\Psi$, and hence in $A^4$, for $i < 1/\delta$; note that the hypothesis $g^i \in A^{100}$ can be removed by induction. By the first trapping condition in Definition \ref{sag-def} this implies that $g^i \in A$ for $i < 1/1000\delta$, and hence $\Vert g \Vert_e \leq 1000\delta$. Thus
\begin{equation}\label{link-2}\Vert g \Vert_e \leq 1000 \Vert \partial_g \Psi \Vert_{\infty} \end{equation}
whenever $g \in A^{100}$.

To establish \eqref{product-to-prove} and hence the product estimate (ii) it therefore suffices to prove a bound
\begin{equation}\label{link-3} \Vert \partial_g \Psi \Vert_{\infty} \ll K^{O(1)} d^{(\eps)}(g)\end{equation} in the opposite direction for all $g \in A^{100}$ (the claim for $g \not \in A^{100}$ being an easy consequence). This argument will depend crucially on the specific form of $\Psi$. The following two lemmas describe the construction of the functions $\phi$ and $\psi$.

\begin{lemma}[Properties of $\phi$]\label{phi-properties}
There is a function $\phi^{(\eps)} : A^{1000} \rightarrow [0,1]$ such that
\begin{enumerate}
\item $\phi^{(\eps)}(x) = 1$ for $x \in A$;
\item $\phi^{(\eps)}(x)=0$ if $x \notin A^2$;
\item \textup{(Lipschitz bound)} For all $g \in A^{1000}$, one has
$$\Vert \partial_g \phi^{(\eps)} \Vert_{\infty} \leq \frac{d^{(\eps)}(g)}{d^{(\eps)}(\id,A^c)}.$$
\end{enumerate}
Here $d^{(\eps)}(y, B) := \inf \{ d^{(\eps)}(b^{-1}y) : b \in B\}$, and $A^c$ is the complement of $A$ in $G$.
\end{lemma}
\begin{proof}
Define
\[ \phi^{(\eps)}(x) := \left(1 - \frac{d^{(\eps)}(x,A)}{d^{(\eps)}(\id, A^c)}\right)_+.\]
Note that this is well-defined since $d^{(\eps)}(\id,A^c) \neq 0$; this would be an issue without the fudge factor of $\eps$ that we have introduced.

Obviously $\phi^{(\eps)}(x) = 1$ for $x \in A$. If $\phi^{(\eps)}(x) \neq 0$ then $d^{(\eps)}(\id, x^{-1} A) = d^{(\eps)}(x,A) <  d^{(\eps)}(\id, A^c)$, and so $x^{-1} A$ contains a point outside of $A^c$. This implies that $x \in A^2$.

The Lipschitz bound is easily established.
\end{proof}

\begin{lemma}[Properties of $\psi$]\label{psi-properties}
There is a function $\psi : A^{1000} \rightarrow [0,1]$ such that
\begin{enumerate}
\item $\psi(x) = 1$ for $x \in A$;
\item $\psi(x)=0$ if $x \notin A^2$;
\item \textup{(Lipschitz bound)} $\Vert \partial_{h^y} \psi \Vert_{\infty} \leq 1/10^4K^3$ for $h \in S$ and $y \in A^2$.
\end{enumerate}
\end{lemma}
\begin{proof}
Let $Q := S^{A^2}$; recall from the definition of strong approximate group that $Q^{N} \subseteq A$, where $N := 10^4 K^3$. Define $\psi(g) = 0$ if $g \notin Q^N A$, $\psi(g) = 1$ if $g \in A$ and $\psi(g) = 1 - i/N$ if $g \in Q^{i+1} A\setminus Q^i A$ for $i = 0,1,\dots , N-1$. The claimed properties of $\psi$ are easily checked. \end{proof}

We now define $\Psi$ to be the convolution
$$ \Psi(x) := \frac{1}{|A|} \sum_{y \in A^2} \phi^{(\eps)}(y) \psi(y^{-1} x)$$
for all $x \in A^{100}$, with the convention that $\Psi(x)=0$ for $x$ outside $A^{100}$.

We note that
\[ \Psi(\id) = \frac{1}{|A|}\sum_x \phi^{(\eps)}(x) \psi(x^{-1}) \geq \frac{1}{|A|} \sum_{x \in A} \phi^{(\eps)}(x) \psi(x^{-1}) = 1,\] a property required in the proof of \eqref{link-2}. Note also that since $\phi$ and $\psi$ are both at most $1$ pointwise and are supported on $A^2$ we have, for all $x$ such that $\Psi(x) \neq 0$,
\[ \Psi(x) = \frac{1}{|A|}\sum_{y } \phi^{(\eps)}(y) \psi(y^{-1} x) = \frac{1}{|A|} \sum_{y \in A^4} \phi^{(\eps)}(y) \psi(y^{-1} x) \leq \frac{|A^4|}{|A|} \leq K^3,\] that is to say
\begin{equation}\label{Psi-inf} \Vert \Psi \Vert_{\infty} \leq K^3.\end{equation}

Let $g \in A^{100}$.  Now since $\id \in A$ we have the crude bound $d^{(\eps)}(\id, A^c) \geq \eps$.   It follows from Lemma \ref{phi-properties} that $\Vert \partial_g \phi^{(\eps)} \Vert_{\infty} \leq d^{(\eps)}(g)/\eps$.  From the identity
$$ \partial_g \Psi(x) = \frac{1}{|A|} \sum_{y \in A^{200}} \partial_g \phi^{(\eps)}(y) \psi(y^{-1} x),$$
we have that
\[ |\partial_g \Psi (x)| \leq \Vert \partial_g \phi \Vert_{\infty} \frac{1}{|A|}\sum_{y \in A^4} \psi(y^{-1} x) \leq \frac{K^3}{\eps} d^{(\eps)}(g).\] This immediately yields the crude bound
\begin{equation}\label{apriori} \Vert \partial_g \Psi \Vert_{\infty} \leq \frac{K^3}{\eps} d^{(\eps)}(g)\end{equation}
in the direction of \eqref{link-3}, the statement we are trying to prove.

Denote by $P(X)$ the bound
\begin{equation}\label{PX-bd} \Vert \partial_g \Psi \Vert_{\infty} \leq X d(g)\end{equation} for all $g \in A^{100}$. We have just demonstrated $P(K^3/\eps)$, and we wish to prove $P(K^{O(1)})$, which is \eqref{link-3}. To this end we will implement a bootstrapping argument, showing that $P(X)$ implies a stronger version of itself, namely $P(X')$ with some $X' < X$, under appropriate conditions.

The hypothesis $P(X)$ (cf. \eqref{PX-bd}) implies an improved Lipschitz bound on $\phi$.  To see this note that if $d^{(\eps)}(g) < 1/1000X$ then from assumption $P(X)$ we have $\Vert \partial_g \Psi \Vert_{\infty} < 1/1000$ and hence, from \eqref{link-2}, that $\Vert g \Vert_e < 1$. By definition of the escape norm this implies that $g \in A$. Phrased in the contrapositive, it follows that $d^{(\eps)}(\id, A^c) \geq 1/1000 X$, and therefore the Lipschitz bound in Lemma \ref{phi-properties} implies that
\begin{equation}\label{phi-lip} \Vert \partial_g \phi \Vert_{\infty} \leq 1000X d^{(\eps)}(g).\end{equation}

The bootstrapping argument hinges on the Taylor expansion identity
\[ \partial_{g^n} \Psi = n \partial_g \Psi + \sum_{i=0}^{n-1} \partial_{g^i} \partial_g \Psi,\]
valid whenever $g,\ldots,g^n \in A^{200}$ (say).  This identity implies, using the triangle inequality and \eqref{Psi-inf}, that

\begin{equation}\label{taylorg}\| \partial_g \Psi \|_\infty \leq \frac{1}{n} \| \partial_{g^n} \Psi \|_\infty + \frac{1}{n} \sum_{i=0}^{n-1} \| \partial_{g^i} \partial_g \Psi \|_\infty \leq \frac{2K^3}{n} + \frac{1}{n} \sum_{i=0}^{n-1} \| \partial_{g^i} \partial_g \Psi \|_\infty.
\end{equation}

To use this, we need to focus attention on the first and second derivatives of $\Psi$.
To bound the first derivative we use the identity
\[ \partial_h \Psi(x) = \frac{1}{|A|}\sum_y \phi(y) \partial_{h^y} \psi(y^{-1} x),\]
valid for $h \in A^{100}$.  Since $\phi \leq 1$, this and the Lipschitz bound on $\psi$ given in Lemma \ref{psi-properties}  imply that
\begin{equation}\label{1st-der}
\Vert \partial_h \Psi \Vert_{\infty} \leq \frac{1}{|A|}\sum_{y \in A^2} \Vert \partial_{h^y} \psi \Vert_{\infty} \leq 1/10^4K^2
\end{equation}
if $h \in S$.

We turn to the second derivative $\partial_h \partial_g \Psi$ for $g \in A$ and $h \in S$. Here we use the identity
\[ \partial_h \partial_g \Psi(x) = \frac{1}{|A|} \sum_y (\partial_g \phi)(y) \partial_{h^y} \psi(y^{-1} x).\]
Recalling that $\phi, \psi$ are supported on $A^2$ and using the Lipschitz bound \eqref{phi-lip} on $\phi$ together with the Lipschitz bound on $\psi$ given in Lemma \ref{psi-properties}, we obtain the bound
\begin{equation}\label{2nd-der} \Vert \partial_h \partial_g \Psi \Vert_{\infty} \leq \frac{1}{|A|}\sum_{y \in A^4} \Vert \partial_g \phi \Vert_{\infty} \Vert \partial_{h^y} \psi \Vert_{\infty} \leq \frac{1}{10} X d^{(\eps)}(g)\end{equation} if $g \in A$ and $h \in S$.

These bounds are useful in \eqref{taylorg} provided that $n$ is such that $g , g^2,\dots, g^n \in S$. However, the second trapping property in the definition of strong approximate group ensures that this is so for a reasonably large value of $n$, indeed for $n$ as large as $\frac{1}{10^6 K^3 \Vert g \Vert_e}$. Taking $n$ this large and substituting into \eqref{taylorg} yields

\[ \Vert \partial_g \Psi \Vert_{\infty} \leq  10^7 K^{6} \Vert g \Vert_e + \frac{1}{10}X d^{(\eps)}(g) \leq X' \Vert g \Vert^{(\eps)}_e,\] where $X' = 10^7 K^6 + \frac{1}{10} X$ and $g \in S$.  The claim also trivially holds when $g \not \in S$.

It is easy to improve this to the stronger statement $P(X')$ using the triangle inequality $\Vert \partial_{gh} \Psi \Vert_{\infty} \leq \Vert \partial_g \Psi \Vert_{\infty} + \Vert \partial_h \Psi \Vert_{\infty}$, already observed in \eqref{ghpsi}. Indeed for every $\eta > 0$ there are, by the definition of $d$, $g_1,\dots g_n$ such that $g = g_1 \ldots g_n$ and
\[ d^{(\eps)}(g) > \Vert g_1 \Vert^{(\eps)}_e + \dots + \Vert g_n \Vert^{(\eps)}_e - \eta.\]
Therefore

\[
\Vert \partial_g \Psi \Vert_{\infty}  \leq \Vert \partial_{g_1} \Psi \Vert_{\infty} + \dots + \Vert \partial_{g_n} \Psi \Vert_{\infty}  \leq X' ( \Vert g_1 \Vert^{(\eps)}_e + \dots + \Vert g_n \Vert^{(\eps)}_e) \leq X'(\eta + d^{(\eps)}(g)).\]

Since $\eta$ was arbitrary, we do indeed obtain the bound $\Vert \partial_g \Psi \Vert_{\infty} \leq X' d(g)$, which is the statement $P(X')$.

By repeating this deduction of $P(X')$ from $P(X)$ many times, we see that the crude bound $P(K^3/\eps)$, established in \eqref{apriori}, eventually implies $P(10^9 K^6)$, and hence \eqref{link-3}. By earlier remarks, this concludes the proof of (ii), the inequality for products.

Finally, we turn to the commutator bound (iii).  Now that we have the product inequality (ii), we may define a function $\phi$ obeying the properties in Lemma \ref{phi-properties} but using $\Vert g \Vert_e$ instead of the fudged quantity $\Vert g \Vert^{(\eps)}_e = \Vert g \Vert_e + \eps$, that is to say with
\[ d(g) :=  \inf \{ \sum_{i=1}^n \|g_i\|_{e} ; g =g_1\ldots g_n, n \geq 1\}.\]
This is because (ii) implies the lower bound $d(\id, A^c) \geq K^{-O(1)}$, and in particular $d(\id, A^c) \neq 0$. Moreover we have the Lipschitz bound

\begin{equation}\label{phi-lip-new} \Vert \partial_g\phi \Vert_{\infty} \leq K^{O(1)} d(g).\end{equation}

We will use this function $\phi$ in establishing (iii), the bound for commutators. Once again we consider an auxillary function $\Phi$, defined now to be the convolution
\[ \Phi(x) := \frac{1}{|A|} \sum_{y \in A^2} \phi(y) \phi(y^{-1} x)\]
again with the convention that $\Phi$ vanishes outside of $A^{100}$. We observe the identity
\[ \partial_g \partial_h \Phi - \partial_h \partial_g \Phi = -T_{hg} \partial_{[g,h]}\Phi,\] for $g,h \in A^{10}$, where $T_g$ denotes the shift defined by $T_gf(x) := f(g^{-1} x)$ if $g^{-1} x$ is well-defined, and $0$ otherwise. It follows that
\[ \Vert \partial_{[g,h]} \Phi \Vert_{\infty} \leq \Vert \partial_h \partial_g \Phi \Vert_{\infty} + \Vert \partial_g \partial_h \Phi \Vert_{\infty}.\] By the first bound in \eqref{2nd-der} (which holds equally well for this $\Phi$) we have
\[ \Vert \partial_{[g,h]} \Phi \Vert_{\infty} \leq \frac{2}{|A|} \sum_{y \in A^4} \Vert \partial_g \phi \Vert_{\infty} \Vert \partial_{h^y} \phi \Vert_{\infty}.\]
From \eqref{phi-lip-new}  we obtain
\[ \Vert \partial_{[g,h]} \Phi \Vert_{\infty} \leq K^{O(1)} d(g) \sup_{y \in A^4} d(h^y) \leq K^{O(1)} \Vert g \Vert_e \sup_{y \in A^4} \Vert h^y \Vert_e.\]
By part (i) , this implies
\[ \Vert \partial_{[g,h]} \Phi \Vert_{\infty} \leq K^{O(1)} \Vert g \Vert_e \Vert h \Vert_e.\]
To conclude, we note that \eqref{link-2} holds for this new auxillary function $\Phi$ as well, since the only fact we used in establishing that other than trapping properties of $A$ was the lower bound $\Phi(\id) \geq 1$.

This, at last, concludes the proof of Theorem \ref{gleason-thm}.
\end{proof}

To conclude this section we assemble the main results of it and the previous section in a portable form. The following is the only result we shall need from \S \ref{finite-deductions-sec} and the present section going forward to the next (and final) part of the paper.

\begin{proposition}\label{sec-7-conclusion}
Suppose that $\A$ is an ultra approximate group and that $\pi : \A^8 \rightarrow L$ is a good model for $\A$ into a connected Lie group $L$ with Lie algebra $\mathfrak{l}$. Let $B$ be an arbitrary compact convex neighbourhood of $0$ in $\mathfrak{l}$. Then, for sufficiently small $r,r'$ with $2r> r'>r > 0$, we may find a large strong ultra approximate subgroup $\A'$ of $\A$ such that
\begin{enumerate}
\item $\pi^{-1}(\exp(rB)) \subset \A' \subset \pi^{-1}(\exp(r'B))$;
\item $(\A')^{10^4}$ is well defined;
\item The escape norm $\Vert g \Vert_{e, \A'}$ satisfies
\begin{enumerate}
\item \textup{(Conjugation)} If $g, h \in (\A')^{10}$ then $\Vert h^{-1} g h \Vert_e = O(\Vert g \Vert_e)$;
\item \textup{(Product)} If $n$ is a nonstandard natural number and $g_1,\ldots,g_n \in (\A')^{10}$ is a nonstandard finite sequence of elements of $(\A')^{10}$ \textup{(}i.e. an ultraproduct of standard finite sequences, see Section \ref{nsa-app}\textup{)} then $\Vert g_1 \ldots g_n\Vert_e = O(\sum_{i=1}^n \Vert g_i \Vert_e )$;
\item \textup{(Commutators)} If $g, h \in (\A')^{10}$ then we have $\Vert [g, h] \Vert_e =O(\Vert g \Vert_e  \Vert h \Vert_e)$.
\end{enumerate}
\item The set $H:=\{g \in \A'; \Vert g \Vert_e =0\}$ is a global internal subgroup, that is to say it is of the form $H=\prod_{\n \rightarrow \alpha} H_\n$, where $H_\n \subset A_\n$ contains $\id$ and is stable under multiplication and inverse, which is contained in $\A'$ and is normalised by $\A'$.
\end{enumerate}
\end{proposition}
\begin{proof}
The existence of $\A'$ satisfying (i) and (ii) follows from part (iii) of Definition \ref{good-model-def}. If $r,r'$ are small enough then Proposition \ref{strong-ultra} ensures that $\A'$ is a ultra strong approximate group in the sense of Definition \ref{sag-def}. Properties (iii)(a), (b) and (c) then follow immediately from Theorem \ref{gleason-thm} and taking ultraproducts, and (iv) then follows from (iii).\end{proof}

\begin{remark} Observe that if $\A$ is a strong ultra approximate group, that is to say an ultraproduct of $K$-strong finite approximate groups, and if $L$ is a locally compact model of $\A$ as given for example by Proposition \ref{locally-compact-model}, then from the strong approximate group hypothesis made on $\A$ we see that the standard part of the escape norm $\st(\Vert g \Vert_{e,\A})$ and the escape norm of $\pi(g) \in L$ with respect to the neighborhood of the identity $\pi(\A)$ of $L$ are comparable. Namely $ \Vert \pi(g) \Vert_{e,\pi(\A)} \ll \st(\Vert g \Vert_{e,\A}) \leq \Vert \pi(g) \Vert_{e,\pi(\A)}$. As a consequence, if we take the standard parts of the escape norm in properties (i) to (iii), then what we obtain is precisely the analogous properties for the escape norm in $L$ with respect to $\pi(\A)$.  In that case, the three properties are essentially equivalent to the original Gleason lemmas in the literature on Hilbert's fifth problem, applied to the locally compact (local) group $L$.
In the sequel however, it will be very important that the three bounds (i) to (iii) obtained in Proposition \ref{sec-7-conclusion} hold at the ultra level in $\ultra \R$ and not only at the level of standard parts.
\end{remark}

\section{Proof of the main theorem}\label{endgame}

In this section, we complete the proof of our main theorem, Theorem \ref{hl-conj-nonst}. We will do so by first reducing to the case when $\A$ has no global internal subgroup. For convenience, we introduce the following definition.

\begin{definition}[No small subgroups]  An ultra approximate group $\A$ has the NSS \emph{property} if $\A$ does not contain any non-trivial global internal subgroup.
\end{definition}

By a \emph{global internal subgroup} of $\A=\prod_{\n \rightarrow \alpha}A_{\n}$, we mean a subset of the form $\prod_{\n \rightarrow \alpha} H_\n$, where $H_\n \subseteq A_\n$ is a genuine subgroup. Note that $\A$ is NSS if and only if, for any $g \in \A \backslash \id$, the escape norm $\|g\|_{e,\A}$ is non-zero (though it may be infinitesimal).  We remark that an analogous NSS condition for locally compact groups plays a key role in the theory of Hilbert's fifth problem.

\begin{example}  Let $N\in \ultra \N$ be an unbounded (nonstandard) integer.  Then the interval $\A := [-N,N]$ (in the nonstandard integers $\ultra \Z$) is NSS.  Note that while $\A$ contains global subgroups such as $\Z$ or $\{ x \in \ultra \Z: x = o(N) \}$, such subgroups are not internal (they are not the ultralimits of standard sets).
\end{example}

Clearly, any ultra approximate subgroup of an NSS ultra approximate group is also an NSS ultra approximate group.  Using the Gleason lemmas from Section \ref{gleason-sec} we can reduce the proof of our main theorem to consideration of the NSS case.

\begin{proposition}[NSS reduction]\label{nss-reduct}  Let $\A$ be an ultra approximate group.  Then there exists a large ultra approximate subgroup $\A'$ of $\A$, with $(\A')^{1000}$ well-defined and contained in $\A^4$, and a global internal subgroup $H$ contained in $\A'$ and normalised by $(\A')^{100}$, such that $\A'/H$ is an \textup{NSS} ultra approximate subgroup, which admits a connected Lie group as a good model.
\end{proposition}

We refer the reader to Definition \ref{good-model-def} for the definition of a good model. Here $\A'/H$ denotes the quotient local group as defined in Lemma \ref{quotient}.

\begin{proof} By Proposition \ref{ultra-sag-exist} there is a ultra (strong) approximate group $\A' \subseteq \A^4$ which is large relative to $\A$, for which $(\A')^{10^4}$ is well-defined and  contained in $\A^4$, and a good model $\pi: (\A')^8 \to L$, where $L$ is a connected Lie group. Let $B$ be an open bounded convex symmetric neighbourhood of the identity in the Lie algebra of $L$. Then for sufficiently small $r>0$, $\exp(rB)$ contains no non-trivial subgroups of $L$.

Let $H$ denote the global internal subgroup $H=\{g \in \A'; \Vert g \Vert_e =0\}$ given by Proposition \ref{sec-7-conclusion}. Since $H$ is normalised by $\A'$, it is also normalised by $(\A')^{1000}$. We may then apply Lemma \ref{quotient} and consider the quotient local group $(\A')^{100}/H$. Then $(\A')^8/H=(\A'/H)^8$ is well-defined.  Since $\A', H$ are nonstandard finite symmetric sets, $\A'/H$ is also; since $(\A')^2$ can be covered by finitely many left-translates of $\A$, $(\A'/H)^2$ can be covered by finitely many left-translates of $\A'/H$.  We conclude that $\A'/H$ is an ultra approximate group.

Since $\exp(rB)$ contains no non-trivial subgroups, the image of $H$ under $\pi$ has to be trivial, thus the homomorphism $\pi$ descends to a homomorphism of $\A'/H$ to $L$, which satisfies the conditions for a good model (see Definition \ref{good-model-def}).
By construction, every element $g \in\A'$ that is not in $H$ has positive (but nonstandard) escape norm $\Vert g \Vert_{e,\A'}$. If $g \in \A'$ and $\langle [g] \rangle \subseteq \A'/H$, where $[g]$ is the class of $g$ in $\A'/H$, then $\langle g \rangle \subseteq \A'^2$. On the other hand  $\A'$ is a strong ultra approximate group, and thus $\Vert g \Vert_{e,\A'}$ is non-zero if and only if $\Vert g \Vert_{e,\A'^2}$ is non-zero. This implies that every non-identity element $[g]$ in $\A'/H$ also has positive escape norm $\Vert [g] \Vert_{e,\A'/H}$. Thus $\A'/H$ is NSS and the claim follows.
\end{proof}

Let us now state Theorem \ref{hl-conj-nonst} in the special case of NSS groups, and show how the general case of Theorem \ref{hl-conj-nonst} follows from it.

\begin{theorem}[NSS approximate groups contain large nilprogressions]\label{hl-conj-nonst-nil} Let $\A$ be an \textup{NSS} ultra approximate group which admits a connected Lie group $L$ as a good model.  Then $\A^4$ contains a nondegenerate ultra nilprogression $P$ in normal form, which is large relative to $\A$.  Furthermore, the rank and step of $P$ are no greater than the dimension of $L$.
\end{theorem}

\begin{proof}[Proof that Theorem \ref{hl-conj-nonst-nil} implies Theorem \ref{hl-conj-nonst}] Let $\A$ be an ultra approximate group. We may find a large ultra approximate subgroup $\A'$ of $\A$ which satisfies the conclusions of Proposition \ref{nss-reduct}. We may then apply Theorem \ref{hl-conj-nonst-nil} to $\A'/H$ and find in $(\A')^4/H$ a nondegenerate ultra nilprogression $P_0$ in normal form with $|P_0| \gg |\A/H|$. We can write $P_0=P_s(\overline{u_1},\ldots,\overline{u_r};N_1,\ldots,N_r)$, where $N_i \in \ultra \N$ are unbounded and $\overline{u_i}\in \A'/H$. We may then pick arbitrary lifts $u_i\in \A'$ and set $P=P(u_1,\ldots,u_r;N_1,\ldots,N_r)$. Then $HP$ is a nondegenerate ultra coset progression in normal form contained in $(\A')^5 \subseteq \A^4$, and $|HP|\geq |H||P_0| \gg |\A|$ as desired.
\end{proof}
We turn now to the proof of Theorem \ref{hl-conj-nonst-nil}, which will occupy the remainder of this section. We begin with a brief sketch, fleshing out a little more the overview given in \S \ref{outline-sec}. The proof will proceed by induction on the dimension of the connected Lie group $L$. The base case of the induction, when $\dim L=0$, is trivial as in this case the NSS ultra approximate group $\A$ is also trivial. To treat the induction step we will consider an element $u$ of $\A$ with smallest possible escape norm. The existence of such an element is guaranteed by our standing hypothesis that approximate groups are finite objects, i.e. that each $A_\n$ in $\A=\prod_{\n  \rightarrow \alpha}A_\n$ is finite. Then we will mod out $\A$ by the geometric $P:=\{u^n, |n|\leq 1/\Vert u\Vert_{e,\A}\}$, where $u$ is an element of $\A$ which the smallest possible escape norm $\Vert u \Vert_{e,\A}$. The quotient local group $\A/P$ (in the sense of Lemma \ref{quotient}) will be shown to be both NSS and to admit a Lie group with dimension at most $\dim L -1 $ as a good model.  It is at this step that we crucially rely on the fact that we are only quotienting out by a \emph{local} group, the progression $P$, rather than a \emph{global} one such as the group $\langle u \rangle$ generated by $u$. We do this in order to avoid accidentally creating torsion with an excessively large quotient.  Indeed, it is because of this component of the induction that it was necessary to cast the entire argument in the setting of local groups rather than global groups, even if one had been willing to restrict the main results of the paper to the global group case. Finally, making key use of the properties of the escape norm given by the Gleason lemmas, we will lift the nilprogression from $\A/P$ to $\A$.

Let us turn to the details.\vspace{11pt}

\noindent \emph{Proof of Theorem \ref{hl-conj-nonst-nil}.} Let $\A$ be an NSS ultra approximate group which admits a connected Lie group $L$ as a good model $\pi:\A^8 \rightarrow L$. We proceed by induction on $\dim L$ and first dispose of the trivial case when $L$ has dimension zero.  As $L$ is connected, it must thus be trivial.  Applying Definition \ref{good-model-def}(iii), we conclude that $\A^8$ is a large global internal subgroup of $\A$.  Since $\A$ is NSS, this kernel must therefore be trivial. Therefore $\A$ is trivial.

Now suppose that $\dim L \geq 1$, and that the claim has already been proven for connected Lie groups of smaller dimension. To complete the proof of Theorem \ref{hl-conj-nonst-nil} it suffices to establish the following lemma.

\begin{lemma}[Induction step]\label{induction} Suppose that $\A$ is an ultra approximate group admitting a connected Lie group $L$ of positive dimension as a good model. Then $\A$ contains large ultra approximate subgroups $\A''' \subseteq \A'' \subseteq \A' \subseteq \A$ with the following properties. Let $u \in \A'$ be such that $\Vert u \Vert_{e,\A'}$ is minimal and non zero, and set $P:=\{u^n : |n|< 1/\Vert u \Vert_{e,\A'}\}$. Then $P$ commutes with $(\A'')^{10}$ and obeys the following properties:
 \begin{enumerate}
 \item the quotient $\A'''/P$ is an ultra approximate group which admits a connected Lie group of dimension $\dim L -1$ as a good model, whose Lie algebra is formed from the Lie algebra of $L$ by quotienting out by a one-dimensional central subalgebra;
 \item if $\A$ is \textup{NSS}, so is $\A'''/P$;
 \item to any large ultra nilprogression $\overline{Q}$ in $\A''/P$ in normal form, one can associate a large ultra nilprogression $Q$ in $\A''$ in normal form, whose rank exceeds the rank of $\overline{Q}$ by at most one, and similarly for the step; and
 \item $(\A''')^4 \subseteq \A''$.
 \end{enumerate}
\end{lemma}

\emph{Proof of Theorem \ref{hl-conj-nonst-nil}.} Indeed  apply the induction hypothesis to $\A'''/P$, which we can do by (i) and (ii).  We may then conclude, using (iv), that $\A''/P$ contains a large ultra nilprogression. Finally, apply (iii) to conclude.\endproof

\begin{proof}[Proof of Lemma \ref{induction}]
Take $B$ to be some small convex neighbourhood of $0$ in the Lie algebra $\mathfrak{l}$ of $L$. We shall take $\A', \A'', \A'''$ to be such that
\begin{equation}\label{nesting-0}\pi^{-1}(\exp(B)) \subseteq \A' \subseteq \pi^{-1}(\exp(1.001B))\end{equation} and
\begin{equation}\label{nesting-1}\pi^{-1}(\exp(\delta B)) \subseteq \A'' \subseteq \pi^{-1}(\exp(1.001\delta B)) \end{equation} and
\begin{equation}\label{nesting-2} \pi^{-1}(\exp(\frac{\delta}{10} B)) \subseteq \A''' \subseteq \pi^{-1}(\exp (1.001 \frac{\delta}{10} B)) .\end{equation}
where $\delta > 0$ is a small (standard) real number to be specified later.

It follows from Proposition \ref{sec-7-conclusion} that large ultra approximate subgroups of $\A$ exist with these properties, and furthermore that, if $B$ is small enough, the escape norm $\Vert \cdot \Vert_{e,\A'}$ satisfies the conjugation, product and commutator inequalities laid out in (iii) of that proposition. Note also that $\A', \A''$ admit $L$ as a good model and have the NSS property.

Property (iv) of the lemma is essentially immediate; we turn to the more substantial (i), (ii) and (iii).

We begin with the proof of (i). Recall that $u \in \A'$ is chosen so that $\Vert u \Vert_{e,\A'}$ is minimal and nonzero.  Observing that $\Vert x \Vert_{e, \A'} \leq 100 \delta$ for $x \in (\A'')^{10}$, it follows from the commutator estimate of Proposition \ref{sec-7-conclusion} (iii)(b) that for such $x$ we have
\[ \Vert [x,u] \Vert_{e,\A'} = O(\Vert x \Vert_{e, \A'} \Vert u \Vert_{e,\A'}) < \Vert u \Vert_{e,\A'}\] provided that $\delta$ is chosen sufficiently small in terms of the implied constant $O( \cdot)$.

Note that $[x,u]$ lies in $\A'$, rather than merely $(\A')^4$, since its escape norm is less than $1$. From the extremal property of $u$, it follows that $[x,u] = \id$, that is to say $x$ commutes with $u$, whenever $x \in (\A'')^{10}$.

Recall that we are taking $P := \{ u^n : n \leq 1/\Vert u \Vert_{e,\A'}\}$.

Since $(\A'')^6$ is well defined, we may apply Lemma \ref{quotient} and form the quotient local group $\A''/P\simeq\prod_{\n \rightarrow \alpha} A''_\n/P_\n$, which is clearly also an ultra approximate group. We now show that $\A''/P$ admits a proper quotient of $L$ as a good model. To do this, we first verify that $\pi(P)$ is a non-trivial central one-parameter local subgroup of $L$.

Since $\dim L \geq 1$, the groups $\A', \A''$ are non trivial, and this implies that $\Vert u \Vert_{e,\A'}$ is infinitesimal, i.e. that $M'_0:=1/\Vert u \Vert_{e,\A'}$ is unbounded. Let $n \in \ultra \N$ be such that $n=o(M'_0)$. We must have $u^n \in \ker \pi$, because $u^{kn} \in \A'$ for all $k \in \N$. Define a map $\phi: [-100,100] \to L$ by setting
\begin{equation}\label{phito}
\phi( t ) := \pi( u^{\lfloor t M'_0 \rfloor} ),
\end{equation}
 where $\lfloor \cdot  \rfloor$ is the (nonstandard) greatest integer function. Then $\pi(u^n) = \phi( \st( n / M'_0 ) )$
for all $n \in [-100 M'_0, 100 M'_0]$, and $\phi$ is a local homomorphism in the sense that $\phi(t) \phi(s) = \phi(t+s)$ whenever $t, s, t+s \in [-100,100]$. Also $ \pi(P) = \phi([-1,1])$. Finally, we verify that $\phi$ is continuous. Because of the local homomorphism property, it is enough to check this at $0$. If $t$ is small, then $(\phi(t))^k=\phi(tk) \in \exp(1.001B)$ for every integer $k \in [0,\frac{1}{t}]$, hence $\phi(t) \in \exp(1.001B/k)$ is close to the identity in $L$, which gives the desired continuity.

As $\phi$ is a continuous homomorphism from $[-1,1]$ to the Lie group $L$, there exists an element $X$ of the Lie algebra $\mathfrak{l}$ such that $\phi(t) = \exp(tX)$ for all $t \in [-1,1]$.  Moreover $X \in 1.001 B$. On the other hand by the definition of the escape norm we have $u^{M'_0} \notin \A'$, and hence $\phi(1)=\exp(X) \notin \exp(B)$ and thus $X \notin B$. In particular $X$ is non-zero, and it follows that $\phi([-1,1])$ is a non-trivial local one-parameter subgroup of $L$. Finally it is central in $L$, because $L$ is connected and $\phi(t)=\pi(u^{\lfloor t M'_0 \rfloor})$ commutes with the neighbourhood of identity $\pi(\A'')$ as shown above. Thus $X$ lies in the centre of the Lie algebra $\mathfrak{l}$.

If we choose a neighbourhood $U$ of the identity in $L$ small enough,  then by Lemma \ref{quotient} we may form the quotient space $U/\phi([-1,1])$, which one easily verifies to be a local Lie group of dimension $\dim L - 1$, whose Lie algebra is obtained from the Lie algebra of $L$ by quotienting out by a one-dimensional central subalgebra.  By Lie's third theorem every local Lie group is locally identifiable with an open neighbourhood of a global connected Lie group $L'$, which in our case still has dimension $\dim L - 1$.  Thus, by shrinking $U$ if necessary, we may find a local homomorphism $\eta: U \to L'$ whose kernel lies in $\phi([-1,1])$.  The local homomorphism $\eta \circ \pi: (\A'')^8 \to L'$ then pushes down to a local homomorphism $\psi: (\A'')^8/P \to L'$. Choosing $\delta$ smaller if necessary, we may assume that $\psi$ is defined on all of $(\A''/P)^8$, thus making $L'$ a good model for $\A''/P$. Note we may also ensure that $\pi(\A'')$ contains no non-trivial subgroup of $L'$, a property that will be needed in Lemma \ref{lift} below. This completes the proof of (i).

We turn now to (ii), which asserted that $\A'''/P$ is NSS. In fact we shall prove the same statement for $\A''/P$, from which the statement for $\A'''/P$ follows (or note that an identical proof works). Key to this endeavour is the following lifting lemma, which we will require again in the proof of (iii).

\begin{lemma}[Lifting lemma]\label{lift}  Let $g \in \A''/P$, and let $\kappa: (\A'')^8 \to (\A'')^8/P$ be the projection map.  Then there exists $\tilde g \in \A''$ such that $\kappa(\tilde g) = g$ and $\| \tilde g \|_{e,\A''} =O(\| g \|_{e,\A''/P})$.
\end{lemma}
Let us first remark on why the NSS property of $\A''/P$ follows quickly from this. Indeed suppose that $g \in \A''/P$ is not the identity. Then the element $\tilde g$ generated by the above lemma is not the identity either, and hence has positive escape norm since $\A''$ is NSS.
By the lemma, $g$ also has positive escape norm. Since $g \neq \id$ was arbitrary, this establishes the NSS property for $\A''/P$.

\begin{proof}[Proof of Lemma \ref{lift}] Fix $g \in \A''/P$.
Let $\tilde g$ be a lift of $g$ in $\A''$ which minimizes the escape norm $\Vert \tilde g \Vert_{e,\A''}$ among all possible lifts of $g$. If $\tilde g$ is trivial, then so is $g$ and there is nothing to prove. Therefore we may assume that $\tilde g$ is not the identity and hence, since $\A''$ is NSS, that it has positive escape norm. Suppose, by way of contradiction, that $\Vert g \Vert_{e,\A''/P} = o(\Vert \tilde g \Vert_{e,\A''})$. Our goal will be to reach a contradiction by finding another lift of $g$ with strictly smaller escape norm than $\tilde g$.

Set $M''_1:=1/\Vert \tilde g \Vert_{e,\A''} \in \ultra \N$.

We now make an important deduction from our hypothesis. For every $n \in \ultra \N$ such that $n=O(M''_1)$, we have $g^n \in \A''/P$. In particular, for every (standard) integer $k \in \N$, $g^{kM''_1} \in \A''/P$. This implies that the group generated by $g^{M''_1}$ lies in $\A''/P$. However, in projection to the Lie model, $\A''/P$ gets mapped into a neighbourhood of the identity in $L'$, which we chose small enough so as not to contain any non-trivial subgroup. We thus conclude that $g^{M''_1}$ maps to the identity in $L'$, and therefore $\tilde g^{M''_1}$ maps into the local one-parameter subgroup $\phi([-1,1])$.

Now there is another element which maps to $\phi([-1,1])$, namely $u$, the element for which $\Vert u \Vert_{e,\A'}$ is minimal.

In order to motivate the rest of the argument, let us temporarily work in a heuristic setting (using informal notation such as $\approx$), returning to tighten the argument rigorously later. Since
\begin{equation}\label{approx-inclusion}\A'' \approx \pi^{-1}(\exp(\delta B)),\end{equation} and since $M''_1$ is the least $n$ for which $\tilde g^n$ escapes $\A''$, we have

\begin{equation}\label{M1-rough} \pi (\tilde g^{M''_1}) \approx \phi (\delta) = \exp(\delta X).\end{equation}
Similarly
\begin{equation}\label{M0-rough} \pi( u^{M''_0}) \approx \phi(\delta).\end{equation}

Now $u^n$ takes at least as long as $\tilde g^n$ to escape from $\A' \approx \pi^{-1}(\exp B)$. Hence (roughly) it takes as least as long to escape from $\A'' \approx \pi^{-1}(\exp \delta B)$ as well, which means that
\begin{equation}\label{mimo} M''_1 \lessapprox M''_0. \end{equation}

We are trying to find a lift of $g$ with smaller escape norm than that of $\tilde g$. To do this it is sensible to look for elements of the form $h := \tilde g u^{-m}$, $m \in \ultra \N$. Provided that $m$ is chosen judiciously, $h$ will also be a lift of $g$ since (by definition) $u$ lies in $P$. Since, measured in $\phi([-1,1])$ by applying $\pi$, the element $u$ is ``shorter'' than $\tilde g$, it seems reasonable that by an appropriate choice of $m$ we can make $h$ shorter than $\tilde g$ as well.

Being a little more precise, suppose that $m,n \in \ultra \N$. Since $u$ is central in $\A''$ we have $h^n = \tilde g^n u^{-mn}$ whenever these expressions are well-defined, and hence $\pi( h^n) = \pi(\tilde g^n) \pi(u^{-mn})$. From \eqref{M1-rough} and \eqref{M0-rough} we have

\begin{equation}\label{comps} \pi(\tilde g^n) \approx \phi(\st(\delta n/M''_1)) \qquad \mbox{and} \qquad \pi(\tilde u^{-mn}) \approx \phi(\st (-\delta mn/M''_0)).\end{equation}
These expressions will be legitimate if $m,n$ are chosen so that the arguments of the $\phi$'s always lie in $[-50,50]$ (say).
It follows that \begin{equation}\label{computation} \pi(h^n) \approx \phi (\st (\frac{1}{M''_1} - \frac{m}{M''_0} ) \delta n).\end{equation}
However (by the Euclidean algorithm) there is a choice of $m \in \ultra \N$ such that $|1/M''_1 - m/M''_0| \leq 1/2M''_0$.  Comparing with \eqref{computation} we see that for $n = 1,\dots, 2M''_0$ we have $\pi(h^n) = \phi(\delta')$ with $\delta' \leq \delta$. Since $\pi^{-1}(\phi([0,\delta])) \subseteq \A''$, we must raise $h$ to at least the power $2M''_0$ before it escapes $\A''$. Since $2M''_0 > M''_0 \gtrapprox M''_1$, this $h$ is a  lift of $g$ with smaller escape norm than $\tilde g$. Note that the computations \eqref{computation} are legitimate for this choice of $m$ and for $n \leq 2M''_0$.

We now perform the above argument rigorously. Instead of the heuristic statement \eqref{approx-inclusion}, we must work with the inclusions
\begin{equation}\label{nesting} \pi^{-1} (\exp(\delta B)) \subseteq \A'' \subseteq \pi^{-1} (\exp (1.001 \delta B)).\end{equation}

To get a precise form of \eqref{M1-rough}, note that by definition of the escape norm we have $\tilde g^{M''_1 - 1} \in \A''$, whilst $\tilde g^{M''_1} \notin \A''$. In particular, as a consequence of \eqref{nesting}, $\pi (\tilde g^{M''_1 - 1}) \in \exp(1.001 \delta B)$, whilst $\pi (\tilde g^{M''_1}) \notin \exp(\delta B)$. Since $M''_1$ is unbounded, the first of these actually implies that $\pi (\tilde g^{M''_1}) \in \exp(1.001 \delta B)$.

Similarly $\pi (u^{M''_0-1}) \in \exp(1.01 \delta B)$, whilst $\pi (u^{M''_0}) \notin \exp(\delta B)$.  Once again, the first of these implies that $\pi (u^{M''_0}) \in \exp(1.001 \delta B)$.

Since $B$ is convex, comparison of these facts shows that $\pi(\tilde g^{M''_1}) = \phi(t)$ and $\pi(u^{M''_0}) = \phi(t')$ with
\begin{equation}\label{t-r} t,t' \in [0.9 \delta, 1.1 \delta].\end{equation}
Suppose that $M'_1$ and $M'_0$ are the escape times of $\tilde g$ and $u$ from $\A'$, respectively. Since $u \in \A'$ was assumed to have minimal escape norm, $M'_1 \leq M'_0$. On the other hand \eqref{nesting} implies that $M''_0/M'_0, M''_1/M'_1 \in [0.99\delta, 1.01\delta]$, and so
\begin{equation}\label{mimo-precise} M''_1 \leq 1.1 M''_0.\end{equation}
As in the heuristic discussion above, take $h := \tilde g u^{-m}$, for some $m \in \ultra N$. Let $n \in \ultra \N$. Then we have
\[ \pi (\tilde g^n) = \phi (\st (tn/M''_1)) \qquad \mbox{and} \qquad \pi(u^{-mn}) = \phi(-\st (t' mn/M''_0))\]
provided that the arguments of the $\phi$'s are in $[-50,50]$, which will always be the case later on in the argument.
Since $u$ is central we have
\begin{equation}\label{h-n} \pi(h^n) = \phi \left(\st \left(\frac{t}{\delta M''_1} - \frac{mt'}{\delta  M''_0}\right) \delta n\right).\end{equation}
Roughly as before, we use the Euclidean algorithm to find $m \in \ultra \N$ such that $|1/M''_1 - mt'/\delta M''_0| \leq t'/2\delta M''_0$. By \eqref{t-r} and \eqref{mimo-precise} it follows that
\[ \left|\frac{t}{\delta M''_1} - \frac{mt'}{\delta M''_0}\right| \leq \frac{t'}{2\delta M''_0} + (1 - \frac{t}{\delta}) \frac{1}{M''_1} < \frac{0.9}{M''_1}. \]
It follows from this and \eqref{h-n} that $\pi(h^n) \in \phi([0,\delta])$ for $n \leq M''_1$, and hence $h^n$ lies in $\A''$ for these same values of $n$. As a consequence, $h$ has smaller $\Vert \cdot \Vert_{e,\A''}$ escape norm than $\tilde g$, contrary to assumption.
\end{proof}

Finally we prove item (iii) of Lemma \ref{induction}. Suppose then that $\overline{Q}$ is a nondegenerate large ultra nilprogression in $\A''/P$ in normal form; we wish to lift this to a large ultra nilprogression $Q$ in $\A''$ of at most one higher rank and step, while preserving the nondegeneracy and normal form properties.  The main difficulty is that if one lifts the generators of $\overline{Q}$ arbitrarily then there is no guarantee that the progression they generate, or even a significant part of it, will be contained in $\A''$. The key to ensuring that we do achieve this lies in making judicious use of the lifting lemma (Lemma \ref{lift}) and the product and commutator properties of the escape norm (Proposition \ref{sec-7-conclusion} (iii) (b) and (c)). At this point we advise the reader to quickly review Definition \ref{normal-def} and Appendix \ref{nilprogression-sec}, where nilprogressions in $C$-normal form are discussed.

We may write the non-degenerate ultra nilprogression in normal form as
\[ \overline{Q}=P(\overline{u}_1,\ldots,\overline{u}_r;N_1,\ldots,N_r),\]
where the $\overline{u}_i$ are in $\A''/P$, the $N_i \in \ultra \N$ are unbounded, and $r$ is the rank of $\overline{Q}$, and some standard step $s$.
From the normal form hypothesis (and taking ultraproducts), we have the following properties:
\begin{enumerate}
\item (Upper-triangular form) For every $1 \leq i < j \leq r$ and $\epsilon_i, \epsilon_j \in \{-1,+1\}$, one has
\begin{equation}\label{comm-2}
[u_i^{\epsilon_i},u_j^{\epsilon_j}] \in P\left( u_{j+1}, \ldots, u_r; O(\frac{N_{j+1}}{N_i N_j}), \ldots, O(\frac{N_r}{N_i N_j}) \right).
\end{equation}
\item (Local properness) The expressions $u_1^{n_1} \ldots u_r^{n_r}$ for nonstandard integers $n_1,\ldots,n_r$ with $|n_i| \leq \frac{1}{C} N_i$ for all $1 \leq i \leq r$ are all well-defined and distinct, if $C$ is a sufficiently large standard real.
\item (Volume bound)  One has
$$ N_1 \ldots N_r \ll |Q| \ll N_1 \ldots N_r.$$
(Note that as the $N_i$ are unbounded, $2\lfloor N_i\rfloor+1$ and $N_i$ are comparable.)
\end{enumerate}
Also, since $u_i^{n_i} \in \overline{Q} \subseteq \A''/P$ for all $1 \leq i \leq r$ and $|n_i| \leq N_i$, we have
$$ \|u_i\|_{e,\A''/P} \leq \frac{1}{N_i}.$$

By Lemma \ref{lift}, we may find lifts $u_i \in \A''$ which project to $\overline{u}_i$ in the quotient local group $\A''/P$, and are such that
\[ \Vert u_i\Vert_{e,\A''} = O( \Vert\overline{u}_i\Vert_{e,\A''/P} )\]
and thus
\begin{equation}\label{uup}
 \Vert u_i\Vert_{e,\A''} \ll \frac{1}{N_i}.
\end{equation}

In order to include $P$ in the lifted progression, we set $u_{r+1}:=u$, the generator of $P$, and $N_{r+1}:=1/\Vert u\Vert_{e,\A''}$.   From \eqref{phito} we see that
\begin{equation}\label{momo}
 M'_0 \ll N_{r+1} \ll M'_0.
\end{equation}
We then define
$$ Q := P( u_1,\ldots,u_r, u_{r+1}; \eps N_1,\ldots,\eps N_{r+1} )$$
for some sufficiently small standard $\eps>0$.  We claim that $Q$ is well-defined in $\A''$ as a nondegenerate ultra nilprogression in normal form, of rank $(r+1)$ and step at most $s+1$.

We begin with the claim that $Q$ is well-defined in $\A''$.  From \eqref{uup} and Proposition \ref{sec-7-conclusion} one has
$$ \Vert g \Vert_{e,\A''} \ll \eps$$
for all $g \in Q$, and in particular every product in $Q$ lies in $\A''$ as required.

It is clear that $Q$ is a nondegenerate ultra noncommutative progression of rank $(r+1)$.  To show that it is a nilprogression of step at most $s+1$, it suffices to show that any iterated commutator $g$ of length $s+2$ in the generators $u_1^{\pm 1},\ldots,u_{r+1}^{\pm 1}$ is trivial.  Using commutator identities such as the Hall-Witt identity
\begin{equation}\label{hall-witt}
[z, [x,y]] = [[y^{-1},z^{-1}],x]^{zy} [[z,x^{-1}],y^{-1}]^{xy}
\end{equation}
where $x^y := y^{-1} xy$ (using the unbounded nature of the $N_i$ to justify all operations) we may restrict attention to iterated commutators $g$ of the form $g = [h,u_i^{\pm 1}]$ where $h$ is an iterated commutator of length $s+1$ and $1 \leq i \leq r+1$.  But by projecting down to $\A''/P$, we know that the image of $h$ vanishes and thus $h \in P$.  Since $P$ is central in $\A''$, the claim follows.

Finally, we need to show that $Q$ is in normal form.  We begin by establishing the upper triangular form \eqref{comm}, i.e. that
$$
[u_i^{\epsilon_i},u_j^{\epsilon_j}] \in P\left( u_{j+1}, \ldots, u_r; O(\frac{N_{j+1}}{N_i N_j}), \ldots, O(\frac{N_r}{N_i N_j}) \right)$$
whenever $1\leq i < j \leq r+1$ and $\epsilon_i, \epsilon_j \in \{-1,+1\}$.

If $j=r+1$, then $u_j = u$ commutes with every element of $\A''$, and in particular with $u_i$, so the claim follows in this case.  Now suppose that $j \leq r$.  From \eqref{comm-2} we then have
$$
[\overline{u_i}^{\epsilon_i},\overline{u_j}^{\epsilon_j}] \in P\left( \overline{u_{j+1}}, \ldots, \overline{u_r}; O(\frac{N_{j+1}}{N_i N_j}), \ldots, O(\frac{N_r}{N_i N_j}) \right) $$
which lifts to
$$
[u_i^{\epsilon_i},u_j^{\epsilon_j}] \in P\left( u_{j+1}, \ldots, u_r; O(\frac{N_{j+1}}{N_i N_j}), \ldots, O(\frac{N_r}{N_i N_j}) \right) \cdot P.$$
Thus we may write $[u_i^{\epsilon_i},u_j^{\epsilon_j}] = g u^n$, where
$$ g \in P\left( u_{j+1}, \ldots, u_r; O(\frac{N_{j+1}}{N_i N_j}), \ldots, O(\frac{N_r}{N_i N_j}) \right) $$
and $n \leq 1 / \|u\|_{e,\A'} = M'_0$.  From \eqref{uup} and Proposition \ref{sec-7-conclusion} one has
$$ \| [u_i^{\epsilon_i},u_j^{\epsilon_j}]\|_{e,\A''} \ll \frac{1}{N_i N_j},$$
and therefore
$$ \| g \|_{e,\A''} \ll \frac{1}{N_i N_j}$$
and hence
\begin{equation}\label{une}
 \| u^n \|_{e,\A''} \ll \frac{1}{N_i N_j}.
 \end{equation}
In particular, $\| u^n \|_{e,\A''}$ is infinitesimal, which implies that $\pi(u^n) = \id$ and hence $n = o(M'_0)$ by \eqref{phito}.  Since $\|u\|_{e,\A''} = 1/N_{r+1}$, we conclude that
$$ |n| \ll \frac{N_{r+1}}{N_i N_j}$$
and thus
$$
[u_i^{\epsilon_i},u_j^{\epsilon_j}] \in P\left( u_{j+1}, \ldots, u_{r+1}; O(\frac{N_{j+1}}{N_i N_j}), \ldots, O(\frac{N_{r+1}}{N_i N_j}) \right).$$
Noting that $\eps>0$ is standard and thus can be absorbed into the $O()$ notation, this gives the desired upper triangular property.

Next, we establish the local properness.  Suppose that
$$ u_1^{n_1} \ldots u_{r+1}^{n_{r+1}} = u_1^{n'_1} \ldots u_{r+1}^{n'_{r+1}}$$
for some $|n_i|, |n'_i| \leq \eps N_i$.  Quotienting by $P$, we conclude that
$$ \overline{u_1}^{n_1} \ldots \overline{u_r}^{n_r} = \overline{u_1}^{n'_1} \ldots \overline{u_r}^{n'_r}.$$
This quotienting can be justified because all products here lie in $Q$ and hence in $\A''$.  By the local properness of $\overline{Q}$, we conclude if $\eps$ is small enough that $n_i = n'_i$ for all $1 \leq i \leq r$; we may then cancel and conclude that
$$ u^{n_{r+1}-n'_{r+1}} = \id.$$
Since $\|u\|_{e,\A''} = 1/N_{r+1}$ and $|n_{r+1}-n'_{r+1}| < N_{r+1}$, this implies that $n_{r+1}=n'_{r+1}$, giving the desired local properness.

From local properness one immediately has the lower bound
$$ |Q| \gg N_1 \ldots N_r N_{r+1}.$$
Now we establish the matching upper bound
$$ |Q| \ll N_1 \ldots N_r N_{r+1}.$$
We first recall from the normal form of $\overline{Q}$ that
$$ |\overline{Q}| \ll N_1 \ldots N_r.$$
From construction it is also clear that the image of $Q$ under projection by $P$ lies in $\overline{Q}$.  It therefore suffices to show that the preimage of any element in $\overline{Q}$ contains at most $O(N_r)$ elements of $Q$.  By construction of the quotient map, we see that the preimage is contained in a translate of $P$, and thus has cardinality $O(M'_0)$; the  claim then follows from \eqref{momo}.
This concludes the proof of Lemma \ref{induction} and thus Theorem \ref{hl-conj-nonst-nil}.
\end{proof}
We can now conclude the proof of Theorem \ref{main-theorem}, the most basic form of our main theorem.

\begin{proof}[Proof of the first part of Theorem \ref{main-theorem}] We argue by contradiction. Negating the quantifiers, we see that there exists some $K \geq 1$ and an infinite sequence of local groups $G_\n$ and finite $K$-approximate groups $A_\n \subseteq G_\n$, $\n \in \N$, for which the conclusion of the theorem fails, namely for which $A_\n^4$ does not contain any coset nilprogression of rank and step at most $\n$ in $\n$-normal form and of cardinality at least $\frac{1}{\n} |A_\n|$.

Now form the ultraproduct $\A=\prod_{\n \rightarrow \alpha} A_\n$ inside $\G=\prod_{\n \rightarrow \alpha} G_\n$. By {\L}os's Theorem (Theorem \ref{los-param}), $\G$ is a local group and $\A$ an ultra approximate subgroup. We can now apply Theorem \ref{hl-conj-nonst}, whose proof we just completed, to conclude that $\A^4$ contains an ultra coset nilprogression $P$ in normal form with $|P|\gg |A|$. Using {\L}os's theorem again we conclude that $P=\prod_{\n \rightarrow \alpha} P_\n$, where for an $\alpha$-large set of $\n$, $P_\n$ is a $1/c$-proper coset nilprogression contained $A_\n^4$ of rank and step at most $1/c$ and of size at least $c|A_\n|$ for some standard positive number $c>0$.   But this contradicts the construction of the $A_\n$, thereby yielding the claim.
\end{proof}

To conclude this section we record another useful conclusion from the above analysis: Hrushovski's Lie model is nilpotent.

\begin{proposition}[Nonstandard finite approximate groups have nilpotent Lie models]\label{nilmodel} Suppose that $\A$ is an ultra approximate group and that $\pi : \A^8 \rightarrow L$ is a good model for $\A$ into a connected Lie group $L$ with Lie algebra $\mathfrak{l}$. Then ${\mathfrak l}$ and $L$ are nilpotent.
\end{proposition}

\begin{proof}  By Proposition \ref{sec-7-conclusion} we may find a large strong ultra approximate subgroup $\A'$ of $\A$ obeying the conclusion of that proposition.  By quotienting out the elements $H$ of $\A'$ of zero escape norm as in the proof of Proposition \ref{nss-reduct}, we obtain an NSS ultra approximate group $\A'/H$.  Now one runs the argument in Theorem \ref{hl-conj-nonst-nil}.  An inspection of this argument shows that if one unfolds the induction from Lemma \ref{induction}, the Lie algebra ${\mathfrak l}$ of $L$ is repeatedly quotiented out by central algebras until it becomes trivial.  Thus, ${\mathfrak l}$ can be obtained from the trivial Lie algebra by a finite tower of central extensions and is therefore nilpotent as required. The nilpotence of $L$ is an immediate consequence of this and basic Lie theory.
\end{proof}

\section{A dimension bound}\label{rank-reduction}

In this section we prove Theorem \ref{dimension-bound}, in which it is shown that the rank of the nilprogression $P$ in the main theorem may be taken to be $O(K^{O(1)})$. We will also show that so long as we work in a global group $G$, and replace $A^4$ with $A^{O_K(1)}$, it may be taken to be $O(\log K)$.  By the usual ultraproduct argument, it will suffice to establish the following nonstandard analysis formulation of the theorem.

\begin{theorem}\label{hl-conj-nonst-last}
Suppose that $\A$ is an ultra global $K$-approximate group, thus $\A = \prod_{\n \to \alpha} A_n$ for some finite $K$-approximate groups, each contained in a global group $G_n$. Then $\A^{12}$ contains an ultra coset nilprogression $P$ in normal form with $|P|\gg |A|$ and rank at most $O(K^2 \log K)$. Moreover, there exists a standard natural number $m$ such that $\A^m$ contains an ultra coset nilprogression $P$ in normal form with $|P|\gg |\A|$ and rank at most $6 \log_2 K$.
\end{theorem}

Recall that the step of a nilprogression in normal form is always less of equal to its rank. The derivation of Theorem \ref{dimension-bound} from Theorem \ref{hl-conj-nonst-last} proceeds analogously to the derivation of Theorem \ref{main-theorem} from Theorem \ref{hl-conj-nonst} and is omitted.

It remains to establish Theorem \ref{hl-conj-nonst-last}.  The arguments here are inspired by some remarks of Hrushovski in \cite[\S 4]{hrush}.  In particular, a key tool will be the following lemma from \cite[Lemma 4.9]{hrush}.

\begin{lemma}[Doubling in a simply connected nilpotent Lie group]\label{dub}  Let $G$ be a connected, simply connected nilpotent Lie group of dimension $d$, and let $A$ be a measurable subset of $G$.  Let $\mu$ be a Haar measure on $G$ \textup{(}note that nilpotent groups are automatically unimodular, and so there is no distinction between left and right Haar measure\textup{)}.  Then $\mu(A^2) \geq 2^d \mu(A)$.
\end{lemma}

\begin{proof} We use an argument of Gelander from \cite[Lemma 4.9]{hrush}.  As is well known (e.g. see \cite{corwin}), in a simply connected nilpotent Lie group, the exponential map $\exp: {\mathfrak g} \to G$ is a diffeomorphism, which pushes forward the Lebesgue measure $\mu_{\mathfrak g}$ on the $d$-dimensional vector space ${\mathfrak g}$, the Lie algebra of $G$, to the Haar measure $\mu$ on $G$.  Thus it will suffice to show that $\mu_{{\mathfrak g}}(\log(A^2)) \geq 2^d \mu_{{\mathfrak g}}(\log A)$, where $\log$ is the inverse of $\exp$.  But as $A^2$ contains $\{a^2: a \in A\}$, $\log(A^2)$ contains the dilate $2 \cdot \log A$ of $\log A$, and the claim follows.
\end{proof}

One is tempted to combine this theorem with the Hrushovski Lie Model Theorem directly (i.e. Theorem \ref{lie-model}), to get some dimensional control on the Lie group $L$.  However, there is a technical obstruction; the Lie model is only available for an ultra approximate subgroup $\A'$ of $\A$, and the covering parameter $K'$ of this subgroup $\A'$ may be much worse than the covering parameter $K$ of the original ultra approximate subgroup $\A$.  To get around this problem, we need to choose the subgroup $\A'$ more carefully.  A clue as to how to proceed is provided by the following basic observation (cf. \cite[Lemma 7.3]{helfgott-sl3}).

\begin{lemma}[Slicing approximate groups by genuine subgroups]\label{slice}  Let $A$ be a \textup{(}possibly infinite\textup{)} $K$-approximate group in a global group $G$, and let $G'$ be a genuine subgroup of $G$.  Then $A' := A^2 \cap G'$ is a $K^3$-approximate subgroup and $A^4 \cap G'$ can be covered by at most $K^3$ left translates of $A'$.
\end{lemma}

\begin{proof}  Since $(A')^2 \subseteq A^4 \cap G'$, it suffices to show that $A^4 \cap G'$ can be covered by $K^3$ left-translates of $A'$.  But $A^4$, can be covered by $K^3$ left-translates of $A$ since $A$ is a $K$-approximate group.  Next, observe that if a left-translate $gA$ of $A$ intersects $A^4 \cap G'$ in at least one point $g'$, then
$$ gA \cap A^4 \cap G' \subseteq gA \cap G' \subseteq g' A'.$$
Thus $A^4 \cap G'$ can be covered by $K^3$ left-translates of $A'$, as required.
\end{proof}

Lemma \ref{slice} suggests that we should look for Lie models of approximate groups $\A'$ that are formed by slicing $\A^2$ with a genuine subgroup $G'$ of $G$.

We turn to the details.  Let $A_\n$ be a sequence of $K$-approximate groups in global groups $G_n$, and let $\A = \prod_{\n \to \alpha} A_n$ be their ultraproduct; thus $\A$ is a ultra $K$-approximate group that lies inside an ultra genuine group $\prod_{\n \to \alpha} G_n$.  By Proposition \ref{locally-compact-model-global}, we may find a model $\pi: \langle\A\rangle \to G$ of $\A^4$ by a locally compact group $G$.

Let $U_0$ be the neighbourhood in Definition \ref{good-model-def}. We have $\pi^{-1}(U_0) \subseteq \A^4$ and $U_0 \subseteq \pi(\A^4)$. By Theorem \ref{gleason-yamabe}, there is an open subgroup $G'$ of $G$ and a closed subgroup $H$ of $G$ contained in $U_0$ and normalized by $G'$ such that $L := G'/H$ is a connected Lie group. Let $U_1 \subseteq G'$ be an open subset such that $H \subseteq U_1 \subseteq U_1^2 \subseteq U_0$ and let $\phi: G' \to L$ denote the quotient map.

Now set $\A' := \A^4 \cap \pi^{-1}(G')$.  From Lemma \ref{slice} applied to $\A^2$, we see that $\A'$ is a $K^6$-approximate group.  We now also claim that $\A'$ is a nonstandard finite set, which would make $\A'$ an ultra $K^6$-approximate group.  To see this, observe first from Definition \ref{good-model-def}(ii) that $\pi(\A^4)$ is contained in some compact set $F$.  As $G'$ is an open subgroup of $G$, it is also closed and $F \cap G'$ is compact. We then see from Definition \ref{good-model-def}(iii) that we can find a nonstandard finite set $\A_*$ such that $\pi^{-1}(F \cap G') \subseteq \A_* \subseteq \pi^{-1}(G')$.  Thus $\A' = \A^4 \cap \A_*$, and so $\A$ is a nonstandard finite set as required.

Note that $\pi(\A')$ contains the open set $U_0 \cap G'$ and is itself contained in a compact subset of $G'$. Hence the set $E := \phi \circ \pi(\A')$ is precompact and contains a neighbourhood of the identity in $L$. Moreover $(\phi \circ \pi)^{-1}(\phi(U_1)) \subseteq \A'$, hence it follows that $\phi \circ \pi : \langle \A' \rangle \rightarrow L$ is a good model for $\A'$. From Lemma \ref{nilmodel}, we conclude that $L$ is nilpotent. Every connected nilpotent Lie group admits a unique maximal compact subgroup which, moreover, is central. Let $N$ be the maximal compact subgroup of $L$ and $\theta: L \to L/N$ be the quotient map.

We claim that $\dim(L/N) \leq 6\log_2 K$. To see this note that, as $\A'$ is a $K^6$-approximate group, we see that $E^2$ is covered by at most $K^6$ left-translates of $E$. Therefore $\theta(E)^2$ can be covered by at most $K^6$ left-translates of $\theta(E)$, and hence $\overline{\theta(E)}^2$ can be covered by at most $K^6$ translates of $\overline{\theta(E)}$, where $\overline{\theta(E)}$ is the topological closure of $\theta(E)$, a compact set with non-empty interior.  If we let $\mu$ be a Haar measure on $L/N$, it follows that
$$ \mu( \overline{\theta(E)}^2 ) \leq K^6 \mu( \overline{\theta(E)} ).$$
On the other hand, from Lemma \ref{dub} one has
$$ \mu( \overline{\theta(E)}^2 ) \geq 2^{\dim(L/N)} \mu( \overline{\theta(E)} ).$$
Since $\overline{\theta(E)}$ has non-empty interior, $\mu( \overline{\theta(E)} ) \neq 0$ and so comparison of these two inequalities implies that
\begin{equation}\label{L/N-bound}
 \dim(L/N) \leq 6 \log_2 K.
\end{equation}

We now explain how to derive the second part of Theorem \ref{hl-conj-nonst-last} from the above; we will turn to the first part later. We consider $\phi^{-1}(N)$, which is the kernel of the projection map from $G'$ to $L/N$ ($= (G'/H)/N$).  This is a compact subgroup of $G'$. Since $\pi(\A')$ contains an open neighbourhood of the identity, we conclude that there exists a standard natural number $m$ such that $\phi^{-1}(N) \subseteq \pi(\A'^{m-1})$, and this implies that $\A'^m$ contains the kernel of $\theta \circ \phi \circ \pi$, which implies that $\theta \circ \phi \circ \pi: \A'^{8m} \to L/N$ is a good model.  By Proposition \ref{nss-reduct} we conclude that $\A'^{4m}$ contains a large approximate subgroup $\A''$ with $(\A'')^{1000}$ well-defined and contained in $\A'^{4m}$, and a global internal subgroup $H'$ of $\A''$ such that $\A''/H'$ is an NSS approximate subgroup with a connected Lie group as a good model. An inspection of the proof of that proposition reveals that we may take $(\A'')^{1000}$ inside $\A'^{m}$ and that we may take the connected Lie group to be $L/N$, which we have shown to have dimension at most $6 \log_2 K$.  Applying Theorem \ref{hl-conj-nonst-nil}, we see that $(\A''/H')^4$ contains a large ultra nilprogression in normal form of rank at most $6 \log_2 K$, and thus $(\A'')^4$ contains a large ultra coset nilprogression in normal form of rank at most $6 \log_2 K$.  As $(\A'')^4$ is contained in $\A'^{m}$, which is in turn contained in $\A^{4m}$, the second part of Theorem \ref{hl-conj-nonst-last} follows (after redefining $m$).

We now turn to the first part of Theorem \ref{hl-conj-nonst-last}. As we see from the last paragraph, the difficulty here is that $\phi^{-1}(N)$ may not be contained in $\pi(\A')$. We will show that nevertheless $\pi(\A')$ still contains a subgroup $\phi^{-1}(N_0)$, where $N_0$ is a closed subgroup of $N$ with small codimension. For this the key is the following lemma, which is potentially of interest its own right. Here, and below, we write $\T^d := \R^d/\Z^d$ for the $d$-dimensional torus. By a \emph{subtorus} we mean a closed connected subgroup of $\T^d$.

\begin{lemma}\label{subtorus} Let $K,d \geq 1$ and $A$ be a closed $K$-approximate group in $\T^d$ containing a neighbourhood of $0$.  That is, $A$ is closed, contains a neighbourhood of $0$, is centrally symmetric, and there is a finite set $X \subseteq \T^d$, $|X| \leq K$, such that $A + A \subseteq A + X$. Then $4A:=A+A+A+A$ contains a subtorus $T \subseteq \T^d$ of codimension at most $O(K^2 \log K)$.
\end{lemma}

Before proving Lemma \ref{subtorus}, we explain how to conclude the proof of Theorem \ref{hl-conj-nonst-last} with Lemma \ref{subtorus} in hand. First we observe that setting $\A'_1:=\A^2 \cap \pi^{-1}(G')$, $\pi(\A'_1)$ is a neighbourhood of $\id$ in $G$. Indeed $\A^4$ is and $\A^4 \subseteq X\A$ for some finite $X$, so that $\pi(\A)$ has non-empty interior, and hence $\pi(\A^2)$ is a neighborhood of $\id$. We now apply the lemma to $A=\phi \circ \pi (\A'_1)$, and conclude that $\phi^{-1}(T) \subseteq \pi({\A'_1}^4) \subseteq \pi(\A'^2)$. Writing $\theta': L \rightarrow L/T$ for the projection map, we see that $\A'^3$ contains the kernel of $\theta' \circ \phi \circ \pi$ and this implies that $\A'^3$ admits the connected nilpotent Lie group $L/T$ as a good model. Moreover $\dim L/T = \dim L/\T^d + \dim \T^d/T =O(K^2 \log K)$ by $(\ref{L/N-bound})$ and by Lemma \ref{subtorus}. The rest of the proof is then identical to the previous case: by Proposition \ref{nss-reduct} and its proof we conclude that $\A'^3$ contains a large approximate subgroup $\A''$ with $(\A'')^{1000}$ well-defined and contained in $\A'^3$, and a global internal subgroup $H'$ of $\A''$ such that $\A''/H'$ is an NSS ultra approximate subgroup admitting $L/N_0$ as a good model. By Theorem \ref{hl-conj-nonst-nil}, we see that $(\A''/H')^4$ contains a large ultra nilprogression in normal form of rank at most $O(K^2 \log K)$, and thus $(\A'')^4$ contains a large ultra coset nilprogression in normal form with rank at most $O(K^2 \log K)$.  As $(\A'')^4 \subseteq \A'^3 \subseteq \A^{12}$, the first part of Theorem \ref{hl-conj-nonst-last} follows.

We now turn to the proof of Lemma \ref{subtorus}. Let $\mu$ be the normalized Haar measure on $\T^d$. Note that the group $\hat{\T}^d$ of characters of $\T^d$ identifies with $\Z^d$. Our main tool is the notion of the ($\alpha$-)large spectrum of an additive set $A$, defined by
\[ \Spec_{\alpha}(A) := \{ \xi \in \Z^d \; : \; |\widehat{1_A}(\xi)| \geq \alpha \mu(A)\}.\]
See \cite[Definition 4.34]{tv-book} for this definition and a further discussion.

If $S \subseteq \Z^d$ is a set of characters, we write
\[ S^{\perp} := \bigcap_{\xi \in S} \ker \xi.\]
Note that $S^\perp$ is a closed subgroup of $\T^d$ with codimension the rank of the subgroup of $\Z^d$ generated by $S$.

To prove Lemma \ref{subtorus} we first reduce to the case in which $\mu(A)$ is somewhat large by establishing the following lemma.

\begin{lemma}\label{reduction} Suppose that $A \subseteq \T^d$ is a closed $K$-approximate group containing a neighbourhood of $0$.
Then there is a subtorus $T_0 \subseteq \T^d$ with $\dim(T_0) \geq d - O(\log K)$ and some $x_0 \in \T^d$ such that, writing $\mu_0$ for the Haar measure on $T_0$, we have $\mu_0((A + x_0) \cap T_0) \gg e^{-O(K \log K)}$.
\end{lemma}

Then we handle the case in which $\mu(A)$ is somewhat large by proving the following, which is a straightforward continuous analogue of the so-called Bogolyubov-Chang lemma \cite{chang-freiman}.

\begin{lemma}\label{large-alpha} Suppose that $A \subseteq \T^d$ is measurable, that $\mu(A) \geq \alpha$, and that $\mu(2A) \leq K \mu(A)$. Then $2A - 2A$ contains a subtorus $T \subseteq \T^d$ of codimension at most $O(K \log(1/\alpha))$. \end{lemma}

To deduce Lemma \ref{subtorus} from Lemmas \ref{reduction} and \ref{large-alpha}, we proceed as follows. Locate an $x_0$ as in Lemma \ref{reduction}, and suppose furthermore that for this $x_0$ the measure of $(A + x_0) \cap T_0$ is close to maximal in the sense that
\begin{equation}\label{maximal-close} \mu_0 ((A + x_0) \cap T_0)  \geq \frac{1}{2} \mu_0 ((A + x) \cap T_0)
\end{equation}
for all $x \in \T^d$.  Set $A_1 := (A + x_0) \cap T_0$. Then , since $A + A \subseteq A + X$, we have
\[ A_1 + A_1 \subseteq (A + X + 2x_0) \cap T_0.\]
By \eqref{maximal-close} it follows that $\mu(2A_1) \leq 2K \mu(A_1)$. We are now in a position to apply Lemma \ref{large-alpha} to $A_1$, with $\alpha = e^{-O(K \log K)}$. We conclude that there is a further subtorus $T \subseteq T_0$ of codimension $O(K^2 \log K)$ inside $2A_1 - 2A_1$. Since $2A_1 - 2A_1 \subseteq 4A$, this concludes the proof of Lemma \ref{subtorus}.

For the proofs of both Lemmas \ref{reduction} and \ref{large-alpha} we will require the following lemma of Bogolyubov type.

\begin{lemma}[Bogolyubov-type lemma]\label{bog-lemma}
Let $A \subseteq \T^d$ have positive measure and let $k \geq 2$ be a natural number. Suppose that
\[ \delta \leq \big( \mu(A)/2\mu(kA)  \big)^{1/(2k-2)}.\]
Then $kA - kA$ contains $(\Spec_{\delta}(A))^{\perp}$.
\end{lemma}
\begin{proof}
It suffices (in fact, it is equivalent) to show that if $x \in (\Spec_{\delta}(A))^{\perp}$ then $f(x)>0$, where $f=1_A^{\ast k} \ast 1_{-A}^{\ast k} = 1_A \ast \dots \ast 1_A \ast 1_{-A} \ast \dots 1_{-A}$ is the convolution of $k$ copies of $1_A$ and $k$ copies of $1_{-A}$.  Now by the Fourier inversion formula we have
\begin{align}\nonumber
f(x)  =
\sum_{\xi \in \Z^d} |\widehat{1_A}(\xi)|^{2k} \overline{\xi(x)}  & \geq \sum_{\xi \in \Spec_{\delta}(A)} |\widehat{1_A}(\xi)|^{2k} -\sum_{\xi \notin \Spec_{\delta}(A)} |\widehat{1_A}(\xi)|^{2k} \\ & \geq \sum_{\xi \in \Z^d} |\widehat{1_A}(\xi)|^{2k} - 2\sum_{\xi \notin \Spec_{\delta}(A)} |\widehat{1_A}(\xi)|^{2k},\label{fourier}
\end{align}
where we have used the fact that $\xi(x)=1$ if $\xi \in \Spec_\delta(A)$ and $x \in (\Spec_{\delta}(A))^{\perp}$.
Now Parseval's identity and the Cauchy-Schwarz inequality imply that
\[ \sum_{\xi \in \Z^d} |\widehat{1_A}(\xi)|^{2k}=\int_{x \in \T^d} 1_A^{*k}(x)^2 d\mu(x) \geq \frac{1}{\mu(kA)}\big( \int_{x \in \T^d} 1_A^{*k}(x) d\mu(x) \big)^2= \frac{\mu(A)^{2k}}{\mu(kA)}.\]
On the other hand, by a second application of Parseval's identity, we have
\[ \sum_{\xi \notin \Spec_{\delta}(A)} |\widehat{1_A}(\xi)|^{2k} < \delta^{2k-2}\mu(A)^{2k-2} \sum_{\xi \in \Z^d} |\widehat{1_A}(\xi)|^2 = \delta^{2k-2} \mu(A)^{2k-1}.\]
Substituting these inequalities into \eqref{fourier} yields\[ f(x) \geq \frac{\mu(A)^{2k}}{\mu(kA)} - 2 \delta^{2k-2} \mu(A)^{2k-1}=\frac{\mu(A)^{2k-1}}{\mu(kA)} \big(\mu(A)-2 \delta^{2k-2}\mu(kA)).\]
The lemma follows immediately.\end{proof}

Lemma \ref{large-alpha} is an immediate consequence of the case $k = 2$ of this lemma and (the continuous variant of) ``Chang's lemma'' \cite{chang-freiman}, which is the following statement.  For a proof, see \cite[Lemma 4.36]{tv-book}.

\begin{lemma}[Chang's lemma]\label{chang-lemma}
Suppose that $\alpha < 1/2$ and that $A \subseteq \T^d$ is a measurable set with $\mu(A) \geq \alpha$. Then $\Spec_{\delta}(A)$ generates a subgroup of $\Z^d$ of rank at most $O(\delta^{-2} \log (1/\alpha))$.
\end{lemma}

\begin{proof}[Proof of Lemma \ref{large-alpha}] Noting that $\mu(A)/\mu(2A) \leq 1/K$, the lemma follows from Lemma \ref{bog-lemma} with $k = 2$ and $\delta := 1/2\sqrt{K}$ followed by an application of Lemma \ref{chang-lemma}.\end{proof}

To prove Lemma \ref{reduction} we will apply Lemma \ref{bog-lemma} with a much larger value of $k$, as well as the following result.
\begin{lemma}\label{san-lemma} There is an absolute constant $c > 0$ with the following property.
Suppose that $A \subseteq \T^d$ is a closed $K$-approximate group containing a neighbourhood of $0$.  Then $\Spec_{1 - c/\log K}(A)$ generates a subgroup of $\Z^d$ of rank $O(\log K)$.
\end{lemma}
\begin{proof}
Let $\eps = c/\log K$, where $c > 0$ is to be chosen later. Suppose that $\xi \in \Spec_{1 - \eps}(A)$ and that $\xi \neq 0$. Let $\eta: \T^d \mapsto \R/\Z$ be such that $\xi(x)=e^{2i\pi \eta(x)}$. Then
\[ \int_{\T^d} 1_A(x) \xi(x) \, d\mu(x) \] is real, since $A$ is symmetric, and at least $(1 - \eps)\mu(A)$. Thus
\[ \int_{\T^d} 1_A(x) \cos ( 2 \pi \eta(x)) \, d\mu(x) \geq (1 - \eps)\mu(A),\] and hence the (symmetric) subset $A'(\xi) \subseteq A$, consisting of those $x$ for which $\cos (2\pi \eta(x)) \geq 99/100$, has measure  $\mu(A'(\xi)) \geq (1 - 100 \eps)\mu(A)$. In particular $\Vert \eta (x) \Vert < \frac{1}{10}$ whenever $x \in A'(\xi)$, where $\Vert \theta \Vert:=\inf_{z \in \Z}|\theta - z|$.

Suppose now that $\Spec_{1 - \eps}(A)$ contains elements $\xi_1,\dots,\xi_m$ which are linearly independent over $\Q$, and let $\eta_1,\ldots,\eta_m$ be the corresponding $\R/\Z$-valued characters. Consider the set $A'' := \bigcap_{i=1}^m A'(\xi_i)$; provided that $m < 1/100\eps$, this will have $\mu(A'') \geq \frac{1}{2}\mu(A)$. Note that $A''=-A''$.

Consider now the homomorphism $\psi : \T^d \rightarrow \T^m$ given by $x \mapsto (\eta_1(x), \ldots , \eta_m(x))$. The image of $A''$ under $\psi$ lies in a box of diameter $1/5$.

Now any subset $U$ of $\T^m=(\R/\Z)^m$ which lies in a box of diameter $<\frac{1}{2}$ is Freiman $2$-isomorphic to an open subset of $\R^m$, and thus by the abelian case of Gelander's Lemma \ref{dub} (which, in this case, is just a very simple case of the Brunn-Minkowski inequality), we have

\[ \mu_m(2U) \geq 2^m \mu_m(U),\]
where $\mu_m$ is the normalized Haar measure on $\T^m=(\R/\Z)^m$. However, we have $\mu(5A'') \leq \mu(5A) \leq K^4 \mu(A) \leq 2K^4 \mu(A'')$, and an application of the Ruzsa covering lemma (here our Lemma \ref{isgroup}) shows that $2A''$ is a $4K^4$-approximate group, and consequently so is $U:=\psi(2A'')$.

Therefore, noting that $\mu_m(U) \neq 0$ since $\mu(A'')\geq \frac{1}{2}\mu(A)>0$, we obtain
\[ 2^m \leq (4K^4)^2 = (2K)^8,\] a contradiction if $m > 8 \log_2 2K$. Such a choice of $m$ is acceptable if $\eps < c/\log K$ with $c$ sufficiently small, and so we are forced to conclude that $\xi_1,\dots,\xi_m$ cannot exist. The lemma follows.
\end{proof}

\begin{proof}[Proof of Lemma \ref{reduction}] Note that $kA\subseteq (k-1)X + A$, and that $|mX| \leq m^{|X|} = m^K$ for all natural numbers $m$. It follows that $\mu(kA) \leq k^K \mu(A)$, and so Lemma \ref{bog-lemma} is applicable with $\delta = 1 - c/\log K$ for some $k \leq K^C$. The conclusion is that $2kA$ contains the subtorus $T_0 := (\Spec_{1 - c/\log K}(A))^{\perp}$ which, by Lemma \ref{san-lemma}, has codimension $O(\log K)$. However $2k A$ is covered by at most $(2k)^K = e^{O(K \log K)}$ translates of $A$, and so one of these translates has $\mu_0(A + x_0) \gg e^{-O(K \log K)}$, which was precisely what we claimed.\end{proof}

To conclude this section we record the observation that the above arguments also yield the following more precise version of Proposition \ref{weak-global}, the weak global Lie model theorem. This builds upon a previous result in this direction by Hrushovski: see \cite[Theorem 4.2]{hrush} and the discussion before \cite[Lemma 4.9]{hrush}.

\begin{theorem}[Strong global Lie Model Theorem]\label{strong-global}
Suppose that $\A$ is a global ultra $K$-approximate group.  Then there is a large ultra approximate subgroup $\tilde \A$ of $\A^m$ for some standard $m \geq 1$ which admits a global model $\tilde\pi : \langle \tilde \A\rangle \rightarrow L$ into a connected, simply connected nilpotent Lie group $L$ of dimension at most $6 \log_2 K$.

Furthermore, there exists a large ultra approximate group $\A'$ of $\A$ which admits a global model $\pi': \langle \A' \rangle \to L'$, a connected nilpotent Lie group, whose maximal \textup{(}central\textup{)} compact subgroup $N$ verifies $L'/N \simeq L$.
\end{theorem}

\section{Applications to growth in groups and geometry}\label{gromov-sec}

In this section we collect a variety of applications of our main results, in particular proving the various results stated in the introduction.

As an application of his method Hrushovski \cite{hrush} established the following strengthening of Gromov's theorem on groups with polynomial growth.

\begin{theorem}\label{hrush-grom} Let $G$ be a finitely generated group and let $K \geq 1$. Suppose $G = \bigcup_{n \geq 1} A_n$, where $A_n$ is an increasing union of finite subsets of $G$ such that $|A^2_n| \leq K|A_n|$ for all $n \geq 1$. Then $G$ is virtually nilpotent.
\end{theorem}

This is indeed a strengthening of Gromov's theorem because if $G$ has polynomial growth with respect to some generating set $S$ then the $A_n$ may be taken to be some subsequence of the word metric balls relative to $S$.

Unsurprisingly, our main theorem also admits an application of this kind. The following is a corollary of Theorem \ref{main-theorem} and subsumes Theorem \ref{hrush-grom} above.

\begin{corollary}[Gromov-type theorem]\label{main-cor}Let $K \geq 1$. Then there is some $K'$, depending on $K$, such that the following holds. Assume $G$ is a group generated by a finite symmetric set $S$ containing the identity. Let $A$ be a finite subset of $G$ such that $|A^2|\leq K|A|$ and $S^{K'} \subseteq A$. Then there is a finite normal subgroup $N \lhd G$ and a subgroup $G_1 \leq G$ containing $N$ such that
\begin{enumerate}
\item $G_1$ has index $O_K(1)$ in $G$;
\item $G_1/N$ has step and rank $O_K(1)$.
\end{enumerate}
In particular $G$ is virtually nilpotent.
\end{corollary}

\begin{proof}[Proof of Corollary \ref{main-cor}] First we make the following simple observation. Suppose $G$ is a group generated by a finite symmetric set $S$ and let $G_0$ be a subgroup of index $n=[G:G_0]$. Then for every $k<n$ the ball $S^k$ meets at least $k+1$ different left cosets of $G_0$ in $G$. Indeed if not then by the pigeonhole principle we have $S^i G_0 = S^{i+1} G_0$ for some $i < k$, and so by multiplying on the left with $S$ it follows that $S^k G_0 = S^{k+1} G_0$. Multiplying on the left  by further copies of $S$ implies that $S^k G_0 = \langle S \rangle G_0 = G$, and so $G_0$ has index at most $k$ in $G$, contrary to assumption.

Now, we apply Corollary \ref{hl-conj-dub}.  Thus there exists a subgroup $G_0$ of $G$ and a normal subgroup $H$ of $G_0$ such that $A$  may be covered by $K'$ left-translates $G_0$ for some $K' = O_K(1)$ depending only on $K'$, and $G_0/H$ is nilpotent of step and rank $O_K(1)$.  In particular, $G_0$ is finite-by-nilpotent.

Using this value of $K'$, we see by assumption that $S^{K'}$ is contained in $A$ and thus $S^{K'}$ is covered by at most $K'$ cosets of $G_0$. From our initial observation we conclude that $[G:G_0]\leq K'$.

Note that for some $s=O_K(1)$ the $s$-th term of the central descending series $C^s(G_0)$ is contained in $H$. Moreover, $G_1:=\bigcap_{g \in G}  g G_0 g^{-1}$ is a normal subgroup of $G$ with index at most $O_K(1)$ contained in $G_0$. Hence $N:=C^s(G_1)$ is a normal subgroup of $G$ contained in $H$. On the other hand, $G_1/N$ is nilpotent of complexity bounded in terms of $K$ only and it has index $O_K(1)$ in $G/N$.

To conclude from this that $G$ is virtually nilpotent, it suffices to show that $G_1$ is. However $G_1$ is actually finite-by-nilpotent (the finite group being $N$) and any such group is virtually nilpotent. To see this note that the kernel of the action by conjugation on $N$ is a nilpotent subgroup of finite index. \end{proof}

\begin{remark}\label{morebounds} Recall that the condition $|A^2|\leq K|A|$ implies the existence of an approximate group $Z$ of size $O(K^{O(1)}|A|)$ and of $O(K^{O(1)})$ left translates of $Z$ which cover $A$ (see \cite[Theorem 4.6]{tao-noncommutative}). Using Remark \ref{bounds-nilpotent} and Theorem \ref{dimension-bound}, we then see that $G_1$ can be taken so that $G_1/N$  is $O(\log K)$-nilpotent in the sense that it admits a generating set $u_1,\ldots,u_\ell$ with $\ell=O(\log K)$ such that $[u_i,u_j] \in \langle u_{j+1}, \ldots, u_\ell \rangle$ for all $i<j$. In particular such a group admits a normal series with cyclic factors of length at most $O(\log K)$.
\end{remark}

\begin{remark}\label{containment} If one assumes that $A$ is a $K$-approximate group instead of the doubling condition $|A^2|\leq K|A|$ in Corollary \ref{main-cor}, then we may also conclude from Theorem \ref{hl-conj} that $N$ and a generating set of $G_1$ are contained in $A^4$. Using Theorem \ref{dimension-bound} we see that if  one additionally wishes to ensure the logarithmic bound as in the previous remark, then one can only guarantee that $N$ lies inside $A^{O_K(1)}$.
\end{remark}

The following corollary is reminiscent of Gromov's theorem but it involves a weaker type of polynomial growth condition in which the generating set may be arbitrarily large. Furthermore it only requires that at one scale.

\begin{corollary}\label{petrunin} Let $d > 0$. Then there is $R(d) > 0$ such that the following holds. Suppose that $G$ is generated by a finite symmetric set $S$ and that there is some scale $r > R(d)$ such that $|S^r| \leq r^d |S|$. Then there is a finite normal subgroup $N \lhd G$ and a subgroup $G_1 \leq G$ containing $N$ such that
\begin{enumerate}
\item $N \subseteq S^r$;
\item $G_1$ has index $O_d(1)$ in $G$;
\item $G_1/N$ is $O(d)$-nilpotent \textup{(}see Remark \ref{bounds-nilpotent} for a definition\textup{)}.
\end{enumerate}

\end{corollary}

\begin{proof} Our assumption is that $|S^r| \leq r^d |S|$. Let $K = 2\cdot10^d$ and $C_K$ be such that, in the last part of Remark \ref{containment}, $N$ lies in $A^{C_K}$. We claim that there is some $r_0$,  $\sqrt{r} \leq r_0 \leq r/2C_K$, such that $A := S^{r_0}$ has $|A^5| \leq 10^d |A|$. Note that $A^2$ is then a $K$-approximate group with $K = 2\cdot10^d$ (see Lemma \ref{isgroup}). Applying Corollary \ref{main-cor} and Remark \ref{containment} and ensuring that $R(d)$ is so large that $R(d) > (K')^2$ ($K'$ being the quantity in Corollary \ref{main-cor}), we obtain a finite normal subgroup $N \lhd G$ and a subgroup $G_1 \leq G$ containing $N$ such that $G_1$ has index $O_d(1)$ in $G$ and $G_1/N$ is $O(\log K) = O(d)$-nilpotent. Furthermore $N$ and a set of generators for $G_1$ are contained in $A^{2C_K} = S^{2r_0 C_K} \subseteq S^r$.

It remains to justify the claim. If it is false then $|S^{5^{i+1} \sqrt{r}}| > 10^d |S^{5^{i} \sqrt{r}}|$ whenever $5^i \sqrt{r} < r/10C_K$, and in particular $|S^r| \geq (10^d)^{\log_5 (\sqrt{r}/10C_K) - 1}|S^{\sqrt{r}}|$. If $r$ is greater than some absolute constant, this is greater than $r^d |S|$, contrary to assumption.\end{proof}

\begin{remark} Note that there is no bound on the size of $N$.  Indeed, if $G$ is a large finite simple group and $S = G$ then $N$ must equal $G$, which shows that $|N|$ can be arbitrarily large compared to $d,r$.
\end{remark}

In \cite{shalom-tao} Y. ~Shalom and the third author gave a quantitative refinement of Gromov's theorem inspired by Kleiner's recent new proof (see also \cite[Corollary 4.2]{lee-makarychev} for an earlier result in that direction). A consequence of their result is that a polynomial growth condition at one large scale is enough to guarantee virtual nilpotence. We take the opportunity to record that this follows easily from Corollary \ref{main-cor}.
\begin{corollary}\label{nilpotent} Let $d>0$. Then there is $R(d)>0$ such that the following holds. Suppose that $G$ is generated by a finite symmetric set $S$ containing the identity and that there is some scale $r > R(d)$ such that $|S^r| \leq r^d$. Then $G$ contains $G'$, where
\begin{enumerate}
\item $G'$ has index $O_{d}((r^d)!)$ in $G$;
\item $G'$ is nilpotent with step $O_{d}(1)$.
\end{enumerate}
\end{corollary}

\begin{proof}  We apply Corollary \ref{petrunin} to obtain groups $N, G_1$ with the properties stated there.  As $N$ is contained in $S^r$, it has cardinality at most $r^d$.  The group $G_1$ acts on $N$ by conjugation; since the permutation group of $N$ has cardinality at most $(r^d)!$, we conclude that the stabiliser $G'$ of this action has index at most $(r^d)!$ in $G_1$.  As $G_1/N$ is nilpotent of step $O_d(1)$, we conclude that $G'$ is nilpotent of step $O_d(1)+1 = O_d(1)$, and the claim follows.
\end{proof}

\begin{remark} As observed in the last section of Gromov's original paper \cite{gromov}, Gromov's theorem on polynomial growth already easily implies a weaker result of this kind in which the hypothesis is that $|S^r| \leq r^d$ for \emph{all} $r = 1,2,\dots, R(d)$. Note that this result of Gromov (and, \emph{a fortiori}, Corollary \ref{nilpotent}) have content even when the group $G$ is finite.   Another weakening of the above result appears in \cite{dries-wilkie}, where $|S^r| \leq r^d$ is assumed for infinitely many $r$ rather than for all $r$.

Corollary \ref{nilpotent} is stronger than the results in \cite{lee-makarychev,shalom-tao} in the sense that the bounds do not depend on the cardinality $|S|$ of $S$.  On the other hand, the results in \cite{lee-makarychev,shalom-tao}, which follow a strategy close to that of Kleiner's work \cite{kleiner}, yield more effective quantitative control on the index and step of $G'$, especially in the case when $S$ is of bounded cardinality.
\end{remark}

Another consequence of our main theorem is that polynomial growth in the sense of Corollary \ref{petrunin} at one large scale implies polynomial growth at all subsequent scales.

\begin{corollary}\label{growth} Let $d>0$. Then there is $R'(d) > 0$ such that the following holds. Suppose that $G$ is generated by a finite symmetric set $S$ and that $|S^r|\leq r^d |S|$ for some $r \geq R'(d)$. Then $|S^{r'}|\leq (r')^{O_d(1)}|S|$ for all $r' \geq r$.
\end{corollary}

\begin{proof}
A simple modification of the proof of Corollary \ref{petrunin} shows that there is some $r_0$, $\sqrt{r} \leq r_0 \leq r/6$, such that $|S^{5r_0}| \leq K |S^{r_0}|$ where $K = 100^d$ (say). Applying Corollary \ref{main-cor} with $A := S^{r_0}$ (as before) we obtain a normal subgroup $H \subseteq S^{4r_0}$ such that $G/H$ is virtually nilpotent with the index, step and number of generators of the nilpotent subgroup $G_1/H$ all being $O_d(1)$.

Now by Corollary \ref{isgroup}, $A^2 = S^{2r_0}$ is a $2K$-approximate group. This means that there is some set $X$, $|X| \leq 2K$, such that $S^{4r_0} \subseteq X S^{2r_0}$. From this it follows that \begin{equation}\label{disp-1} S^{2mr_0} \subseteq X^m S^{2r_0}\end{equation} for every positive integer $m$.

Let $\pi:G \rightarrow G/H$ be the quotient homomorphism. We have \begin{equation}\label{disp-2} |S^{2mr_0}| \leq |H| |\pi(S^{2mr_0})|,\end{equation} since the cardinality of any fibre is at most $|H|$.

From \eqref{disp-1} and the fact that $\pi$ is a homomorphism we have
\begin{equation}\label{disp-3} \pi(S^{2mr_0}) \subseteq \pi(X)^m \pi(S^{2r_0}).\end{equation}

On the other hand, since $H \subseteq S^{4r_0}$, we have
\begin{equation}\label{disp-4}  |H| |\pi(S^{2r_0})| \leq |S^{6r_0}|. \end{equation}
Moreover, since $\sqrt{r} \leq r_0 \leq r/6$, we have

\begin{equation}\label{disp-5} |S^{6r_0}| \leq |S^r| \leq r^d |S| \leq r_0^{2d} |S|.\end{equation}

Putting \eqref{disp-2}, \eqref{disp-3}, \eqref{disp-4} and \eqref{disp-5} together gives
\begin{equation}\label{disp-6} |S^{2m r_0}| \leq |\pi(X)^m| r_0^{2d} |S|.\end{equation}
Now $\pi(X)$ is a set of size $O_d(1)$, contained in a virtually nilpotent group in which the index and step of the nilpotent subgroup are $O_d(1)$. Every such group is a quotient of one fixed virtually nilpotent group with number of generators, index and step of the nilpotent subgroup also $O_d(1)$ and whose generators are lifts of the elements in $\pi(X)$. Hence there is a bound of the form
\[ |\pi(X)^m| \leq m^{O_d(1)}\] for all $m > 1$.
Comparing this with \eqref{disp-6} confirms that
\[ |S^{r'}| \leq r^{\prime O_d(1)} |S|\] whenever $r'$ is a multiple $2mr_0$ with $m > 1$. It is not hard to see that the same estimate therefore holds for all $r'$, at the expense of increasing the exponent $O_d(1)$ if necessary.
\end{proof}

A consequence/reformulation of the preceding result is the following.

\begin{corollary}\label{finite-diameter} Let $\alpha>0$. Then there are $r_0 \in \N$ and $\beta>0$ with $\lim_{\alpha \rightarrow 0} \beta(\alpha) =0$ such that the following holds. Let $G$ be a finite group generated by a symmetric set $S$ and assume that the diameter of the associated Cayley graph satisfies $\operatorname{diam}_S(G) \leq (|G|/|S|)^\alpha$. Then $|S^r|\geq \min\{r^{1/\beta}|S|,|G|\}$ if $r\geq r_0(\alpha)$.\end{corollary}
\begin{proof}
If this does not hold for some $r, r_0$ and $\beta$ then as soon as $r_0$ is large enough (in terms of $\beta$) Corollary \ref{growth} applies and yields $|S^n| \leq n^e|S|$ for all $n \geq r$ and some $e = e(\beta) > 0$. In particular, when $n$ reaches the diameter of $G$, we obtain $S^n = G$ so $|G| \leq (\diam_S(G))^e |S|$. This contradicts our hypothesis if $e < 1/\alpha$.
\end{proof}

We shall apply Corollary \ref{finite-diameter} later on to deduce an isoperimetric inequality; see Corollary \ref{isop}.

Finally we show that by repeatedly applying Corollary \ref{main-cor} we can obtain the following more precise result, which says something non trivial for finite groups as well. We say that a polycyclic group has length at most $L$ if it is obtained from the trivial group by at most $L$ successive extensions by a cyclic group.

\begin{corollary}\label{polycyclic} Let $G$ be a group which has a left-invariant metric $d : G \times G \rightarrow [0,\infty)$ satisfying the following conditions for some $K \geq 1$:
\begin{enumerate}
\item \textup{(Uniform doubling property)} We have $|B(2r)| \leq K |B(r)|$ for every $r > 0$;
\item \textup{(Finiteness condition)} There are at most $K$ different subgroups of the form $\langle B(r) \rangle$ as $r$ ranges over $(0,\infty)$.
\end{enumerate}
Then $G$ has a subgroup of index at most $O_K(1)$ which is polycyclic of length $O_K(1)$.
\end{corollary}
\begin{proof} Given $d \in \N$ and $R\geq 0$ we claim that if there are at most $d$ groups of the form $\langle B(r) \rangle$ for $r\leq R$, then $\langle B(R) \rangle$ contains a polycyclic subgroup of index $O_{K,d}(1)$. This is clearly enough to establish the corollary. To prove the claim, we proceed by induction on $d$. It is clear for $d=1$, since $\langle B(R) \rangle$ is then the trivial group.

Let $R_0$ be the upper bound of those $R'\geq 0$ such that there are at most $d-1$ groups of the form $\langle B(r) \rangle$ for $r\leq R'$. Without loss of generality\footnote{Recall that $B(R)$ is the \emph{closed} ball $\{g \in G; d(1,g) \leq R\}$.} $0<R_0 \leq R$. Then $\langle B(r) \rangle= \langle B(R) \rangle$ whenever $R_0 \leq r \leq R$. By the induction hypothesis, $\langle B(R_0/2) \rangle$ contains a polycyclic subgroup $P$ of index $O_{K,d}(1)$ and length $O_{K,d}(1)$.

Let $K' = O_{K}(1)$ be the constant obtained in Corollary \ref{main-cor}. Setting $S=B(R_0)$ and $A=B(K'R_0)$, we may apply Corollary \ref{main-cor} and conclude that $G = \langle B(R_0) \rangle$ contains a subgroup $G_1$ of index $O_K(1)$ such that $G_1$ has a normal subgroup $N \subset B(4K'R_0)$ with $G_1/N$ nilpotent with step and number of generators $O_K(1)$. It is enough to show that $G_1$ has a polycyclic subgroup of index $O_{K,d}(1)$, because then so will $G = \langle B(R) \rangle$.

By the uniform doubling assumption and a covering argument, $B(4K'R_0)$ can be covered by $O_K(1)$ translates of $B(R_0/2)$. It follows that $N$ can be covered by $O_{K,d}(1)$ translates of $P$, and in particular $[N:N \cap P]=O_{K,d}(1)$. Now $N \cap P$ is a subgroup of $P$ and hence is also polycyclic of length $O_{K,d}(1)$; in particular, it is generated by $O_{K,d}(1)$ elements. Therefore so is $N$, and hence $N_0$, the intersection of all subgroups of $N$ of index at most $[N : N \cap P]$, has index $O_{K,d}(1)$ in $N$. (To see this recall Schreier's theorem that if $S$ is a symmetric generating set for a group $\Gamma$, and if $\Gamma' \leq \Gamma$ has index $k$, then $S^{2k-1}$ contains a set of generators for $\Gamma'$.)

The group $N_0$, being a subgroup of $N \cap P$, is polycyclic. It is also characteristic in $N$ and hence, since $N$ is normal in $G_1$, $N_0$ is also normal in $G_1$.

However $G_1$ acts by conjugation $N/N_0$, and the kernel of this action is a subgroup $G'_1$ of $G_1$ with index $O_{K,d}(1)$. Now $(N \cap G'_1)/N_0$ is central in $G'_1/N_0$ and of size $O_{K,d}(1)$. We thus have $N_0 \leq N \cap G'_1 \leq G'_1$, where each successive quotient is polycyclic of length $O_{K,d}(1)$. It follows that $G'_1$ is polycyclic of length $O_{K,d}(1)$, which is what we wanted to establish.
\end{proof}

\begin{remark} There are examples of groups which satisfy the assumptions of Corollary \ref{polycyclic} yet have no nilpotent subgroup of index $O_K(1)$. For instance, let $p$ be a large prime and set $G:= (\Z/p\Z)^2 \rtimes\Z$, where the action is by an element of $\operatorname{SL}_2(\Z/p\Z)$ which is a diagonal matrix $\gamma$ of the form $\gamma:=\operatorname{diag}(x,x^{-1})$, where $x \in \F_p^*$ is a generator of the multiplicative group of $\F_p$. Then no subgroup of $G$ of index less than $p-1$ is nilpotent (note that such a subgroup must contain $(\Z/p\Z)^2$ and be the preimage of the subgroup of $\Z$ with that index). However we can endow $G$ with a uniformly doubling weighted word metric (with $3$ generators) by letting the two standard generators of $(\Z/p\Z)^2$ each have weight $\frac{1}{p}$ and $\gamma$ have weight $1$.\end{remark}

We turn now to some geometric applications of the above results. \vspace{11pt}

\textsc{Manifolds with a lower bound on Ricci Curvature.} A.~Petrunin suggested to us some years ago\footnote{See also {\tt http://mathoverflow.net/questions/11091}.} that a result such as Corollary \ref{petrunin} would give a purely group-theoretical proof of a theorem of Fukaya and Yamaguchi \cite{fukaya-yamaguchi} according to which fundamental groups of almost non-negatively curved manifolds are virtually nilpotent. Recall that a closed manifold $M$ is said to be almost non-negatively curved if one can find a sequence of Riemannian metrics on it for which $\operatorname{diam}(M) \leq 1$ while $K_M \geq -1/n$ where $K_M$ is the sectional curvature. Indeed, a simple application of the Bishop-Gromov inequalities combined with Corollary \ref{petrunin} yields the following improvement assuming only a condition on the Ricci curvature.

\begin{corollary}[Ricci gap]\label{ricci} Given $d \in \N$, there is $\eps(d)>0$ such that the following holds. Let $M$ be an $d$-dimensional compact Riemannian manifold with Ricci curvature bounded below by $-\eps$ and diameter at most $1$. Then $\pi_1(M)$ has normal subgroup of index $O_d(1)$, which is finite-by-(O(d)-nilpotent). In particular $\pi_1(M)$ is virtually nilpotent.
\end{corollary}
\begin{proof} Fix a base point $x_0$ on the universal cover $\widetilde{M}$ and let $\mathcal{F}$ be a Dirichlet fundamental domain based at $x_0$ for the action of $\Gamma:=\pi_1(M)$: that is,
\[ \mathcal{F}:=\{p \in \widetilde{M}: d(x_0,p) \leq d(\gamma \cdot x_0,p) \; \mbox{for all} \; \gamma \in \Gamma\}.\] Set $S:=\{\gamma \in \Gamma: d(\gamma \cdot x_0,x_0) \leq 3\}$. Note that $\operatorname{diam}(\mathcal{F}) \leq 1$ and that $S$ is symmetric and contains $1$. Observe further that $S$ generates $\Gamma$ and that for every integer $r \geq 1$ we have $B(x_0,r) \subset S^r \cdot \mathcal{F} \subset B(x_0,3r+1)$, where $B(x_0,r)$ is the ball of radius $r$ on $\widetilde{M}$ for the Riemannian metric lifted from $M$. It follows that

\begin{equation}\label{above-inequalities} \frac{|S^r|}{|S|} \leq \frac{|B(x_0,3r+1)|}{|B(x_0,1)|}.\end{equation}

From the assumed Ricci curvature bound and the Bishop-Gromov volume comparison estimates (see \cite[Theorem 4.19]{gallot-hulin}) we have the bound

$$\frac{|B(x_0,r)|}{|B(x_0,1)|} \leq \frac{|B_{-\eps}(r)|}{|B_{-\eps}(1)|},$$

\noindent where $B_{-\eps}(r)$ is a metric ball in the comparison model space with constant curvature $-\eps$ and dimension $d$. The volume of this ball is  $|B_{-\eps}(r)|=\frac{1}{(\sqrt{\eps})^d}|B_{-1}(r/\sqrt{\eps})|=c_d\int_{0}^{r} \big( \frac{\sinh(\sqrt{\eps}t)}{\sqrt{\eps}} \big)^{d-1} dt$, where $c_d>0$ is the volume of the $d-1$-dimensional unit sphere (see \cite[p. 138]{gallot-hulin} for this volume computation). As $\eps$ tends to $0$, this tends to $c_dr^d/d$.  Combining this with \eqref{above-inequalities} we obtain that for every $R_0 \geq 1$ there is some $\eps_0 = \eps_0(d, R_0)$ such that

$$\frac{|S^r|}{|S|} \leq 2(3r+1)^d$$
for all $r \leq R_0$ provided that $0 < \eps < \eps_0$. Letting $R_0 = R_0(2d)$ be as in Corollary \ref{petrunin}, we obtain the existence of some $\eps=\eps(d)>0$ for which the conclusion of that statement holds. This completes the proof.
\end{proof}

\begin{remark} The fact that $\pi_1(M)$ is virtually nilpotent under the above Ricci bounds assumptions was obtained by Cheeger and Colding in \cite{cheeger-colding} (and had been conjectured earlier by Gromov) and their proof was recently completed and extended by Kapovitch and Wilking \cite{kapovitch-wilking}, who also established that the index of the nilpotent subgroup is uniformly bounded by a constant depending on the dimension $d$ only, an improvement which seems beyond the scope of our methods. This extended earlier work of Kapovitch, Petrunin and Tuschmann in \cite{kapovitch-petrunin-tuschmann} which proved the same result under sectional curvature bounds instead of Ricci. The work of these authors is, unlike our work, differential-geometric in nature. The linear dependence in $d$ of the nilpotency length proven in our Corollary \ref{ricci} seems new however. \end{remark}

\textsc{An isoperimetric inequality.} It has been well-known since the work of Varopoulos on Kesten's conjecture (\cite{varopoulos,pittet-saloff}) that isoperimetric inequalities on Cayley graphs are closely related to lower bounds on the volume growth. Using this idea and Corollary \ref{finite-diameter}, we can derive the following property of finite Cayley graphs with a polynomial upper bound on the diameter.

\begin{corollary}[Isoperimetric inequality on finite groups] \label{isop} Let $\alpha>0$. Then there are $r_0 \in \N$ and $\beta>0$ with $\lim_{\alpha \rightarrow 0} \beta(\alpha) =0$ such that the following holds. Let $G$ be a finite group generated by a symmetric set $S$ and assume that the diameter of the associated Cayley graph satisfies $\operatorname{diam}_S(G) \leq (|G|/|S|)^\alpha$. Then for every subset $E$ in $G$ with $\frac{1}{2}r_0 \leq |E|\leq \frac{1}{2}|G|$, $|\partial E| \geq \frac{1}{8}|S|^\beta |E|^{1-\beta}$.
\end{corollary}
\begin{proof} This follows almost immediately from Corollary \ref{finite-diameter} and the following well-known lemma, which may be found in \cite[Chapter 5]{gromov-pansu-lafontaine} or \cite{pittet-saloff} and references therein. For the convenience of the reader we offer a self-contained proof.\end{proof}

\begin{lemma}[Isoperimetry versus growth] Let $G$ be group and $S$ some finite symmetric generating set containing $1$. Let $B(r)=S^r$ be the word ball of radius $r$ in the word metric. Let $\partial E= SE \setminus E$ be the boundary of a subset $E \subset G$. If $E \subseteq G$ is a set, write $r(E)$ for the infimum of those $r$ for which $|B(r)| \geq 2|E|$. Then for all $E$ with $|E|<|G|/2$ we have $|E| \leq 4r(E) |\partial E|$.
\end{lemma}

\begin{proof} We give a proof for the reader's convenience. Let $f=1_E$ the indicator function of the set $E$, and $f_r:=\frac{1}{|B(r)|} \sum_{g \in B(r)} g \cdot f$ be the average of $f$ over balls of radius $r$. By the triangle inequality we have $||g \cdot f -f||_1 \leq |g|\cdot \max_{s \in S} ||s \cdot f -f||_1$, where $|g|$ is the distance to the identity in the word metric. Moreover  $||s \cdot f -f||_1=|sE \vartriangle E|\leq 2|\partial E|$ for every $s\in S$. Hence $||f_r -f||_1 \leq 2r|\partial E|$. On the other hand for every $x \in E$, there are at most $|E|$ elements $g \in B(r)$ such that $g \cdot 1_E (x) \neq 0$. Therefore if $|B(r)| \geq 2|E|$ then $f_r(x)\leq \frac{1}{2}$ and hence $||f_r -f||_1 \geq \frac{1}{2}|E|$. The claim follows.
\end{proof}

In \cite{benjamini-kozma}, Benjamini and Kozma conjecture that one can take $\beta=\alpha$ in the Corollary \ref{isop} (at the expense of introducing a possible multiplicative constant $c_\alpha$ in place of $|S|^\beta/8$ in (ii)). This, however, is beyond the scope of our method. We would like to thank Itai Benjamini for drawing our attention to their work and its connection to Gromov-type theorems.\vspace{11pt}

\textsc{A generalized Margulis lemma.} In hyperbolic geometry, the Margulis lemma asserts that there is a constant $\eps=\eps(n)>0$, the Margulis constant, such for any discrete subgroup $\Gamma$ of isometries of the hyperbolic $n$-space $\mathbb{H}^n$, and any point $x\in \mathbb{H}^n$, the almost stabiliser $\Gamma_\eps(x): =\{ \gamma \in \Gamma: d(\gamma \cdot x,x) < \eps \}$ is virtually cyclic. This lemma is important for describing the geometry of cusps in hyperbolic manifolds, or for establishing volume lower bounds (see e.g. \cite{thurston}). Various generalisations of this lemma have been established in the past for more general Riemannian manifolds under curvature upper and lower bounds (e.g.  \cite[chap. 6]{burago-zalgaller}). Typically in these results, unless the manifold has strictly negative curvature, ``virtually cyclic'' in the conclusion of the lemma must be replaced by ``virtually nilpotent''.

In \cite[\S 5.F]{gromov-pansu-lafontaine} Gromov raises the issue of establishing a \emph{generalized Margulis lemma} under very weak assumptions on the metric space and he proposes a conjectural statement in this direction. Below we answer Gromov's question affirmatively.

A metric space $X$ is said to have \emph{bounded packing} with packing constant $K$ if there is $K>0$ such that every ball of radius $4$ in $X$ can be covered by at most $K$ balls of radius $1$. Say that a subgroup $\Gamma$ of isometries of $X$ acts \emph{discretely} on $X$ if every orbit is discrete in the sense that $\{\gamma \in \Gamma : \gamma \cdot x \in \Sigma\}$ is finite for every $x \in X$ and for every bounded set $\Sigma \subseteq X$.

\begin{corollary}[Generalized Margulis Lemma]\label{gromov-conj-again} Let $K \geq 1$ be a parameter. Then there is some $\eps(K) > 0$ such that the following is true. Suppose that $X$ is a metric space with packing constant $K$, and that $\Gamma$ is a subgroup of isometries of $X$ which acts discretely. Then for every $x \in X$ the ``almost stabiliser'' $\Gamma_\eps(x) = \langle S_{\eps}(x)\rangle$, where $S_{\eps}(x) : =\{ \gamma \in \Gamma: d(\gamma \cdot x,x) < \eps \}$, is virtually nilpotent.
\end{corollary}
\begin{proof}  Each set $S_{r}(x)$ is symmetric and contains the identity. Now by the assumption on $X$ the ball $B(x,4)$ can be covered by collection of balls $B(x_i, 1)$, $i = 1,2,\dots,K$.  Suppose that for $i = 1,2,\dots,k$ there is at least one element $\gamma_i \in S_4(x)$ with $\gamma_i \cdot x \in B(x_i,1)$. Suppose now that $\gamma \in S_4(x)$ is arbitrary; then there is some $i \in \{1,2,\dots,k\}$ such that $\gamma \cdot x \in B(x_i,1)$. But this means that $d(\gamma \cdot x, \gamma_i \cdot x) < 2$, and therefore $\gamma_i^{-1}\gamma \in S_2(x)$. This implies that $S_4(x) \subseteq \bigcup_{i=1}^k \gamma_i S_2(x)$, which yields (since $S_2(x)^2 \subseteq S_4(x)$) the doubling estimate $|S_2(x)^2| \leq K|S_2(x)|$.

Let $K'=K'(K)>0$ be the constant from Corollary \ref{main-cor}. Set $\eps:=2/K'$, $S=S_\eps(x)$ and $A=S_2(x)$. A direct application of Corollary \ref{main-cor} shows that $\Gamma_\eps(x)=\langle S \rangle$ is virtually nilpotent.\end{proof}

\begin{remark} This confirms Gromov's conjecture, which suggested the same conclusion under the slightly stronger hypotheses that every ball of radius $R$ in $X$ can be covered by at most $C(R/r)^m$ balls of radius $r$ for all $0<r<R\leq 1$ and some fixed constants $C,m>0$.\end{remark}

The assumptions of this generalized Margulis lemma are satisfied for example if $X$ is a complete Riemannian manifold with a lower bound on its Ricci curvature, by an immediate application of the Bishop-Gromov volume comparison estimates. In this case, the result was proved by Cheeger-Colding \cite{cheeger-colding} and Kapovitch-Wilking \cite{kapovitch-wilking}, namely:

\begin{corollary}\label{gromov-conj-2} Let $d \geq 1$ be an integer. Then there is $\eps=\eps(d)>0$ with the following property. Suppose that $M$ is a $d$-dimensional complete Riemannian manifold with a Ricci curvature lower bound $\operatorname{Ric} \geq -(d-1)$ and that $\Gamma$ is a subgroup of $\operatorname{Isom}(M)$ which acts properly discontinuously by isometries on $M$. Then for every $x \in M$ the ``almost stabliser'' $\Gamma_\eps(x) :=\{ \gamma \in \Gamma:  d(\gamma \cdot x,x) < \eps \}$ is virtually nilpotent.
\end{corollary}

In fact the result in \cite[Theorem 1]{kapovitch-wilking} is a stronger version of Corollary \ref{gromov-conj-2},  establishing that $\Gamma_\eps(x)$ has a nilpotent subgroup of index $O_d(1)$. This stronger result seems to be beyond the scope of our method.

We also note that Corollary \ref{polycyclic} applies to the Margulis lemma in the context of Riemannian $d$-manifolds with a lower bound on sectional curvature, because then the Gromov short basis has bounded cardinality from Toponogov's theorem (see for instance \cite[37.3]{burago-zalgaller}). We thus get this way an alternate proof of the Fukaya-Yamaguchi theorem \cite{fukaya-yamaguchi} according to which almost non-negatively curved $n$-manifolds have $O_n(1)$-virtually polycyclic fundamental group. Again, by \cite{kapovitch-petrunin-tuschmann} we know better, namely that they are $O_n(1)$-virtually nilpotent, but once again this seems beyond the scope of our method.

Finally we would like to remark that the usual proofs of the classical Margulis lemma bear some resemblance to the proof of our main theorem in as much as they use a similar ``shrinking commutator trick'' to establish nilpotence. While we proved this shrinking commutator estimate for the escape norm associated to an approximate group as part of the Gleason lemmas (Theorem \ref{gleason-thm}), in the Margulis lemma, one proves a similar estimate for the norm $\Vert \gamma \Vert_x=d(\gamma \cdot x,x)$ by a riemannian geometric argument using the assumed curvature bounds. This ``shrinking commutator trick'' dates back at least to Bieberbach \cite{bieberbach} in his proof of Jordan's theorem on finite linear groups.

\appendix

\section{Basic theory of ultralimits and ultraproducts}\label{nsa-app}

In this appendix we review the machinery of ultralimits and ultraproducts.  We will borrow some terminology from nonstandard analysis in order to do this, although we will not rely too heavily on nonstandard machinery in this paper.

We will assume the existence of a \emph{standard universe} ${\mathfrak U}$ which contains all the objects and spaces that one is interested in (such as the natural numbers $\N$, the real numbers $\R$, the classical Lie groups, etc.).  The precise construction of this universe is not particularly important for our purposes, so long as it forms a set.  We refer to objects and spaces inside the standard universe as \emph{standard objects} and \emph{standard spaces}, with the latter being sets whose elements are in the former category.

We will rely heavily on the existence of a \emph{nonprincipal ultrafilter}.

\begin{lemma}[Ultrafilter lemma]\label{ultralem}  There exists a collection $\alpha$ of subsets of the natural numbers $\N$ with the following properties:
\begin{enumerate}
\item \textup{(Monotonicity)} If $A \in \alpha$ and $B \supseteq  A$, then $B \in \alpha$.
\item \textup{(Closure under intersection)} If $A,B \in \alpha$, then $A \cap B \in \alpha$.
\item \textup{(Maximality)} If $A \subseteq  \N$, then either $A \in \alpha$ or $\N \backslash A \in \alpha$, but not both.
\item \textup{(Non-principality)} If $A \in \alpha$, and $A'$ is formed from $A$ by adding or deleting finitely many elements to or from $A$, then $A' \in \alpha$.
\end{enumerate}
We refer to a collection $\alpha$ obeying the above axioms as a \emph{nonprincipal ultrafilter}.
\end{lemma}

\begin{proof}  The collection of cofinite subsets of $\N$ already obeys the monotonicity, closure under intersection, and non-principality properties.  Using Zorn's lemma\footnote{By using this lemma, our results thus rely on the axiom of choice, which we will of course assume throughout this paper.  On the other hand, it is possible to rephrase the purely combinatorial results in this paper, such as Theorem \ref{main-theorem}, in the language of Peano arithmetic.  Applying a famous theorem of G\"odel \cite{godel}, we then conclude that Theorem \ref{main-theorem} is provable in ZFC if and only if it is provable in ZF.  In fact it is possible, with significant effort, to directly translate these ultrafilter arguments to a much lengthier argument in which neither ultrafilters nor the axiom of choice are used.  However, this would require one to ``finitise'' or ``proof-mine'' such infinitary results as the Heine-Borel theorem or Theorem \ref{goldbring-thm}, and this in turn would require finitisations of the construction of Haar measure and the Peter-Weyl theorem. This would lead to a vastly messier argument.}, one can enlarge this collection to a maximal collection which, it may be verified, has all the required properties.
\end{proof}

Throughout the paper, we fix a non-principal ultrafilter $\alpha$.  A property $P(\n)$ depending on a natural number $\n$ is said to hold \emph{for $\n$ sufficiently close to $\alpha$} if the set of $\n$ for which $P(\n)$ holds lies in $\alpha$.  A set of natural numbers lying in $\alpha$ will also be called an \emph{$\alpha$-large set}.

Once we have fixed this ultrafilter, we can define nonstandard objects and spaces.

\begin{definition}[Nonstandard objects]
Given a sequence $(x_\n)_{\n \in \N}$ of standard objects in ${\mathfrak U}$, we define their \emph{ultralimit} $\lim_{\n \to \alpha} x_\n$ to be the equivalence class of all sequences $(y_\n)_{\n \in \N}$ of standard objects in ${\mathfrak U}$ such that $x_\n = y_\n$ for $\n$ sufficiently close to $\alpha$.  Note that the ultralimit $\lim_{\n \to \alpha} x_\n$ can also be defined even if $x_\n$ is only defined for $\n$ sufficiently close to $\alpha$.

An ultralimit of standard natural numbers is known as a \emph{nonstandard natural number}, an ultralimit of standard real numbers is known as a \emph{nonstandard real number}, and so on.

For any standard object $x$, we identify $x$ with its own ultralimit $\lim_{\n \to \alpha} x$.  Thus, every standard natural number is a nonstandard natural number, etc.

Any operation or relation on standard objects can be extended to nonstandard objects in the obvious manner.  Indeed, if $O$ is a $k$-ary operation, we define
$$ O( \lim_{\n \to \alpha} x^1_\n, \ldots, \lim_{\n \to \alpha} x^k_\n ) := \lim_{\n \to \alpha} O( x^1_\n, \ldots, x^k_\n ) $$
and if $R$ is a $k$-ary relation, we define $R(\lim_{\n \to \alpha} x^1_\n, \ldots, \lim_{\n \to \alpha} x^k_\n )$ to be true iff $R(x^1_\n,\ldots,x^k_\n)$ is true for all $\n$ sufficiently close to $\alpha$.  One easily verifies that these nonstandard extensions of $O$ and $R$ are well-defined.
\end{definition}

\begin{example}
The sum of two nonstandard real numbers $\lim_{\n \to \alpha} x_\n$, $\lim_{\n \to \alpha} y_\n$ is the nonstandard real number
$$\lim_{\n \to \alpha} x_\n + \lim_{\n \to \alpha} y_\n = \lim_{\n \to \alpha} x_\n + y_\n,$$
and the statement $\lim_{\n \to \alpha} x_\n < \lim_{\n \to \alpha} y_\n$ means that $x_\n < y_\n$ for all $\n$ sufficiently close to $\alpha$.
\end{example}

\begin{definition}[Ultraproducts]  Let $(X_\n)_{\n \in \N}$ be a sequence of standard spaces $X_\n$ in ${\mathfrak U}$ indexed by the natural numbers.  The \emph{ultraproduct} $\prod_{\n \to \alpha} X_\n$ of the $X_\n$ is defined to be the space of all ultralimits $\lim_{\n \to \alpha} x_\n$, where $x_\n \in X_\n$ for all $\n$.  We refer to the ultraproduct of standard sets as an \emph{nonstandard set}; in a similar vein, an ultraproduct of standard groups is a \emph{nonstandard group}, and an ultraproduct of standard finite sets is a \emph{nonstandard finite set}.  We refer to $\ultra X := \prod_{\n \to \alpha} X$ as the \emph{ultrapower} of a standard set $X$; the identification of $x$ with $\lim_{\n \to \alpha} x$ causes $X$ to be identified with a subset of $\ultra X$.  We will refer to the ultrapower  $\ultra {\mathfrak U}$ of the standard universe ${\mathfrak U}$ as the \emph{nonstandard universe}.
\end{definition}

\begin{remark} Nonstandard sets in nonstandard analysis behave analogously in some ways to measurable sets\footnote{Actually, the notion of an \emph{elementary set} (e.g. a finite union of intervals) would be an even closer analogy here than the notion of a measurable set.} in measure theory; for instance, the union or intersection of two nonstandard sets is again a nonstandard set.  Also, just as a subset of a measurable set need not be measurable, a subset of a nonstandard set need not be another nonstandard set.  For instance, the nonstandard natural numbers $\ultra \N$ is a nonstandard set (being the ultraproduct of the sequence $\N, \N, \ldots$), but the standard natural numbers $\N$, despite being a subset of $\ultra\N$, is not a nonstandard set.
\end{remark}

A fundamental property of ultralimits is that they preserve first-order statements and predicates, a fact known as \emph{{\L}os's theorem}.  Here is one formalisation of this theorem.

\begin{theorem}[{\L}os's theorem with parameters]\label{los-param-basic}  Let $m$ be a standard natural number, and for each $1 \leq i \leq m$, let $x_i = \lim_{\n \to \alpha} x_{i,\n}$ be a nonstandard object.  If $P(y_1,\ldots,y_m)$ is a predicate, then $P(x_1,\ldots,x_m)$ is true \textup{(}as quantified over the nonstandard universe $\ultra {\mathfrak U}$\textup{)} if and only if $P(x_{1,\n},\ldots,x_{m,\n})$ is true for all $\n$ sufficiently close to $\alpha$ \textup{(}as quantified over the standard universe ${\mathfrak U}$\textup{)}.
\end{theorem}

\begin{proof} (Sketch)  By definition, {\L}os's theorem is true for ``primitive'' predicates which take the form $R(x_1,\ldots,x_k)$ for some primitive $k$-ary relation $R$ and objects $x_1,\ldots,x_k$, or of the form $x_{k+1} = O(x_1,\ldots,x_k)$ for some primitive $k$-ary operator $O$.  From the ultrafilter axioms, we also see that {\L}os's theorem is closed with respect to boolean operations; for instance, if Theorem \ref{los-param} holds for $P(x_1,\ldots,x_m)$ and $Q(x_1,\ldots,x_m)$, then it also holds for $\neg P$ or $P \wedge Q$.

Now, we claim that if {\L}os's theorem holds for the predicate $P(x_1,\ldots,x_m)$, then it also holds for the quantified predicates $\exists x_m: P(x_1,\ldots,x_m)$ and $\forall x_m: P(x_1,\ldots,x_m)$ (where now there are only $m-1$ free variables $x_1,\ldots,x_{m-1}$, with $x_m$ being bound).  We show this just for the existential quantifier $\exists$, as the case of the universal quantifier $\forall$ is similar (and can be deduced from the existential case by negation).  Suppose first that $\exists x_m: P(x_1,\ldots,x_m)$ is true in $\ultra {\mathfrak U}$. Then there exists $x_m = \lim_{\n \to \alpha} x_{m,\n}$ such that $P(x_1,\ldots,x_m)$ holds; by hypothesis, this implies that $P( x_{1,\n},\ldots,x_{m,\n})$ holds for $\n$ sufficiently close to $\alpha$, and thus $\exists x_m: P( x_{1,\n},\ldots,x_{m-1,\n}, x_m)$ holds for $\n$ in ${\mathfrak U}$ sufficiently close to $\alpha$ as desired.  Conversely, if
$\exists x_m: P( x_{1,\n},\ldots,x_{m-1,\n}, x_m )$ holds in ${\mathfrak U}$ for $\n$ sufficiently close to $\alpha$, then by the axiom of (countable) choice, we may find $x_{m,\n} \in {\mathfrak U}$ for such $\n$ such that $P( x_{1,\n},\ldots,x_{m-1,\n}, x_{m,\n} )$ holds.  Setting $x_m := \lim_{\n \to \alpha} x_{m,\n}$, we conclude that $P(x_1,\ldots,x_m)$ holds, and the claim follows.

The above discussion yields {\L}os's theorem for any predicate that can be built out of primitive predicates by a finite number of boolean operations and quantifications.  However, it is easy to see that all predicates are logically equivalent to a predicate of this form.  For instance, $\forall a \forall b \forall c:(a+b)+c=a+(b+c)$ is equivalent to
$$ \forall a \forall b \forall c \exists d \exists e \exists f: (d=a+b) \wedge (e=b+c) \wedge (f=d+c) \wedge (f=a+e).$$
This completes the proof.
\end{proof}

In applications, we will actually use a slight generalisation of {\L}os's theorem.

\begin{theorem}[{\L}os's theorem with parameters and ultraproducts]\label{los-param}  Let $m, k$ be standard natural numbers.  For each $1 \leq i \leq m$, let $x_i = \lim_{\n \to \alpha} x_{i,\n}$ be a nonstandard object, and for each $1 \leq j \leq k$, let $A_j = \prod_{\n \to \alpha} A_{j,\n}$ be a nonstandard set.  If $P(y_1,\ldots,y_m; B_1,\ldots,B_k)$ is a predicate over $m$ objects and $k$ sets, with the sets $A_1,\ldots,A_k$ only appearing in $P$ through the membership predicate $x \in B_j$ for various $j$ and various objects $B_j$, then $P(x_1,\ldots,x_m; A_1,\ldots,A_k)$ is true \textup{(}as quantified over the nonstandard universe $\ultra {\mathfrak U}$\textup{)} if and only if $P(x_{1,\n},\ldots,x_{m,\n}; A_{1,\n},\ldots,A_{k,\n})$ is true for all $\n$ sufficiently close to $\alpha$ \textup{(}as quantified over the standard universe ${\mathfrak U}$\textup{)}.
\end{theorem}

\begin{proof}  We replace each appearance of $x \in B_j$ in $P$ with a new primitive relation $R_j(x,\n)$, which is interpreted in ${\mathfrak U}$ as $x \in A_{j,\n}$.  This replaces the predicate $P(y_1,\ldots,y_m; B_1,\ldots,B_k)$ by a predicate $Q(y_1,\ldots,y_m,\n)$, with $P(x_{1,\n},\ldots,x_{m,\n}; A_{1,\n},\ldots,A_{k,\n})$ logically equivalent to $Q(x_{1,\n},\ldots,x_{m,\n},\n)$.  One easily verifies that
$P(x_{1},\ldots,x_{m}; A_{1},\ldots,A_{k})$ is logically equivalent to $Q(x_1,\ldots,x_m, \lim_{\n \to \alpha} \n)$, and the claim now follows from Theorem \ref{los-param-basic}.
\end{proof}

\begin{example}  Any ultraproduct $G := \prod_{\n \to \alpha} G_\n$ of groups $G_\n$ is again a group, because one can write the property of $G$ being a group as a predicate $P(G)$ that involves membership in $G$ (as well as the constant $\id$ and the group operations $\cdot, ()^{-1}$, of course).  Conversely, if $G = \prod_{\n \to \alpha} G_\n$ is a group, then $G_\n$ is a group for all $\n$ sufficiently close to $\alpha$.
\end{example}

\begin{example}   Let $G = \prod_{\n \to \alpha} G_\n$ be an ultraproduct of groups (and thus also a group), and let $A = \prod_{\n \to \alpha} A_\n$ and $B = \prod_{\n \to \alpha} B_\n$ be subsets of $G$ that are nonstandard sets.  Then, for $\n$ sufficiently close to $\alpha$, $A_\n$ and $B_\n$ are subsets of $G_\n$ and $B_\n$ (because this statement can be written as a predicate involving membership in $A_\n, B_\n, G_\n$).  In a similar (but more complicated) spirit, for any standard $K \in \N$, $A$ can be covered by $K$ left-translates of $B$ if and only if, for $\n$ sufficiently close to $\alpha$, $A_\n$ can be covered by $K$ left-translates of $B_\n$.
\end{example}

A nonstandard real number $x \in \ultra \R$ is said to be \emph{bounded} if one has $|x| \leq C$ for some standard $C>0$, and \emph{unbounded} otherwise.  Similarly, we say that $x$ is \emph{infinitesimal} if $|x| \leq c$ for all standard $c>0$; in the former case we write $x = O(1)$, and in the latter $x=o(1)$.  For every bounded real number $x \in \ultra \R$ there is a unique standard real number $\st(x) \in \R$, called the \emph{standard part} of $\R$, such that $x = \st(x)+o(1)$, or equivalently that $\st(x) - \eps \leq x \leq \st(x)+\eps$ for all standard $\eps>0$.  Indeed, one can set $\st(x)$ to be the supremum of all the real numbers $y$ such that $x>y$ (or equivalently, the infimum of all the real numbers $y$ such that $x<y$).
We write $X = O(Y)$, $X \ll Y$, or $Y \gg X$ if we have $X \leq CY$ for some standard $C$.

Given a sequence $f_\n: X_\n \to Y_\n$ of standard functions between standard sets $X_\n, Y_\n$, one can form the \emph{ultralimit} $f :=\lim_{\n \to \alpha} f_\n$, which is a function from the ultraproduct $X := \prod_{\n \to \alpha} X_\n$ to the ultraproduct $Y := \prod_{\n \to \alpha} Y_\n$ defined by the formula
$$ f( \lim_{\n \to \alpha} x_\n ) := \lim_{\n \to \alpha} f_\n(x_\n).$$
Such ultralimits will be called \emph{nonstandard functions} (and are also known as \emph{internal functions} in the nonstandard analysis literature).  In particular, since standard finite sequences $(a_n)_{n=1}^N$ of standard reals $a_n \in \R$ with some standard length $N \in \N$ can be viewed as a function $n \mapsto a_n$ from $\{1,\ldots,N\}$ to $\R$, one can thus define \emph{nonstandard finite sequences} $(a_n)_{n=1}^N$ of nonstandard reals $a_n \in \ultra \R$ with some nonstandard length $N \in \ultra \N$ as an ultralimit of standard finite sequences $(a_{n_\n,\n})_{n_\n = 1}^{N_\n}$, thus $N = \lim_{\n \to \alpha} N_\n$ and
$$ a_{\lim_{\n \to \alpha} n_\n} = \lim_{\n \to \alpha} a_{n_\n,\n}.$$
One can then transplant various operations on standard finite sequences to their nonstandard counterparts, and can in particular define the sum
$$ \sum_{n=1}^N a_n \in \ultra \R$$
of a nonstandard finite sequence $(a_n)_{n=1}^N = \lim_{\n \to \alpha} (a_{n_\n,\n})_{n_\n = 1}^{N_\n}$ by the formula
$$\sum_{n=1}^N a_n := \lim_{\n \to \alpha} \sum_{n_\n=1}^{N_\n} a_{n_\n,\n}.$$

\section{Local groups}\label{local-sec}

In this appendix we recall the basic definitions and notations of (symmetric) local group theory, following Goldbring \cite{goldbring-local}.

\begin{definition}[Local group]\label{local-def}  A \emph{symmetric local group} $G = (G, \id, \cdot, ()^{-1})$ is a topological space $G$ with a distinguished element $\id \in G$ (the \emph{identity element}), together with a globally defined inversion map $()^{-1}: G \to G$ and a partially defined product map $\cdot: \Omega \to G$, obeying the following axioms:
\begin{itemize}
\item[(i)] (Partial closure) $\Omega$ is an open neighbourhood of $(G \times \{1\}) \cup (\{1\} \times G)$ in $G \times G$.
\item[(ii)] (Continuity)  The maps $()^{-1}: x \mapsto x^{-1}$ and $\cdot: (x,y) \mapsto x \cdot y$ are continuous on $G$ and $\Omega$ respectively.
\item[(iii)] (Local associativity) If $g,h,k \in G$ are such that $(g \cdot h) \cdot k$ and $g \cdot (h \cdot k)$ are well-defined (thus $(g,h)$, $(g \cdot h, k)$, $(h,k)$, $(g, h \cdot k)$ all lie in $\Omega$), then $(g \cdot h) \cdot k = g \cdot (h \cdot k)$.
\item[(iv)] (Identity)  For any $g \in G$, one has $\id \cdot g = g \cdot \id = g$.
\item[(v)] (Invertibility) If $g \in G$, then $g \cdot g^{-1}$ and $g^{-1} \cdot g$ are well-defined (i.e. $(g,g^{-1}), (g^{-1},g) \in \Omega$) and are equal to $\id$.
\end{itemize}
If necessary, we will write $\id, \Omega$ as $\id_G, \Omega_G$ to reduce confusion.  If $\Omega = G \times G$, we call $G$ a \emph{global group} or a \emph{topological group}.

If $G$ has the structure of a smooth finite-dimensional real manifold, and the inversion map $()^{-1}$ and product map $\cdot$ are smooth maps, we say that $G$ is a \emph{local Lie group}.
\end{definition}

\begin{remark}
One can also consider non-symmetric local groups, in which the inversion map $()^{-1}$ is
only defined on an open neighbourhood $\Lambda$ of the identity. However, the theory of non-symmetric local
groups contains some minor additional technicalities caused by the existence of non-invertible elements
which we wish to avoid here.  As we will not consider non-symmetric local groups anywhere in this paper, we will often omit the adjective ``symmetric'' from the term ``local group'' when there is no chance of confusion.

Following \cite{goldbring-local}, we do not explicitly assume that $G$ is Hausdorff.  In practice, though, one can reduce to the Hausdorff case because the closure of the identity element will turn out to be a closed normal subgroup that one can quotient out by.
\end{remark}

\begin{example}\label{restrict}  If $G$ is a symmetric local group and $U$ is a symmetric open neighbourhood of the identity (thus $g^{-1} \in U$ whenever $g \in U$), then $U$ can also be viewed as a symmetric local group, by restricting the domain $\Omega$ of the product maps to $\{ (g,h) \in \Omega \cap (U \times U): g \cdot h \in U \}$ (and also restricting the topological structure of $G$ to $U$).  We will sometimes write this symmetric local group as $G\downharpoonright_U$ to emphasise that it is the restriction of $G$ to $U$.   In particular, an important source of local groups comes from restricting a global group to an open symmetric neighbourhood of the identity.

One can also restrict $G$ to non-open symmetric neighbourhoods of the identity, but the resulting object obtained is not necessarily a symmetric local group (see e.g. Example \ref{closint} below).
\end{example}

We say that two symmetric local groups $G, G'$ are \emph{locally identical} if they have a common restriction, thus there exists a $U$ which is an open symmetric neighbourhood of the identity $1_G = 1_{G'}$ in both $G$ and $G'$ for which the group operations on $G$ and $G'$, when restricted to $U$, agree completely (in particular, they have the same domain and range).  This is an equivalence relation, and we will focus on those properties of symmetric local groups that are preserved up to local identity.

In a similar spirit, we say that two subsets $A, B$ of a symmetric local group in $G$ are \emph{locally identical} if there exists an open neighbourhood $U$ of the identity in $G$ such that $A \cap U = B \cap U$.  For instance, all neighbourhoods of the identity are locally identical.  Note that every open neighbourhood if the identity contains an open symmetric neighbourhood, so we can assume here that $U$ is symmetric without loss of generality.

\begin{remark}  Symmetric local groups are defined as topological groups, but if one wishes, one can restrict attention to \emph{discrete} symmetric local groups, in which every set is open.  In this case, all references to continuity, openness, and the Hausdorff property in Definition \ref{local-def} can be omitted as being automatically satisfied.  On the other hand, all discrete local groups are locally equivalent to the trivial local group $\{\id\}$.
\end{remark}

\begin{example} If ${\mathfrak g}$ is a (finite-dimensional) Lie algebra, and $B$ is a sufficiently small symmetric open neighbourhood of the identity in ${\mathfrak g}$, then $\exp(B)$ is a symmetric local group, with the multiplication law given by the Baker-Campbell-Hausdorff formula.
\end{example}

\begin{example}\label{closint}  The closed interval $[-1,1]$ in $\R$ with the addition operation is \emph{not} a symmetric local group, because the set $\{ (x,y) \in [-1,1] \times [-1,1]: x+y \in [-1,1]\}$ is not open in $[-1,1] \times [-1,1]$.  However, the \emph{open} interval $(-1,1)$ is a symmetric local group.
\end{example}

Given any finite number of elements $g_1,\ldots,g_m$ in a global group $G$, one can use the associativity axiom to unambiguously define the product $g_1 \ldots g_m$.  In a symmetric local group, one can only define this product $g_1 \ldots g_m$ locally.  We formalise this as a definition:

\begin{definition}[Finite products]  Let $g_1,\ldots,g_m$ be a finite number of elements in a symmetric local group $G$.  We say that the product $g_1 \ldots g_m$ is \emph{well-defined in $G$} (or \emph{well-defined} for short) if, for each $1 \leq i \leq j \leq m$, we can find a group element $g_{[i,j]} \in G$ with the following properties:
\begin{itemize}
\item For each $1 \leq i \leq m$, we have $g_{[i,i]} = g_i$.
\item If $1 \leq i \leq j < k \leq m$, the product $g_{[i,j]} \cdot g_{[j+1,k]}$ is well-defined (i.e. $(g_{[i,j]}, g_{[j+1,k]}) \in \Omega$) and equal to $g_{[i,k]}$.
\end{itemize}
By induction we see that if these group elements $g_{[i,j]}$ exist, then they are unique.  We then define $g_1 \ldots g_k := g_{[1,k]}$.  If $g_1 = \ldots = g_k = g$, we abbreviate $g_1 \ldots g_k$ as $g^k$.  By abuse of notation, we also write $g_1 \ldots g_m \in G$ to denote the assertion that $g_1 \ldots g_m$ is defined in $G$.

We adopt the convention that $g_1 \ldots g_m = \id$ when $m=0$.
\end{definition}

An easy induction using the local associativity axiom shows that if $g_1,\ldots,g_m \in G$ is such that $g_i \ldots g_j$ is well-defined whenever $1 \leq i < j \leq m$ with $(i,j) \neq (1,m)$, and $(g_i \ldots g_j) \cdot (g_{j+1} \ldots g_k)$ is well-defined whenever $1 \leq i \leq j < k \leq m$, then $g_1 \ldots g_m$ is well-defined, and we have
$$ (g_i \ldots g_k) = (g_i \ldots g_j) \cdot (g_{j+1} \ldots g_k)$$
for all $1 \leq i \leq j < k \leq m$.

\begin{remark} It is worth pointing out one subtlety here: in order for $g_1 \ldots g_m$ to be well-defined, it is necessary that \emph{all} possible ways of decomposing this $m$-fold product into pairwise products be well-defined.  For instance, for $g_1 g_2 g_3$ to be well-defined, both $(g_1 \cdot g_2) \cdot g_3$ and $g_1 \cdot (g_2 \cdot g_3)$ need to be well-defined.  Similarly, if $g_1,g_2,g_3,g_4$ are such that $g_1 g_2 g_3$, $(g_1 g_2 g_3) \cdot g_4$, $g_2 g_3 g_4$, and $g_1 \cdot (g_2 g_3 g_4)$ are well-defined, this is not yet sufficient to deduce that $g_1 g_2 g_3 g_4$ is well-defined, because $(g_1 g_2) \cdot (g_3 g_4)$ need not be well-defined.  For instance, in the (additive) local group $\{-1,0,+1\}$, the expression $(+1) + (-1) + (-1) + (+1)$ is not well-defined, because $(-1)+(-1)$ is not well-defined.

Related to this is the well-known fact that local associativity does not imply global associativity: it is possible for two different ways of decomposing an $m$-fold product into pairwise products to both exist, but give distinct values; see \cite{olver} for further discussion.  For instance, there exists a local group $G$ and elements $g_1,g_2,g_3,g_4 \in G$ such that $((g_1 \cdot g_2) \cdot g_3) \cdot g_4$ and $g_1 \cdot (g_2 \cdot (g_3 \cdot g_4))$ both exist, but are not equal to one another.  Of course, in this case, we do not consider $g_1 g_2 g_3 g_4$ to be well-defined.
\end{remark}

Another easy induction also shows that for each $m \geq 1$, the set of tuples $(g_1,\ldots,g_m) \in G^m$ for which $g_1 \ldots g_m$ is well-defined is an open subset of $G^m$.

Now we extend the notion of products and inverses from individual group elements to sets of such elements.

\begin{definition}  Let $G$ be a symmetric local group. A subset $A$ of $G$ is said to be \emph{symmetric} if the set $A^{-1} := \{g^{-1}: g \in A \}$ is contained in $A$. If $A_1,\ldots,A_m$ are subsets of $G$, we say that $A_1 \ldots A_m$ is \emph{well-defined in $G$} (or \emph{well-defined} for short) if $g_1 \ldots g_m$ is well-defined for all $g_1 \in A_1, \ldots, g_m \in A_m$, in which case we write $A_1 \ldots A_m := \{ g_1 \ldots g_m: g_1 \in A_1, \ldots, g_m \in A_m\}$.  If $A_1=\ldots=A_m=A$, we abbreviate $A_1 \ldots A_m$ as $A^m$.  By abuse of notation, we write $A_1 \ldots A_m \subset G$ for the assertion that $A_1 \ldots A_m$ is well-defined in $G$. We adopt the convention that $A_1 \ldots A_m = \{\id\}$ when $m=0$.  In particular, $A^0 = \{\id\}$ for any $A \subset G$.
\end{definition}

An easy induction (see \cite[Lemma 2.5]{goldbring-local}) shows that for any local group $G$ and any open neighbourhood $U_0$ of the identity, there exists a nested sequence $U_0 \supset U_1 \supset U_2 \supset \ldots$ of symmetric open neighbourhoods of the identity such that $U_{m+1}^2 \subset U_m$ for every $m \geq 0$, which in particular implies that $U_m^m$ is well-defined in $U_0$, and thus $A_1 \ldots A_{m}$ is well-defined in $U_0$ whenever $A_1,\ldots,A_{m} \subset U_m$.

We make the trivial remark that multiplication of sets is associative: if $A_1 \ldots A_m$ is well-defined, then for any $1 \leq i \leq j < k \leq m$, $(A_i \cdot A_j) \cdot (A_{j+1} \ldots A_k)$ and $A_i \ldots A_k$ are well-defined and equal to each other.

By passing to neighbourhoods such as $U_m$, one can improve the group-like properties of a local group.  To illustrate this principle, let us first introduce the following definition.

\begin{definition}[Cancellative local groups]\label{cancel-def}  A symmetric local group $G$ is said to be \emph{cancellative} if the following assertions hold:
\begin{itemize}
\item[(i)] Whenever $g,h,k \in G$ are such that $gh$ and $gk$ are well-defined and equal to each other, then $h=k$.  (Note that this implies in particular that $(g^{-1})^{-1} = g$.)
\item[(ii)] Whenever $g,h,k \in G$ are such that $hg$ and $kg$ are well-defined and equal to each other, then $h=k$.
\item[(iii)] Whenever $g,h \in G$ are such that $gh$ and $h^{-1}g^{-1}$ are well-defined, then $(gh)^{-1} = h^{-1}g^{-1}$.  (In particular, if $U \subset G$ is symmetric and $U^m$ is well-defined in $G$ for some $m \geq 1$, then $U^m$ is also symmetric.)
\end{itemize}
\end{definition}

Clearly all global groups are cancellative.  A local group need not be cancellative everywhere; however, we can restrict to a large subset on which it is cancellative, by using the following proposition.

\begin{proposition}\label{cancelled}  Let $G$ be a symmetric local group, and let $U$ be an open symmetric neighbourhood of the identity in $G$ such that $U^6$ is well-defined.  Then the restriction of $G$ to $U$ is cancellative.
\end{proposition}

In particular, the restriction of $G$ to the open symmetric neighbourhood $U_6$ discussed earlier is cancellative.  We shall see later that the property of being cancellative is \emph{hereditary} in that it is inherited by passing to subgroups and quotients, and because of this we will be able to easily restrict attention to the cancellative case in our arguments.

\begin{proof}  If $g, h \in U$, then $(gh)^{-1} g h h^{-1} g^{-1}$ is well-defined in $G$.  By evaluating this well-defined expression in two different ways we conclude property (iii).  In a similar spirit, by evaluating $g^{-1} gh$ and $g^{-1} gk$ for $g,h,k \in U$ in two different ways, we obtain (i); and similarly for (ii).
\end{proof}

\begin{lemma}\label{sym} Let $G$ be a symmetric local group, and let $U, V$ be open sets with $\id \in V$.  Then $\overline{U} \subset U \cdot V$ if $U \cdot V$ is well-defined, and similarly $\overline{U} \subset V \cdot U$ if $V \cdot U$ is well-defined.
\end{lemma}

\begin{proof}  We prove the first claim only, as the second is similar.  Suppose that $g$ is an adherent point of $U$.  By continuity, we can find an open neighbourhood $W$ of $g$ and an open neighbourhood $Y$ of the identity such that $g \cdot g^{-1} \cdot W \cdot Y^{-1}$ is well-defined and $Y^{-1} \subset V$.  By continuity, the set $\{ h \in W: g^{-1} h \in Y \}$ is an open neighbourhood of $g$, and thus contains an element $h$ of $U$.  Writing $v := g^{-1} h$ and expanding out $g \cdot g^{-1} \cdot h \cdot v^{-1}$ in two different ways, we conclude that $g = h v^{-1}$, and thus $g \in U \cdot V$ as required.
\end{proof}

We can give the class of local groups the structure of a category by defining the notion of a (continuous) homomorphism.

\begin{definition}[Homomorphisms]\label{hom-def}  Let $G, H, K$ be symmetric local groups.  A \emph{continuous homomorphism} $\phi: G \to H$ is a continuous map from $G$ to $H$ with the following properties:
\begin{itemize}
\item[(i)] $\phi$ maps the identity of $G$ to the identity of $H$: $\phi(\id_G) = \id_H$.
\item[(ii)] For every $g \in G$, we have $\phi(g)^{-1}=\phi(g^{-1})$.
\item[(iii)] If $g,h \in G$ are such that $g \cdot h$ is well-defined, then $\phi(g) \cdot \phi(h)$ is well-defined and is equal to $\phi(g \cdot h)$.
\end{itemize}
We will often omit the adjective ``continuous'' when $G$ is discrete.

A \emph{local homomorphism} from $G$ to $H$ is a continuous homomorphism $\phi: U \to H$ from a symmetric open neighbourhood $U$ of the identity of $G$ to $H$, where of course we give $U$ the structure of the restricted local group $G\downharpoonright_U$ from Example \ref{restrict}.  Two local homomorphisms $\phi: U \to H$, $\phi': U' \to H$ are \emph{equivalent} if there exists a neighbourhood $V$ of the identity contained in both $U$ and $U'$ such that $\phi$ and $\phi'$ agree on $V$; this is an equivalence relation.  A \emph{local morphism} is an equivalence class of local homomorphisms.

Given two local homomorphisms $\phi: U \to H$ and $\psi: V \to K$ from $G$ to $H$ and $H$ to $K$ respectively, we define the composition map $\psi \circ \phi: U' \to K$ by $\psi \circ \phi(g) := \psi(\phi(g))$, where $U' := \{ g \in U: \phi(u) \in V \}$.  This allows one to define a composition of two local morphisms in the obvious manner.
\end{definition}

\begin{example}  There are no non-trivial global morphisms from the unit circle $\R/\Z$ to $\R$.  However, there do exist non-trivial local morphisms, such as (the equivalence class of) the map $\phi$ from $(-1/4,1/4) \mod 1$ to $\R$ defined by setting $\phi(x \mod 1) := x$ for all $x \in (-1/4,1/4)$.  The concept of a local homomorphism is closely related to that of a \emph{Freiman homomorphism} in additive combinatorics, as discussed for example in \cite{tv-book}.
\end{example}

One easily verifies that continuous homomorphisms and local morphisms both obey the axioms of a category; in particular, the composition of two continuous homomorphisms is a continuous homomorphism, and the composition of two local morphisms is again a local morphism.  As usual in category theory, we can now say that two local groups $G, G'$ are \emph{locally isomorphic} if there exists a local morphism $\phi$ from $G$ to $G'$ with an inverse $\phi'$ from $G'$ to $G$ which is also a local morphism, such that the compositions $\phi \circ \phi'$ or $\phi' \circ \phi$ are equivalent to the identity.  Thus, for instance, the unit circle $\R/\Z$ and the line $\R$ are locally isomorphic.  This notion of local isomorphism generalises the notion of local identity from Remark \ref{restrict}.

\begin{definition}[Sub-local groups \cite{goldbring-local}]  Given two symmetric local groups $G'$ and $G$, we say that $G'$ is a \emph{sub-local group} of $G$ if $G'$ is the restriction of $G$ to a symmetric neighbourhood of the identity, and there exists an open neighbourhood $V$ of $G'$ with the property that whenever $g, h \in G'$ are such that $gh$ is defined in $V$, then $gh \in G'$; we refer to $V$ as an \emph{associated neighbourhood} for $G'$.  If $G'$ is also a global group, we say that $G'$ is a \emph{subgroup} of $G$.

If $G'$ is a sub-local group of $G$, we say that $G'$ is \emph{normal} if there exists an associated neighbourhood $V$  for $G'$ with the additional property that whenever $g' \in G', h \in V$ are such that $h g' h^{-1}$ is well-defined and lies in $V$, then $hg'h^{-1} \in G'$.  We call $V$ a \emph{normalising neighbourhood} of $G'$.
\end{definition}

\begin{example}  If $G, G'$ are the (additive) local groups $G := \{-2,-1,0,+1,+2\}$ and $G' := \{-1,0,+1\}$, then $G'$ is a sub-local group of $G$ (with associated neighbourhood $V = G'$).  Note that this is despite $G'$ not being closed with respect to addition in $G$; thus we see why it is necessary to allow the associated neighbourhood $V$ to be strictly smaller than $G$.  In a similar vein, the open interval $(-1,1)$ is a sub-local group of $(-2,2)$.

The interval $(-1,1) \times \{0\}$ is also a sub-local group of $\R^2$; here, one can take for instance $(-1,1)^2$ as the associated neighbourhood.  As all these examples are abelian, they are clearly normal.
\end{example}

\begin{example}  Let $T: V \to V$ be a linear transformation on a finite-dimensional vector space $V$, and let $G := \Z \ltimes_T V$ be the associated semi-direct product.  Let $G' := \{0\} \times W$, where $W$ is a subspace of $V$ that is not preserved by $T$.  Then $G'$ is not a normal subgroup of $G$, but it is a normal sub-local group of $G$, where one can take $\{0\} \times V$ as a normalising neighbourhood of $G'$.
\end{example}

Observe that any sub-local group of a cancellative local group is again a cancellative local group.

One also easily verifies that if $\phi: U \to H$ is a local homomorphism from $G$ to $H$ for some open neighbourhood $U$ of the identity in $G$, then $\ker(\phi)$ is a normal sub-local group of $U$, and hence of $G$.  Note that the kernel of a local morphism is well-defined up to local identity.  If $H$ is Hausdorff, then the kernel $\ker(\phi)$ will also be closed.

Conversely, normal sub-local groups give rise to local homomorphisms into quotient spaces.

\begin{lemma}[Quotient spaces \cite{goldbring-local}]\label{quotient}  Let $G$ be a cancellative local group, and let $H$ be a normal sub-local group with normalising neighbourhood $V$.  Let $W$ be a symmetric open neighbourhood of the identity such that $W^6 \subset V$.  Then there exists a cancellative local group $W/H$ and a surjective continuous homomorphism $\phi: W \to W/H$ such that, for any $g, h \in W$, one has $\phi(g)=\phi(h)$ if and only if $gh^{-1} \in H$, and for any $E \subset W/H$, one has $E$ open if and only if $\phi^{-1}(E)$ is open.
\end{lemma}

\begin{proof}  We define an equivalence relation on $W$ by declaring $g \sim h$ if $g h^{-1} \in H$.   Using the cancellative properties of $V$ (and hence of $W^6$) we see that this is indeed an equivalence relation.  We let $W/H := \{ [g]_\sim: g \in W \}$ be the set of equivalence classes $[g]_\sim := \{ h \in W: g \sim h \}$, with the obvious projection map $\pi: W \to W/H$.  We define an inversion relation on $W/H$ by setting $[g]_\sim^{-1} := [g^{-1}]_\sim$, and a product operation by setting $[g]_\sim [h]_\sim$ to equal $[g'h']_\sim$ if $g'h' \in W$ for at least one representative $g', h'$ of $[g]_\sim, [h]_\sim$ respectively.

We now verify that these relations are well-defined.  To make the inversion relation well-defined, we need to verify that if $g \sim h$, then $g^{-1} \sim h^{-1}$.  But from the cancellative properties of $W^6$, we have $g^{-1} (h^{-1})^{-1} = g^{-1} (gh^{-1})^{-1} g$, and the claim follows as $W^6$ is a normalising neighbourhood for $H$.  Similarly, to make the multiplication relation well-defined, we need to verify that if $g, g', h, h'$ are such that $g \sim g'$, $h \sim h'$, and $gh, g'h' \in W$, then $gh \sim g'h'$.  But $(gh) (g'h')^{-1} = (g (g')^{-1}) g' (h (h')^{-1}) (g')^{-1}$, and the claim follows as $W^6$ is a normalising neighbourhood for $H$.  Similar arguments (which we omit) show that $W/H$ obeys the identity, inverse, and local associativity axioms.

Next, we give $W/H$ the quotient topology, declaring a set $E$ in $W/H$ open iff its inverse image $\pi^{-1}(E)$ is open in $W$ (or equivalently, in $G$).  One easily verifies that $W/H$ becomes a symmetric local group, and the claim follows.
\end{proof}

\begin{example}  Let $G$ be the additive local group $G := (-2,2)^2$, and let $H$ be the sub-local group $H := \{0\} \times (-1,1)$, with normalising neighbourhood $V := (-1,1)^2$.  If we then set $W := (-0.1,0.1)^2$, then the hypotheses of Lemma \ref{quotient} are obeyed, and $W/H$ can be identified with $(-0.1,0.1)$, with the projection map $\phi: (x,y) \mapsto x$.
\end{example}

\begin{example} Let $G$ be the torus $(\R/\Z)^2$, and let $H$ be the sub-local group $H = \{ (x,\alpha x) \mod \Z^2: x \in (-0.1,0.1)\}$, where $0 < \alpha < 1$ is an irrational number, with normalising neighbourhood $(-0.1,0.1)^2 \mod \Z^2$.  Set $W := (-0.01, 0.01)^2 \mod\Z^2$.  Then the hypotheses of Lemma \ref{quotient} are again obeyed, and $W/H$ can be identified with the interval $I := (-0.01(1+\alpha),0.01(1+\alpha))$, with the projection map $\phi: (x,y) \mod \Z^2 \mapsto y - \alpha x$ for $(x,y) \in (-0.01,0.01)^2$.  Note, in contrast, that if one quotiented $G$ by the \emph{global} group $\langle H \rangle = \{ (x,\alpha x) \mod \Z^2: x \in \R \}$ generated by $H$, the quotient would be a non-Hausdorff space (and would also contain a dense set of torsion points, in contrast to the interval $I$ which is ``locally torsion free'').  It is because of this pathological behaviour of quotienting by global groups that we need to work with local group quotients instead.
\end{example}

\begin{remark}  As we have seen in the above discussion, many familiar concepts in (global) group theory have analogues in the local group setting.  We will however mention one important global group-theoretic concept that does not have a convenient local analogue, and that is the notion of the global group $\langle A \rangle$ generated by a set $A$ of generators.  The problem is that this global group $\langle A \rangle$ consists of words in $A$ of \emph{arbitrarily} length, whereas in a local group one can typically only multiply together a bounded number of elements of $A$.  However, sets such as $A^m$ or $(A \cup A^{-1} \cup \{\id\})^m$ for various choices of exponent $m$ can sometimes serve as a partial substitute for this concept in local group theory, though one of course has to keep track of the precise value of $m$ throughout the argument.
\end{remark}

\textsc{Locally compact local groups.} Recall that a topological space $X$ is said to be \emph{locally compact} if and every point in $X$ has a compact neighbourhood.  In particular, one can speak of a locally compact symmetric local group.

To verify local compactness of a symmetric local group, it suffices to do so at the identity.

\begin{lemma}\label{locid}  Let $G$ be a symmetric local group.  Then $G$ is locally compact if and only if there is a compact symmetric neighbourhood of the identity.
\end{lemma}

\begin{proof}  \cite[Lemma 2.16]{goldbring-local} The ``only if'' part is clear (since $\id$ already has a compact neighbourhood).  Now we turn to the ``if'' part.  Let $K$ be a compact symmetric neighbourhood of the identity.  By continuity, there exists an open neighbourhood $V$ of $g$ such that $g \cdot g^{-1} \cdot V \cdot V$ is well-defined and $g^{-1} \cdot V \cdot V \subset K$.  In particular, $h \mapsto g^{-1} h$ is a homeomorphism from $V \cdot V$ to $g^{-1} \cdot V \cdot V$ which is inverted by the map $k \mapsto gk$.  By Lemma \ref{sym}, we conclude that $h \mapsto g^{-1} h$ is also a homeomorphism from $\overline{V}$ to $g^{-1} \cdot \overline{V} = \overline{g^{-1} \cdot V}$.  In particular, since $\overline{g^{-1} \cdot V}$ is a closed subset of $K$, it is compact, and so $\overline{V}$ is compact also.  Thus $g$ has a precompact neighbourhood as required.
\end{proof}

\begin{corollary}\label{oik}  If $G$ is a locally compact symmetric local group, and $U$ is a symmetric open neighbourhood of the identity, then $U$ is also a locally compact local group.
\end{corollary}

\begin{proof}  By Lemma \ref{locid}, $G$ contains a symmetric precompact open neighbourhood $V$ of the identity.  By continuity, one can find a symmetric open neighbourhood $W$ of the identity such that $W \cdot W$ is well-defined in $V \cap U$.  By Lemma \ref{sym}, we conclude that the closure $\overline{W}$ in $U$ is the same as the closure of $\overline{W}$ in $G$; as it is contained in the precompact set $V$, it is thus precompact.  The claim then follows from another application of Lemma \ref{locid}.
\end{proof}

An important subclass of the locally compact local groups are the \emph{(symmetric) local Lie groups}, defined as those (symmetric) local groups which are also smooth finite-dimensional real manifolds, such that the group operations are smooth on their domain of definition.  We have the following basic theorem.

\begin{theorem}[Lie's third theorem]\label{lie-third}  Every local Lie group is locally isomorphic to a global Lie group.  Furthermore, one can take the global Lie group to be both connected and simply connected.
\end{theorem}

See e.g. \cite{serre} for a proof.

We have the following deep structure theorem for locally compact global groups, due to Gleason and Yamabe \cite{yamabe2}.

\begin{theorem}[Gleason-Yamabe]\label{gleason-yamabe}
Suppose that $G$ is a locally compact global group. Then there is an open subgroup $G'$ of $G$ with the following property: inside any neighbourhood of the identity $U\subseteq G'$, there is a compact normal subgroup $H$ such that $G'/H$ is isomorphic to a connected global Lie group.
\end{theorem}

The analogous theorem for locally compact local groups was established more recently by Goldbring.

\begin{theorem}[Goldbring]\label{goldbring-thm}
Suppose that $G$ is a locally compact local group. Then some restriction $G'$ of $G$ to a symmetric neighbourhood of the identity has the following property. Inside any neighbourhood of the identity $U\subseteq G'$, there is a compact normal subgroup $H$ such that $G'/H$ is isomorphic to a local Lie group.
\end{theorem}

\begin{proof}
The only self-contained proof of Theorem \ref{goldbring-thm} in the literature is in the thesis \cite{goldbring-thesis}, where it follows from a combination of \S 4.5 and \cite[Proposition 4.7.1]{goldbring-thesis}.  A more easily accessible account of essentially the same material follows by combining \cite[Proposition 4.1]{dries} (reduction to the NSS case) with \cite[\S 8]{goldbring-local} (treatment of the NSS case). Alternatively (though ultimately more circuitously) one may apply the main result of \cite{dries}, which shows that $G$ has a restriction in common with a \emph{global} locally compact group, followed by Theorem \ref{gleason-yamabe}.  For our applications, we only need to apply Theorem \ref{goldbring-thm} when $G$ is metrisable, although the general case can be deduced from the metrisable case without much effort.
\end{proof}

\section{Nilprogressions and related objects}\label{nilprogression-sec}

In this appendix we prove two basic facts about coset nilprogressions in normal form, namely that after shrinking the length parameter slightly they are approximate groups, and are globalisable: that is to say isomorphic to subsets of a global group.  The proofs of these facts are quite short due to the strength of the normal form axioms.  One can establish similar assertions without the normal form hypothesis, but the arguments are much more complicated in that they require one to work with an explicit basis for the free nilpotent group. They are not needed in this paper.

\begin{lemma}\label{cosn}  Let $P = P_H(u_1,\ldots,u_r;N_1,\ldots,N_r)$ be a coset nilprogression in $C$-normal form.  Then for all $\eps > 0$ that are sufficiently small depending on $r, C$, one has
\begin{equation}\label{1nn}
 (1+N_1) \ldots (1+N_r) |H| \ll_{\eps,C,r} |P_H(u_1,\ldots,u_r; \eps N_1,\ldots,\eps N_r)| \ll_{C} (1+N_1) \ldots (1+N_r) |H|.
 \end{equation}
and hence, by the volume bounds on $P$,
$$
 |P_H(u_1,\ldots,u_r; \eps N_1,\ldots,\eps N_r)| \gg_{\eps,C,r} |P|.
$$
Furthermore, $P_H(u_1,\ldots,u_r; \eps N_1,\ldots,\eps N_r)$ is a $O_{\eps,C,r}(1)$-approximate group.
\end{lemma}

\begin{proof}  By quotienting out the finite group $H$, which is normalised by $P_H(u_1,\ldots,u_r; \eps N_1,\ldots,\eps N_r)^6$ (say) if $\eps$ is small enough, we may assume that $H$ is trivial.  The upper bound in \eqref{1nn} is then immediate from the upper bound in \eqref{cnn}, while the lower bound follows from the local properness axiom in Definition \ref{normal-def}.

From \eqref{1nn} and the Ruzsa covering lemma we see that for $\eps$ small enough, \[ P_H(u_1,\ldots,u_r; 2\eps N_1,\ldots,2\eps N_r)\] is covered by $O_{\eps,C,r}(1)$ translates of $P_H(u_1,\ldots,u_r; \eps N_1,\ldots,\eps N_r)$, and so the final claim follows from Lemma \ref{loc}.
\end{proof}

\begin{remark} It is in fact possible to show that $|P_H(u_1,\ldots,u_r; \eps N_1,\ldots,\eps N_r)|$ decays at a polynomial rate in $\eps$, and that
$P_H(u_1,\ldots,u_r; \eps N_1,\ldots,\eps N_r)$ is a $O_{C,r}(1)$-approximate group uniformly in $\eps$, but we will not need these stronger conclusions here.
\end{remark}

\begin{lemma}\label{normglob} Let $P = P_H(u_1,\ldots,u_r;N_1,\ldots,N_r)$ be a coset nilprogression in $C$-normal form.  Then for all $\eps > 0$ that are sufficiently small depending on $r, C$, the set $P_H(u_1,\ldots,u_r;\eps N_1,\ldots,\eps N_r)$ is isomorphic to a subset of a global group $G$.
\end{lemma}

From this lemma (and Lemma \ref{cosn}) we see that Theorem \ref{local-global} follows immediately from Theorem \ref{main-theorem}.

\begin{proof}  We first establish the claim under the additional hypothesis that the $N_1,\ldots,N_r$ are sufficiently large depending on $r,C$; we will remove this hypothesis at the end of the argument.

Let $v_1,\ldots,v_r$ be lifts of the generators $u_1,\ldots,u_r$ of $P/H$ to $P$.
By Definition \ref{normal-def} and the normality of $H$, one has
\begin{equation}\label{vivj}
[v_i,v_j] \in P\left( v_{j+1},\ldots,v_r; O_C( \frac{N_{j+1}}{N_i N_j} ), \ldots, O_C( \frac{N_r}{N_i N_j} ) \right) H
\end{equation}
for all $1 \leq i < j \leq r$; note that the hypothesis that the $N_i$ are large ensure that the right-hand side is well-defined in $P$.

Consider a word in $P\left( v_{j+1},\ldots,v_r; O_C( \frac{N_{j+1}}{N_i N_j} ), \ldots, O_C( \frac{N_r}{N_i N_j} ) \right)$, which therefore contains $O_C( \frac{N_{j+1}}{N_i N_j} )$ copies of $v_{j+1}^{\pm 1}$, $O_C( \frac{N_{j+2}}{N_i N_j} )$ copies of $v_{j+2}^{\pm 1}$, and so forth.  Let us the leftmost copy of $v_{j+1}^{\pm 1}$ and move it all the way to the left.  Each time it passes through a $v_k^{\pm 1}$ for some $j+1 < k \leq r$, we ue \eqref{vivj} and create $O_C( \frac{N_l}{N_{j+1} N_k} )$ new copies of $v_l^{\pm 1}$ for each $l>k$, plus an element of $H$ which can be pushed all the way to the right using the normality of $H$. Thus, if one initially had $a_k \frac{N_k}{N_i N_j}$ copies of $v_k^{\pm 1}$ for each $j+1 < k \leq r$ before one started moving the leftmost $v_{j+1}^{\pm 1}$ to the left, then by the end of the move, one would have
\begin{equation}\label{ain}
 a_l \frac{N_l}{N_i N_j} + O_C\left( \sum_{j+1 < k < l} \frac{a_k N_k}{N_i N_j} \frac{N_l}{N_{j+1} N_k} \right)
 \end{equation}
copies of $v_l^{\pm 1}$ for each $j+1 < l \leq r$.  We may simplify the expression \eqref{ain} as
$$ \left(a_l + O_C\left(\sum_{j+1<k<l} \frac{1}{N_{j+1}} a_k\right) \right) \frac{N_l}{N_i N_j}.$$
Thus we have effectively replaced the sequence $(a_k)_{j+1 < k \leq r}$ by the sequence
$$ \left(a_l + O_C\left(\sum_{j+1<k<l} \frac{1}{N_{j+1}} a_k\right) \right)_{j+1 < l \leq r}.$$
We iterate this process $O_C( \frac{N_{j+1}}{N_i N_j} ) = O_C(N_{j+1})$ times, and note that the $a_k$ were initially of size $O_C(1)$, and end up at a sequence, all of whose entries are of size $O_{C,r}(1)$.  In other words, after moving all copies of $v_{j+1}^{\pm 1}$ to the left, and all copies of $H$ to the right, we end up with $O_{C,r}( \frac{N_k}{N_i N_j} )$ copies of $v_k^{\pm 1}$ in the middle for each $j+1 < k \leq r$.
  We conclude that
$$ [v_i,v_j] \in v_{j+1}^{n_{i,j,j+1}} P\left( v_{j+2},\ldots,v_r; O_{C,r}( \frac{N_{j+2}}{N_i N_j} ), \ldots, O_{C,r}( \frac{N_r}{N_i N_j} ) \right) H$$
for some $n_{i,j,j+1} = O_C( \frac{N_{j+1}}{N_i N_j} )$; note that as long as the $N_i$ are large enough, all words that appear in this reorganisation will lie inside $P$ and so the algebraic manipulations can be justified.  Iterating this procedure $r-j$ times (which will be justified if the $N_i$ are large enough) we see that
\begin{equation}\label{vivj-2}
[v_i,v_j] = v_{j+1}^{n_{i,j,j+1}} \ldots v_r^{n_{i,j,r}} h_{i,j}
\end{equation}
for some $n_{i,j,k} = O_{C,r}(\frac{N_k}{N_i N_j})$ and $h_{i,j} \in H$.  Also, one has
\begin{equation}\label{vii}
v_i h v_i^{-1} = \phi_i(h)
\end{equation}
for some (outer) automorphism $\phi_i: H \to H$ of $H$.

Now let $G$ be the global group generated by $H$ and formal generators $e_1,\ldots,e_r$, subject to the relations
\begin{equation}\label{eij}
[e_i,e_j] = e_{j+1}^{n_{i,j,j+1}} \ldots e_r^{n_{i,j,r}} h_{i,j}
\end{equation}
and
\begin{equation}\label{eit}
e_i h e_i^{-1} = \phi_i(h)
\end{equation}
for $1 \leq i < j \leq r$.
We claim that for $\eps$ small enough, there is an injective homomorphism from $P_H(u_1,\ldots,u_r;\eps N_1,\ldots,\eps N_r)$ to $G$, which will give the claim.

To see this, first observe from the normality of $H$ that
$$ P_H(u_1,\ldots,u_r;\eps N_1,\ldots,\eps N_r) = P(v_1,\ldots,v_r;\eps N_1,\ldots,\eps N_r) H.$$
Organising the words in $P(v_1,\ldots,v_r;\eps N_1,\ldots,\eps N_r)$ by moving all occurrences of $v_1$ to the left (using \eqref{vivj-2}) and all occurrences of $H$ to the right (using the normality of $H$) we then have
\begin{equation}\label{phur-0}
 P_H(u_1,\ldots,u_r;\eps N_1,\ldots,\eps N_r) \subseteq   P(v_1;\eps N_1) P(v_2,\ldots,v_r; O_{C,r}(\eps N_2), \ldots, O_{C,r}(\eps N_r)) H
 \end{equation}
assuming $\eps$ is small enough in order to justify all the algebraic manipulations.  Iterating this we see that
\begin{equation}\label{phur}
 P_H(u_1,\ldots,u_r;\eps N_1,\ldots,\eps N_r) \subseteq   P(v_1;O_{C,r}(\eps N_1)) \ldots P(v_r; O_{C,r}(\eps N_r)) H.
\end{equation}

Thus it suffices to establish an injective homomorphism $\phi$ from the set
\begin{equation}\label{vvv}
 \{ v_1^{n_1} \ldots v_r^{n_r} h: n_i = O_{C,r}(\eps N_i); h \in H \}
\end{equation}
to $G$.  From the local properness property in Definition \ref{normal-def}, all the products in \eqref{vvv} are distinct if $\eps$ is small enough.  We may thus define $\phi$ by the formula
$$ \phi( v_1^{n_1} \ldots v_r^{n_r} h ) := e_1^{n_1} \ldots e_r^{n_r} h.$$

Next, we show that $\phi$ is injective.  Indeed, suppose that there exist $n_i,n'_i = O_{C,r}(\eps N_i)$ and $h,h' \in H$ with
$$\phi( v_1^{n_1} \ldots v_r^{n_r} h )  = \phi( v_1^{n'_1} \ldots v_r^{n'_r} h' ) $$
and thus
$$ e_1^{n_1} \ldots e_r^{n_r} h = e_1^{n'_1} \ldots e_r^{n'_r} h'.$$
By the universal properties of $G$, there is a homomorphism from $G$ to $\Z$ that maps $e_1$ to $1$ and annihilates the other $e_i$ and $H$.  This implies that $n_1 = n'_1$.  We can then eliminate $n_1, n'_1$ and work with the subgroup $G_2$ of $G$ generated by $e_2,\ldots,e_r$ and $H$.  From abstract nonsense we see that $G_2$ is universal with respect to the constraints \eqref{eij}, \eqref{eit} for $i \geq 2$, and that $G$ is the semidirect product of $G_2$ with $\Z$ using the conjugation action of $e_1$ on $G_2$ defined using \eqref{eij}, \eqref{eit} for $i=1$.  In particular, there is a homomorphism from $G_2$ to $\Z$ that maps $e_2$ to $1$ and annihilates the $e_i$ and $H$ for $i>2$.  This gives $n_2=n'_2$.  Continuing in this fashion we see that $n_i = n'_i$ for all $i$ and hence $h=h'$, which establishes injectivity.

Finally, we need to show that $\phi$ is a homomorphism.  It suffices to show that if $n_i,n'_i,n''_i = O_{C,r}(\eps N_i)$ and $h,h',h'' \in H$ are such that
\begin{equation}\label{vvhvvh}
v_1^{n_1} \ldots v_r^{n_r} h v_1^{n'_1} \ldots v_r^{n'_r} h' = v_1^{n''_1} \ldots v_r^{n''_r} h''
\end{equation}
then
\begin{equation}\label{eeheeh}
e_1^{n_1} \ldots e_r^{n_r} h e_1^{n'_1} \ldots e_r^{n'_r} h' = e_1^{n''_1} \ldots e_r^{n''_r} h''.
\end{equation}
To see this, we rearrange the word on the left-hand side of \eqref{vvhvvh} by moving all occurrences of $v_1$ to the left, and all occurrences of elements of $H$ to the right, using \eqref{vivj-2} and \eqref{vii}; if $\eps$ is small enough, then all manipulations take place inside $P$ and can thus be justified.  Iterating this process, we must eventually be able to express this word in the form $v_1^{\tilde n_1} \ldots v_r^{\tilde n_r} \tilde h$ for some $\tilde n_i = O_{C,r}(\eps N_i)$ and $\tilde h \in H$.  By injectivity, we then have $\tilde n_i = n''_i$ and $\tilde h = h''_i$.  But then if one formally replaces all the $v_i$ by $e_i$ and uses \eqref{eij}, \eqref{eit} in place of \eqref{vivj-2}, \eqref{vii} in the rearrangement procedure just described, we conclude \eqref{eeheeh}, and the claim follows.

Now we remove the hypothesis that the $N_1,\ldots,N_r$ are sufficiently large depending on $r,C$.  Let $F: \R^+ \to \R^+$ be a function depending on $r, C$ to be chosen later.  By the pigeonhole principle, we can find a threshold $M \ge 1$ with $M = O_{F}(1)$ such that every length $N_i$ is either less than $M$, or larger than $F(M)$.  If we let $1 \leq i_1 < \ldots < i_{r'} \leq r$ be those indices $i_j$ with $N_{i_j} > F(M)$, then we see (if $F$ is sufficiently rapidly growing) that $P_H( u_{i_1},\ldots,u_{i_{r'}}; N_{i_1},\ldots,N_{i_r})$ will be a coset nilprogression in $O_{C,r,M}(1)$-normal form.  For $F$ sufficiently rapidly growing, the preceding argument then applies to conclude that
$P_H( u_{i_1},\ldots,u_{i_{r'}}; \eps N_{i_1},\ldots,\eps N_{i_r})$ is isomorphic to a subset of a global group if $\eps$ is small enough depending on $C,r,M$, and the claim follows.
\end{proof}

\begin{remark}  From \eqref{phur} we see that every element in $P_H(u_1,\ldots,u_r;\eps N_1,\ldots,\eps N_r)$ takes the form
\begin{equation}\label{vih}
 v_1^{a_1} \ldots v_r^{a_r} h
\end{equation}
for some integers $a_i,\ldots,a_r$ with $a_i = O_{C,r}(\eps N_i)$ and $h \in H$.  Conversely, it is clear that if $|a_i| \leq \eps N_i$ then all expressions of the form \eqref{vih} lie in $P_H(u_1,\ldots,u_r;\eps N_1,\ldots,\eps N_r)$.  Informally, we thus see that the nilprogression $P_H(u_1,\ldots,u_r;\eps N_1,\ldots,\eps N_r)$  is comparable in some sense to the \emph{nilbox}
$$ \{v_1^{a_1} \ldots v_r^{a_r} h: |a_i| \leq \eps N_i; h \in H \}.$$
We will however not exploit this description of nilprogressions in this paper.
\end{remark}

A variant of the above analysis also gives polynomial growth of progressions in $C$-normal form in the global case.

\begin{proposition}[Polynomial growth]\label{poly} Let $P = P_H(u_1,\ldots,u_r;N_1,\ldots,N_r)$ be a coset nilprogression in $C$-normal form in a global group.  Then for all $m \geq 1$, one has $|P^m| \ll_{C,r} m^{O_{C,r}(1)} |P|$.
\end{proposition}

\begin{proof}  We allow all implied constants to depend on $C,r$.  As $H$ is normalised by $P$, we may quotient out by $H$ and reduce to the case when $H$ is trivial.  Then
$$ P^m \subseteq   P(u_1,\ldots,u_r;mN_1,\ldots,mN_r)$$
and so it suffices (by the volume bound \eqref{cnn}) to show that
$$ |P(u_1,\ldots,u_r;mN_1,\ldots,mN_r)| \ll m^{O(1)} (N_1+1) \ldots (N_r+1).$$
By modifying the proof of \eqref{phur-0}, one easily verifies that
$$ P(u_1,\ldots,u_r;m N_1,\ldots,m N_r) \subseteq   P(v_1;m N_1) P(v_2,\ldots,v_r; O(m^2 N_2), \ldots, O(m^2 N_r));$$
iterating this, one sees that
$$ P(u_1,\ldots,u_r;m N_1,\ldots,m N_r) \subseteq   P(v_1;O(m^{O(1)}) N_1) \ldots P(v_r; O(m^{O(1)} N_r)),$$
and the claim follows.
\end{proof}

\bibliographystyle{abbrv}
\bibliography{bibfile}

\begin{thebibliography}{10}

\bibitem{benjamini-kozma}
I.~Benjamini and G.~Kozma.
\newblock A resistance bound via an isoperimetric inequality.
\newblock {\em Combinatorica}, 25(6):645--650, 2005.

\bibitem{bieberbach}
L.~Bieberbach.
\newblock \"{U}ber einen satz des herrn c. jordan in der theorie der endlichen
  gruppen linearer substitutionen.
\newblock {\em Sitzber. Preuss. Akad. Wiss}, 1911.

\bibitem{bilu-2}
Y.~Bilu.
\newblock Addition of sets of integers of positive density.
\newblock {\em J. Number Theory}, 64(2):233--275, 1997.

\bibitem{bilu}
Y.~Bilu.
\newblock Structure of sets with small sumset.
\newblock {\em Ast\'erisque}, (258):xi, 77--108, 1999.
\newblock Structure theory of set addition.

\bibitem{breuillard-green}
E.~Breuillard and B.~Green.
\newblock Approximate groups. {I}: the torsion-free nilpotent case.
\newblock {\em J. Inst. Math. Jussieu}, 10(1):37--57, 2011.

\bibitem{bg}
E.~Breuillard and B.~Green.
\newblock Approximate groups. {II}: the solvable linear case.
\newblock {\em Quart. J. of Math., Oxford}, 62(3):513--521, 2011.

\bibitem{bgt}
E.~Breuillard, B.~Green, and T.~Tao.
\newblock Approximate subgroups of linear groups.
\newblock {\em Geom. And Funct. Anal.}, 21(4):774--819, 2011.

\bibitem{burago-zalgaller}
Y.~D. Burago and V.~A. Zalgaller.
\newblock {\em Geometric inequalities}, volume 285 of {\em Grundlehren der
  Mathematischen Wissenschaften [Fundamental Principles of Mathematical
  Sciences]}.
\newblock Springer-Verlag, Berlin, 1988.
\newblock Translated from the Russian by A. B. Sosinski{\u\i}, Springer Series
  in Soviet Mathematics.

\bibitem{chang-freiman}
M.-C. Chang.
\newblock A polynomial bound in {F}reiman's theorem.
\newblock {\em Duke Math. J.}, 113(3):399--419, 2002.

\bibitem{cheeger-colding}
J.~Cheeger and T.~H. Colding.
\newblock Lower bounds on {R}icci curvature and the almost rigidity of warped
  products.
\newblock {\em Ann. of Math. (2)}, 144(1):189--237, 1996.

\bibitem{corwin}
L.~J. Corwin and F.~P. Greenleaf.
\newblock {\em Representations of nilpotent {L}ie groups and their
  applications. {P}art {I}}, volume~18 of {\em Cambridge Studies in Advanced
  Mathematics}.
\newblock Cambridge University Press, Cambridge, 1990.
\newblock Basic theory and examples.

\bibitem{croot-sisask}
E.~Croot and O.~Sisask.
\newblock A probabilistic technique for finding almost-periods of convolutions.
\newblock {\em Geom. Funct. Anal.}, 20(6):1367--1396, 2010.

\bibitem{fisher-katz-peng}
D.~Fisher, N.~H. Katz, and I.~Peng.
\newblock Approximate multiplicative groups in nilpotent {L}ie groups.
\newblock {\em Proc. Amer. Math. Soc.}, 138(5):1575--1580, 2010.

\bibitem{freiman-book}
G.~A. Freiman.
\newblock {\em Foundations of a structural theory of set addition}.
\newblock American Mathematical Society, Providence, R. I., 1973.
\newblock Translated from the Russian, Translations of Mathematical Monographs,
  Vol 37.

\bibitem{fukaya-yamaguchi}
K.~Fukaya and T.~Yamaguchi.
\newblock The fundamental groups of almost non-negatively curved manifolds.
\newblock {\em Ann. of Math. (2)}, 136(2):253--333, 1992.

\bibitem{gallot-hulin}
S.~Gallot, D.~Hulin, and J.~Lafontaine.
\newblock {\em Riemannian geometry}.
\newblock Universitext. Springer-Verlag, Berlin, 1987.

\bibitem{gill-helfgott}
N.~Gill and H.~Helfgott.
\newblock Growth in solvable subgroups of
  $\mbox{GL}_r(\mathbb{Z}/p\mathbb{Z})$.
\newblock {\em arXiv:1008.5264}, 2010.
\newblock Preprint.

\bibitem{gill-helfgott-small}
N.~Gill and H.~Helfgott.
\newblock Growth of small generating sets in
  $\mbox{SL}_n(\mathbb{Z}/p\mathbb{Z})$.
\newblock {\em arXiv:1002.1605}, 2010.
\newblock Preprint.

\bibitem{gleason}
A.~M. Gleason.
\newblock The structure of locally compact groups.
\newblock {\em Duke Math. J.}, 18:85--104, 1951.

\bibitem{gleason2}
A.~M. Gleason.
\newblock Groups without small subgroups.
\newblock {\em Ann. of Math. (2)}, 56:193--212, 1952.

\bibitem{godel}
K.~G\"odel.
\newblock Consistency of the axiom of choice and of the generalized
  continuum-hypothesis with the axioms of set theory.
\newblock {\em Proc. Nat. Acad. Sci}, 24:556--557, 1938.

\bibitem{goldbring-thesis}
I.~Goldbring.
\newblock Nonstandard methods in lie theory.
\newblock Ph.D. Thesis, University of Illinois at Urbana-Champaign, 2009.

\bibitem{goldbring-local}
I.~Goldbring.
\newblock Hilbert's fifth problem for local groups.
\newblock {\em Ann. of Math. (2)}, 172(2):1269--1314, 2010.

\bibitem{green-ruzsa}
B.~Green and I.~Z. Ruzsa.
\newblock Freiman's theorem in an arbitrary abelian group.
\newblock {\em J. Lond. Math. Soc. (2)}, 75(1):163--175, 2007.

\bibitem{green-tao-compression}
B.~Green and T.~Tao.
\newblock Compressions, convex geometry and the {F}reiman-{B}ilu theorem.
\newblock {\em Q. J. Math.}, 57(4):495--504, 2006.

\bibitem{gromov}
M.~Gromov.
\newblock Groups of polynomial growth and expanding maps.
\newblock {\em Inst. Hautes \'Etudes Sci. Publ. Math.}, (53):53--73, 1981.

\bibitem{gromov-pansu-lafontaine}
M.~Gromov.
\newblock {\em Metric structures for {R}iemannian and non-{R}iemannian spaces}.
\newblock Modern Birkh\"auser Classics. Birkh\"auser Boston Inc., Boston, MA,
  english edition, 2007.
\newblock Based on the 1981 French original, With appendices by M. Katz, P.
  Pansu and S. Semmes, Translated from the French by Sean Michael Bates.

\bibitem{MHall}
M.~Hall, Jr.
\newblock {\em The theory of groups}.
\newblock Chelsea Publishing Co., New York, 1976.
\newblock Reprinting of the 1968 edition.

\bibitem{helfgott-sl2}
H.~A. Helfgott.
\newblock Growth and generation in {${\rm SL}_2(\Z/p\Z)$}.
\newblock {\em Ann. of Math. (2)}, 167(2):601--623, 2008.

\bibitem{helfgott-sl3}
H.~A. Helfgott.
\newblock Growth in {${\rm SL}_3(\Z/p\Z)$}.
\newblock {\em J. Eur. Math. Soc. (JEMS)}, 13(3):761--851, 2011.

\bibitem{hirschfeld}
J.~Hirschfeld.
\newblock The nonstandard treatment of {H}ilbert's fifth problem.
\newblock {\em Trans. Amer. Math. Soc.}, 321(1):379--400, 1990.

\bibitem{hrush}
E.~Hrushovski.
\newblock Stable group theory and approximate subgroups.
\newblock {\em Preprint}, 2010.
\newblock Preprint.

\bibitem{kaplansky}
I.~Kaplansky.
\newblock {\em Lie algebras and locally compact groups}.
\newblock The University of Chicago Press, Chicago, Ill.-London, 1971.

\bibitem{kapovitch-petrunin-tuschmann}
V.~Kapovitch, A.~Petrunin, and W.~Tuschmann.
\newblock Nilpotency, almost nonnegative curvature, and the gradient flow on
  {A}lexandrov spaces.
\newblock {\em Ann. of Math. (2)}, 171(1):343--373, 2010.

\bibitem{kapovitch-wilking}
V.~Kapovitch and B.~Wilking.
\newblock Structure of fundamental groups of manifolds with ricci curvature
  bounded below.
\newblock {\em arXiv:1105.5955}, 2011.
\newblock Preprint.

\bibitem{kleiner}
B.~Kleiner.
\newblock A new proof of {G}romov's theorem on groups of polynomial growth.
\newblock {\em J. Amer. Math. Soc.}, 23(3):815--829, 2010.

\bibitem{lee-makarychev}
J.~Lee and Y.~Makarychev.
\newblock Eigenvalue multiplicity and volume growth.
\newblock {\em arXiv:0806.1745}, 2008.
\newblock Preprint.

\bibitem{montgomery}
D.~Montgomery and L.~Zippin.
\newblock Small subgroups of finite-dimensional groups.
\newblock {\em Ann. of Math. (2)}, 56:213--241, 1952.

\bibitem{montgomery-zippin-book}
D.~Montgomery and L.~Zippin.
\newblock {\em Topological transformation groups}.
\newblock Interscience Publishers, New York-London, 1955.

\bibitem{olver}
P.~J. Olver.
\newblock Non-associative local {L}ie groups.
\newblock {\em J. Lie Theory}, 6(1):23--51, 1996.

\bibitem{pittet-saloff}
C.~Pittet and L.~Saloff-Coste.
\newblock A survey on the relationships between volume growth, isoperimetry,
  and the behavior of simple random walk on cayley graphs, with examples.
\newblock {\em Survey}, 2000.
\newblock Preprint.

\bibitem{pyber-szabo}
L.~Pyber and E.~Szab\'o.
\newblock Growth in finite simple groups of lie type of bounded rank.
\newblock {\em arXiv:1005.1858}, 2010.
\newblock Preprint.

\bibitem{ruzsa-freiman}
I.~Z. Ruzsa.
\newblock Generalized arithmetical progressions and sumsets.
\newblock {\em Acta Math. Hungar.}, 65(4):379--388, 1994.

\bibitem{ruzsa-finite-field}
I.~Z. Ruzsa.
\newblock An analog of {F}reiman's theorem in groups.
\newblock {\em Ast\'erisque}, (258):xv, 323--326, 1999.
\newblock Structure theory of set addition.

\bibitem{sanders-monomial}
T.~Sanders.
\newblock From polynomial growth to metric balls in monomial groups.
\newblock {\em arXiv:0912.0305}, 2009.
\newblock Preprint.

\bibitem{sanders}
T.~Sanders.
\newblock On a non-abelian balog-szemer\'edi-type lemma.
\newblock {\em Preprint}, 2010.
\newblock Preprint.

\bibitem{sanders-freiman}
T.~Sanders.
\newblock On the bogolyubov-ruzsa lemma.
\newblock {\em arXiv:1011.0107}, 2010.
\newblock Preprint.

\bibitem{sanders-nonabelian-idempotent}
T.~Sanders.
\newblock A quantitative version of the non-abelian idempotent theorem.
\newblock {\em Geom. Funct. Anal.}, 21(1):141--221, 2011.

\bibitem{serre}
J.-P. Serre.
\newblock {\em Lie algebras and {L}ie groups}, volume 1500 of {\em Lecture
  Notes in Mathematics}.
\newblock Springer-Verlag, Berlin, 2006.
\newblock 1964 lectures given at Harvard University, Corrected fifth printing
  of the second (1992) edition.

\bibitem{shalom-tao}
Y.~Shalom and T.~Tao.
\newblock A finitary version of {G}romov's polynomial growth theorem.
\newblock {\em Geom. Funct. Anal.}, 20(6):1502--1547, 2010.

\bibitem{tao-noncommutative}
T.~Tao.
\newblock Product set estimates for non-commutative groups.
\newblock {\em Combinatorica}, 28(5):547--594, 2008.

\bibitem{tao-solvable}
T.~Tao.
\newblock Freiman's theorem for solvable groups.
\newblock {\em Contrib. Discrete Math.}, 5(2):137--184, 2010.

\bibitem{tv-book}
T.~Tao and V.~Vu.
\newblock {\em Additive combinatorics}, volume 105 of {\em Cambridge Studies in
  Advanced Mathematics}.
\newblock Cambridge University Press, Cambridge, 2006.

\bibitem{thurston}
W.~P. Thurston.
\newblock {\em Three-dimensional geometry and topology. {V}ol. 1}, volume~35 of
  {\em Princeton Mathematical Series}.
\newblock Princeton University Press, Princeton, NJ, 1997.
\newblock Edited by Silvio Levy.

\bibitem{dries}
L.~van~den Dries and I.~Goldbring.
\newblock Globalizing locally compact local groups.
\newblock {\em J. Lie Theory}, 20(3):519--524, 2010.

\bibitem{vdd-goldbring}
L.~van~den Dries and I.~Goldbring.
\newblock Seminar notes on hilbert's 5th problem.
\newblock {\em Preprint}, 2010.
\newblock Preprint.

\bibitem{dries-wilkie}
L.~van~den Dries and A.~J. Wilkie.
\newblock Gromov's theorem on groups of polynomial growth and elementary logic.
\newblock {\em J. Algebra}, 89(2):349--374, 1984.

\bibitem{varopoulos}
N.~T. Varopoulos, L.~Saloff-Coste, and T.~Coulhon.
\newblock {\em Analysis and geometry on groups}, volume 100 of {\em Cambridge
  Tracts in Mathematics}.
\newblock Cambridge University Press, Cambridge, 1992.

\bibitem{yamabe}
H.~Yamabe.
\newblock A generalization of a theorem of {G}leason.
\newblock {\em Ann. of Math. (2)}, 58:351--365, 1953.

\bibitem{yamabe2}
H.~Yamabe.
\newblock On the conjecture of {I}wasawa and {G}leason.
\newblock {\em Ann. of Math. (2)}, 58:48--54, 1953.

\end{thebibliography}

\end{document}